\documentclass{amsart}

\usepackage{amssymb}
\usepackage{mathtools}
\usepackage{stmaryrd}

\numberwithin{equation}{section}

\newtheorem{theo}{Theorem} 
\newtheorem{mainprop}[theo]{Proposition} 
\newtheorem{maincorol}[theo]{Corollary} 

\newtheorem{lemma}{Lemma}[section]
\newtheorem{prop}[lemma]{Proposition}
\newtheorem{corol}[lemma]{Corollary}
\newtheorem{theoint}[lemma]{Theorem}
\newtheorem{claim}[lemma]{Claim}

\newtheorem{inttheo}[lemma]{Theorem}

\theoremstyle{remark}
\newtheorem{remark}[lemma]{Remark}

\newtheorem{step}{Step}
\newtheorem{cas}{Case}

\theoremstyle{definition}
\newtheorem{defi}[lemma]{Definition}

\newcommand{\hdot}{\dot{H}^1}
\newcommand{\vell}{\boldsymbol{\ell}}

\newcommand{\NN}{\mathbb{N}}
\newcommand{\RR}{\mathbb{R}}

\newcommand{\eps}{\varepsilon}

\newcommand{\tlh}{\tilde{h}}

\newcommand{\tu}{\tilde{u}}

\newcommand{\indxt}{\indic_{\{|x|\geq |t|\}}}

\newcommand{\vu}{\vec{u}}

\newcommand{\CCC}{\mathcal{C}}

\newcommand{\EEE}{\mathcal{E}}
\newcommand{\HHH}{\mathcal{H}}

\newcommand{\LLL}{\mathcal{L}}
\newcommand{\MMM}{\mathcal{M}}
\newcommand{\NNN}{\mathcal{N}}

\newcommand{\TTT}{\mathcal{T}}

\newcommand{\YYY}{\mathcal{Y}}
\newcommand{\ZZZ}{\mathcal{Z}}

\newcommand{\ZZZZ}{\boldsymbol{\mathcal{Z}}}

\newcommand{\tV}{\widetilde{V}}

\newcommand{\indic}{1\!\!1}
\newcommand{\spartial}{ {\slash\!\!\! \partial} }

\newcommand{\ent}[1]{\left\lfloor #1 \right\rfloor} 
\newcommand{\ind}{\indic_{\{|x|>|t|\}}}
\newcommand{\bigvec}{\overrightarrow}

\newcommand{\dix}{\frac{2(N+2)}{N-2}}
\newcommand{\cinq}{\frac{N+2}{N-2}}
\newcommand{\huit}{\frac{2(N+1)}{N-2}}

\DeclareMathOperator{\loc}{loc}
\DeclareMathOperator{\vect}{span}

\DeclareMathOperator{\sgn}{sgn}
\DeclareMathOperator{\app}{app}
\DeclareMathOperator{\id}{Id}
\DeclareMathOperator{\supp}{supp}

\title[Critical wave equation close to the ground state]{Exterior energy bounds for the critical wave equation close to the ground state}

\author[T.~Duyckaerts]{Thomas Duyckaerts$^1$}
\author[C.~Kenig]{Carlos Kenig$^2$}
\author[F.~Merle]{Frank Merle$^3$}
\thanks{$^1$Universit\'e Sorbonne Paris Nord, LAGA (UMR 7539) and Institut Universitaire de France}
\thanks{$^2$University of Chicago. Partially supported by NSF Grants DMS-14363746 and DMS-1800082}
\thanks{$^3$Universit\'e de Cergy-Pontoise, Laboratoire de Math\'ematiques AGM (UMR 8088) and Institut des Hautes \'Etudes Scientifiques}


\begin{document}
\begin{abstract}
 By definition, the \emph{exterior asymptotic energy} of a solution to a wave equation on $\RR^{1+N}$ is the sum of the limits as $t\to \pm\infty$ of the energy in the the exterior $\{|x|>|t|\}$ of the wave cone. In our previous work \cite{DuKeMe12}, we have proved that the exterior asymptotic energy of a solution of the linear wave equation in odd space dimension $N$  is bounded from below by the conserved energy of the solution. 

In this article, we study the analogous problem for the linear wave equation with a potential 
\begin{equation}
 \label{abstractLW}
\tag{*}
\partial_t^2u+L_Wu=0,\quad L_W:=-\Delta -\frac{N+2}{N-2}W^{\frac{4}{N-2}}
\end{equation} 
obtained by linearizing the energy critical wave equation at the ground-state solution $W$, still in odd space dimension. This equation admits nonzero solutions of the form $A+tB$, where $L_WA=L_WB=0$ with vanishing asymptotic exterior energy.
We prove that the exterior energy of a solution of \eqref{abstractLW} is bounded from below by the energy of the projection of the initial data on the orthogonal complement of the space of initial data corresponding to these solutions. This will be used in a subsequent paper to prove soliton resolution for the energy-critical wave equation with radial data in all odd space dimensions.

 We also prove analogous results for the linearization of the energy-critical wave equation around a Lorentz transform of $W$, and give applications to the dynamics of the nonlinear equation close to the ground state in space dimensions $3$ and $5$.
\end{abstract}

\maketitle
\tableofcontents

\section{Introduction}
Consider the wave equation on $\RR^N$, $N\geq 3$, with an energy-critical focusing nonlinearity:
\begin{equation}
 \label{NLW}
 \partial_t^2u-\Delta u=|u|^{\frac{4}{N-2}}u,
\end{equation} 
and initial data
\begin{equation}
 \label{ID}
 \vec{u}_{\restriction t=0}=(u_0,u_1)\in \HHH,
\end{equation} 
where $\vec{u}:=(u,\partial_t u)$, $\HHH:=\hdot(\RR^N)\times L^2(\RR^N)$, and $\hdot(\RR^N)$ is the usual homogeneous Sobolev space. The equation is locally well-posed in $\HHH$ (see e.g. \cite{LiSo95,KeMe08,BuCzLiPaZh13}): for any initial data $(u_0,u_1)\in \HHH$, there exists a unique maximal solution $\vec{u}\in C^0((T_-,T_+),\HHH)$ with $u\in L^{\frac{2(N+1)}{N-2}}\left(I\times \RR^N\right)$ for all intervals $I\Subset (T_-,T_+)$. The energy:
$$E(\vec{u}(t))=\frac{1}{2}\int_{\RR^N} |\nabla_{t,x}u(t,x)|^2\,dx-\frac{N-2}{2N}\int_{\RR^N} |u(t,x)|^{\frac{2N}{N-2}}\,dx$$
and the momentum 
$$P(\vec{u}(t))=\int_{\RR^N}\nabla u(t,x)\partial_tu(t,x)\,dx$$
of a solution are conserved, where 
$$\nabla u=(\partial_{x_j}u)_{1\leq j\leq N},\quad \nabla_{t,x}u=(\partial_tu,\nabla u).$$
The equation \eqref{NLW} has the following scaling invariance: if $u$ is a solution of \eqref{NLW} and $\lambda>0$, then 
$$\frac{1}{\lambda^{\frac{N}{2}-1}}u\left( \frac{t}{\lambda},\frac{x}{\lambda} \right)$$
 is also a solution. Denote for simplicity $\hdot=\hdot(\RR^N)$, $L^2=L^2(\RR^N)$. For $f\in \hdot$, $g\in L^2$ and $\lambda>0$, we denote
 $$f_{(\lambda)}(x)=\frac{1}{\lambda^{\frac N2-1}} f\left( \frac{x}{\lambda} \right),\quad g_{[\lambda]}(x)=\frac{1}{\lambda^{\frac{N}{2}}}\left( \frac{x}{\lambda} \right).$$
 We let:
$$ \Sigma=\{Q\in \hdot\setminus\{0\}\;:\;-\Delta Q=|Q|^{\frac{4}{N-2}}Q\},$$
 and
\begin{equation}
\label{defW}
W:=\left(1+\frac{|x|^2}{N(N-2)}\right)^{1-\frac{N}{2}}. 
\end{equation}  
 Then $W\in \Sigma$, and, as a consequence of  \cite{Pohozaev65,GiNiNi81}, the only radial elements of $\Sigma$ are $\pm\frac{1}{\lambda^{\frac{N}{2}-1}}W\left(\frac{\cdot}{\lambda}\right)$.  Note that $W$ is the ground state, i.e. it minimizes the energy $\frac 12\int |\nabla Q|^2-\frac{N-2}{2N}\int |Q|^{\frac{2N}{N-2}}$ among the elements of $Q$ of $\Sigma$ (see \cite{Aubin76,Talenti76}).

 If $x,y$ are two vectors in $\RR^N$, we denote by $x\cdot y$ their scalar product and $|x|$ the Euclidean norm of $x$. Let $u$ be a function on $\RR\times\RR^N$ and $\ell=|\vell|<1$. Then the \emph{Lorentz transform} of $u$ with parameter $\ell$ is defined as:
\begin{equation}
 \label{defLl}
 \LLL_{\vell} u(t,x)=u_{\vell}(t,x)=u\left(\frac{t-\vell\cdot x}{\sqrt{1-\ell^2}},\left(-\frac{t}{\sqrt{1-\ell^2}}+\frac{1}{\ell^2} \left(\frac{1}{\sqrt{1-\ell^2}}-1\right)\vell\cdot x\right)\vell+x\right).
\end{equation} 
We note that 
\begin{equation}
 \label{propLl}
(\partial_t^2-\Delta)\left( \LLL_{\vell}u \right)=\LLL_{\vell}\left((\partial_t^2-\Delta)u\right).
\end{equation} 
In particular, if $u$ is a $C^2$ solution of \eqref{NLW}, then $\LLL_{\vell}u$ is also a solution of \eqref{NLW}.

If $Q\in \Sigma$, $Q_{\vell}$ is a \emph{traveling wave} solution of \eqref{NLW}. Indeed,
\begin{equation}
\label{travelling_wave}
Q_{\vell}(t,x)=Q_{\vell}(0,x-t\vell).
\end{equation} 
The \emph{soliton resolution conjecture} predicts that every solution $\vec{u}$ of \eqref{NLW} such that $T_+=+\infty$, or which is bounded in $\HHH$ as $t\to T_+$ decomposes asympotically as $t\to T_+$, up to a term which is negligible in $\HHH$, as a sum of decoupled solitary waves (modulated by the transformations of the equation) and a radiation term. The radiation term is a solution of the linear wave equation if $T_+=+\infty$ or a fixed element of $\HHH$ if $T_+<\infty$.

Various results in this direction were recently proved. The dynamics below the energy of the ground state was studied in \cite{KeMe08}, exactly at the energy of the ground state in  \cite{DuMe08} and just above this energy in \cite{DuKeMe12}, \cite{KrNaSc13a}, \cite{KrNaSc13b}, \cite{KrNaSc15} These works are all in accordance with the soliton resolution conjecture.

In \cite{DuKeMe13} the complete re\-solution was obtained in space dimension $3$, for spherically symmetric initial data. In several articles, a weaker version of this conjecture was proved for solutions that remain bounded in $\HHH$, namely that the expansion as a sum of solitary waves and a radiation term holds for a sequence of times going to the maximal time of existence (that can be finite or infinite): 
\begin{itemize}
\item for spherically symmetric solutions: see \cite{Rodriguez16} in any odd space dimension, \cite{CoKeLaSc18} in space dimension $4$, \cite{JiaKenig17} in space dimension $6$.
\item without symmetry assumption: see \cite{DuJiKeMe17} in space dimension $3$, $4$ or $5$.
\end{itemize}
The proofs of most of these results rely on bounds from below of the energy outside wave cones, for some classes of solutions of \eqref{NLW}. These inequalities arise from the observation that for any solution $u_F$ of the free wave equation in odd space dimension:
\begin{equation}
\label{FW}
\partial_t^2u_F-\Delta u_F=0, 
\end{equation} 
with $\vec{u}_F(0)=(u_0,u_1)\in \HHH$, one has:
\begin{equation}
\label{ext_energ_intro}
\sum_{\pm}\lim_{t\to\pm\infty}  \int_{|x|\geq |t|}|\nabla_{t,x}u_F(t,x)|^2\,dx\geq \int |\nabla u_0|^2+u_1^2. 
\end{equation} 
In particular, the only solution $u$ of \eqref{FW} such that
\begin{equation}
 \label{no_channels}
\sum_{\pm}\lim_{t\to\pm\infty}  \int_{|x|\geq |t|}|\nabla_{t,x}u(t,x)|^2\,dx=0 
\end{equation} 
is the null solution (see \cite{DuKeMe12}). For even dimension, the quantitative estimate \eqref{ext_energ_intro} fails for solutions of \eqref{FW}, even for radial data (see \cite{CoKeSc14}). However, the fact that the property \eqref{no_channels} implies $u\equiv 0$ is still valid (see \cite[Proposition 1]{DuKeMe19Pc}). 
 
It turns out that this property fails for some solutions of the nonlinear equation \eqref{NLW}, but the study of this issue has been fundamental in the proofs of the results mentioned in the last paragraph. Of course, since solitary waves travel at velocity $<1$, \eqref{no_channels} holds if $u$ is a solitary wave, and we conjecture that these solutions are the only ones satisfying \eqref{no_channels}. This conjecture was proved in \cite{DuKeMe13} in space dimension $3$, assuming that the solution is radial (see Propositions 2.1 and 2.2 there), and is the key ingredient in the proof of the soliton resolution in that article. The proof is specific to the $3$-dimensional radial wave equation and cannot be generalized to a different dimension.

Note that the knowledge, from the articles cited above, that the soliton resolution conjecture holds along a sequence of times reduces the proof of the soliton resolution for all times to the understanding of the dynamics of the equation close to a sum of solitons, that are decoupled by scaling and space translations. It is important in particular to understand which solutions of \eqref{NLW} satisfy \eqref{no_channels} in a neighborhood of such a multi-soliton. In this article, we initiate this program by considering the case of the neighborhood of the ground state soliton $W$, proving, when $N$ is odd, an analog of \eqref{ext_energ_intro} for the linearized equation around $W$: 
\begin{equation}
 \label{LW}
 \partial_t^2u+L_Wu=0,
\end{equation} 
where $L_W$ is the linearized operator around $W$:
\begin{equation} 
 \label{def_LW}
 L_W=-\Delta-\frac{N+2}{N-2}W^{\frac{4}{N-2}},
\end{equation}
and more generally for the linearized equation around $W_{\vell}$ (see equation \eqref{LWl} below). 

We also give an application to the nonlinear equation \eqref{NLW} proving that in space dimensions $3$ and $5$,
the only solutions satisfying a slightly stronger statement than \eqref{no_channels}, in a neighborhood of the solitary waves $W_{\vell}$ is a solitary wave (see Corollary \ref{Cor:rigidity_intro}). 

Define:
\begin{equation*}
\ZZZ=\left\{Z\in \hdot\;:\; L_WZ=0\right\}.
 \end{equation*} 
Let 
$$\Lambda W:=x\cdot \nabla W+\left( \frac{N}{2}-1 \right)W.$$
Since \eqref{NLW} is invariant by scaling and space translation, we have
$$\Lambda W\in \ZZZ,\quad \partial_{x_j}W\in \ZZZ,\; j\in \llbracket 1, N\rrbracket,$$
where $\llbracket 1, N\rrbracket=\{1,\ldots,N\}$.
Furthermore the following nondegeneracy property of $W$ is known (see e.g. \cite{Rey90})
$$\ZZZ=\vect\Big\{\Lambda W,\; \partial_{x_j}W,\; j\in \llbracket 1, N\rrbracket\Big\}.$$
Note that for all $j\in \llbracket 1,N\rrbracket$, $\partial_{x_j}W\in L^2$, and that $\Lambda W\in L^2$ if and only if $N\geq 5$.

Define
\begin{equation}
\label{defZZZZ}
\ZZZZ=\left(\ZZZ\times \ZZZ\right)\cap \HHH, 
\end{equation} 
which is a finite dimensional subspace of $\HHH$. Indeed, $\ZZZZ$ is spanned by $(\Lambda W,0)$, $(\partial_{x_j}W,0)$, $(0,\partial_{x_j}W)$ if $N=3$ and $4$, and by $(\Lambda W,0)$, $(0,\Lambda W)$, $(\partial_{x_j}W,0)$, $(0,\partial_{x_j}W)$ if $N\geq 5$. 
Note that if $u$ is a solution of \eqref{LW} with initial data $(u_0,u_1)\in \ZZZZ$, then $u(t,x)=u_0(x)+tu_1(x)$. 

Let $\vell\in \RR^N$ such that $|\vell|<1$. Linearizing the equation \eqref{NLW} around $W_{\vell}$ we obtain the following generalization of \eqref{LW}:
\begin{equation} 
 \label{LWl}
 \partial_t^2u-\Delta u-\frac{N+2}{N-2}W_{\vell}^{\frac{4}{N-2}}u=0.
\end{equation} 
Note that the linear potential in \eqref{LWl} is now time-dependent. Let us mention that the global well-posedness of equation \eqref{LWl} is easy to prove. Indeed, the local well-posedness can be proved by a fixed point argument relying on the Strichartz inequality recalled in \eqref{Strichartz} below and on the H\"older inequality:
$$ \left\|W_{\vell}^{\frac{4}{N-2}}u\right\|_{L^1(I,L^2)}\leq \left\|W_{\vell}^{\frac{4}{N-2}}\right\|_{L^{\frac{2(N+1)}{N+4}}\left(I,L^{\frac{2(N+1)}{3}}\right)}\|u\|_{L^{\frac{2(N+1)}{N-2}}(I\times \RR^N)},$$
where $I$ is a bounded interval. By linearity of the equation and the fact that the $L^{\frac{2(N+1)}{3}}$ norm of $W_{\vell}(t)$ is independent of $t$, the time of existence obtained by this argument is independent of the initial data and of the initial time, and the global well-posedness follows. Note however that this does not yield a solution which is uniformly bounded in the energy norm, but in fact can grow exponentially in this norm.

Consider the finite dimensional subspace of 
$\HHH$ 
$$\ZZZZ_{\vell}=\left\{\left(\LLL_{\vell}(v_0(x)+tv_1(x)),
\partial_t \LLL_{\vell}(v_0(x)+tv_1(x))\right)_{\restriction t=0},\; (v_0,v_1)\in \ZZZZ\right\}.$$
The solution of \eqref{LWl} with initial data $(u_0,u_1)\in \ZZZZ_{\vell}$ is exactly:
\begin{equation}
 \label{formula}
u(t,x)=u_0(x-t\vell)+t(\vell\cdot \nabla u_0+u_1)(x-t\vell).
\end{equation} 
Indeed, assume $\vell=\ell(1,0,\ldots,0)$ to fix ideas, so that 
$$\LLL_{\ell}(v(t,x))=v\left(\frac{t-x_1\ell}{\sqrt{1-\ell^2}}, \frac{x_1-t\ell}{\sqrt{1-\ell^2}},x' \right),\quad x':=(x_2,\ldots,x_N).$$
Then 
\begin{multline*}
\LLL_{\vell}(v_0(x)+tv_1(x))=\\v_0\left( \frac{x_1-t\ell}{\sqrt{1-\ell^2}},x'\right)-\frac{\ell(x_1-t\ell)}{\sqrt{1-\ell^2}} v_1\left( \frac{x_1-t\ell}{\sqrt{1-\ell^2}},x'\right)+t\sqrt{1-\ell^2}\,v_1\left( \frac{x_1-t\ell}{\sqrt{1-\ell^2}},x'\right). 
\end{multline*}
This proves that solutions of \eqref{LWl} with initial data in $\ZZZZ_{\vell}$ are of the form $A(x-t\vell)+tB(x-t\vell)$. 
Letting $t=0$ in this formula (and its time derivative) we see that $u_0=A$ and $u_1=-\vell\cdot \nabla A+B$, which yields \eqref{formula}. 

Using \eqref{formula}, one can check that for these solutions \eqref{no_channels} holds. The main result of this article proves that the solutions of \eqref{LWl} with initial data in $\ZZZZ_{\vell}$ are (at least in odd space dimensions) the only solutions of \eqref{LWl} such that \eqref{no_channels}  holds. If $V$ is a closed subspace of $\HHH$, we denote by $V^{\bot}$ its orthogonal complement in $\HHH$ and by $\pi_{V}$ the orthogonal projection on $V$. Then:
\begin{theo}
\label{T:LW}
 Assume $N\geq 3$ is odd, and let $\eta_0\in (0,1)$. Then there exists a constant $C(N,\eta_0)>0$ such that for all $(u_0,u_1)\in \HHH$, for all $\vell\in \RR^N$ with $|\vell|\leq \eta_0$,
 \begin{equation}
  \label{A2}
 \left\|\pi_{\ZZZZ_{\vell}^{\bot}}(u_0,u_1)\right\|^2\leq C(N,\eta_0)\sum_{\pm} \lim_{t\to \pm\infty} \int_{|x|\geq |t|}|\nabla_{t,x}u(t,x)|^2,
 \end{equation} 
 where $u$ is the solution of \eqref{LWl} with initial data $(u_0,u_1)$.
\end{theo}
Theorem \ref{T:LW} will be used in \cite{DuKeMe19Pb} to prove the soliton resolution in general odd space dimension, in a radial context, for $\HHH$-bounded solutions. 

%

We will also prove, as a consequence of Theorem \ref{T:LW}:
\begin{maincorol}
\label{Cor:rigidity_intro}
Assume $N\in \{3,5\}$. Let $\eta_0\in (0,1)$. There exists a constant $\eps_0=\eps_0(\eta_0)>0$ such that, for any $\vell\in \RR^N$ with $|\vell|\leq \eta_0$, for all $(u_0,u_1)\in \HHH$ such that $\|(u_0,u_1)-\vec{W}_{\vell}(0)\|_{\HHH}\leq \eps_0$, 
the solution $u$ of \eqref{NLW} with initial data $(u_0,u_1)$ satisfies one of the following:
\begin{itemize}
 \item $u$ is a traveling wave, i.e. there exists $\lambda>0$, $\vell_0\in \RR^N$ with $|\vell_0|<1$ and $X\in \RR^N$ such that
 \begin{equation}
 \label{TW}
 u(t,x)=\frac{1}{\lambda^{\frac{N-2}{2}}}W_{\vell_0}\left(\frac{t}{\lambda},\frac{x -X}{\lambda}\right). 
 \end{equation} 
 \item There exists $X\in \RR^N$, $|X|\lesssim \eps_0$ such that, when $N=5$, for all small $\tau_0>0$, the solution $u$ of \eqref{NLW} is well defined for $\{|x-X|>|t|-\tau_0, \; t\in\RR\}$ and satisfies:
\begin{gather}
\label{channel}
\sum_{\pm}\lim_{t\to \pm\infty} \int_{|x-X|\geq |t|-\tau_0}\left|\nabla_{t,x}u(t,x)\right|^2\,dx>0. 
 \end{gather}
 For $N=3$, \eqref{channel} holds for $\tau_0=0$
 \end{itemize}
 \end{maincorol} 
(See Theorems \ref{T:NL3} and \ref{T:NL5} for more precise statements, and Section \ref{S:preliminaries} for the definition of a solution of \eqref{NLW} outside a wave cone).

Combining the preceding corollary with the soliton resolution up to a sequence of times proved in \cite{DuJiKeMe17}, we can also prove 
 \begin{maincorol}
\label{Cor:rigidity_intro2}
Assume $N\in \{3,5\}$. Let $(u_0,u_1)\in \HHH$ be such that $E(u_0,u_1)<2E(W,0)$. If the solution $u$ of \eqref{NLW} with initial data $(u_0,u_1)$ is global and is not a traveling wave then there exists $\tau_0>0$, such that \eqref{channel} holds.
 \end{maincorol} 
Since $E(u_0,u_1)<2E(W,0)$ in Corollary \ref{Cor:rigidity_intro2}, the assumption that $u$ is not a traveling wave is equivalent to the fact that $u$ is not of the form \eqref{TW}.
 
 We believe that Corollaries \ref{Cor:rigidity_intro} and Corollaries \ref{Cor:rigidity_intro2} are also true for all $N$ odd, however our proof of Corollary \ref{Cor:rigidity_intro} cannot be carried out in dimension $N\geq 7$ because of a technical difficulty due to the weaker long-time perturbation theory statement available in these dimensions. Let us mention however that in large odd dimension, Corollary \ref{Cor:rigidity_intro2} remains valid, for radial solutions, as a consequence of \cite{Rodriguez16} and Theorem \ref{T:LW}:
 \begin{mainprop}
\label{P:rigidity_intro3}
Assume $N\geq 7$. Let $(u_0,u_1)\in \HHH$, radial  and such that $E(u_0,u_1)<2E(W,0)$. If the solution $u$ of \eqref{NLW} with initial data $(u_0,u_1)$ is global and is not a stationary solution, then there exists $\tau_0>0$, such that \eqref{channel} holds with $X=0$.
 \end{mainprop} 
We refer to \cite[Theorem 1.4]{Rodriguez16} for the proof of the soliton resolution conjecture for all times, in the odd-dimensional radial setting, with an assumption similar to $E(u_0,u_1)<2E(W,0)$,

We will prove Corollary \ref{Cor:rigidity_intro2} and Proposition \ref{P:rigidity_intro3} in Subsection \ref{SS:rig2}. Let us mention that the conclusion of Proposition \ref{P:rigidity_intro3} remains valid, in a radial setting, without the assumption $E(u_0,u_1)<2E(W,0)$. In other words, in odd space dimension, there is no pure radial multisoliton for equation \eqref{NLW}, in sharp contrast with the completely integrable case. We will obtain this stronger statement, as a byproduct of our proof of the soliton resolution conjecture, in our subsequent paper \cite{DuKeMe19Pb}.

The analog of Theorem \ref{T:LW}, in the case of the free wave equation in odd space dimensions was proved by the authors in \cite{DuKeMe12}, using the explicit representation formula for the free wave equation. Note that in the free case, no orthogonal projection is needed, since the analog of $\ZZZZ_{\vell}$ is $\{(0,0)\}$. To treat the linearized operator, we first consider $|x|>R_0$, $R_0$ large, so as to treat the linearized operator as a perturbation of the free one. This makes it necessary to use the version of the ``channel'' property introduced by the authors in \cite[Lemma 4.2]{DuKeMe11a}, extended to $N=5$ in \cite[Proposition 4.1]{KeLaSc14} and to general odd $N$ in \cite{KeLaLiSc15}. Here a projection is needed, and it is the orthogonal projection to the orthogonal complement of specific subspaces, whose dimension increases with $N$ (see Theorem \ref{T:KLLS} below). For the non-radial free case, more exceptional subspaces can be seen to arise, each one for each spherical harmonic degree. The main challenge that we need to overcome, to establish Theorem \ref{T:LW}, is to eliminate the redundant counterexamples that arise from the use of \cite{KeLaLiSc15}, as we take $R_0\to 0$.

The outline of the paper is as follows. After some preliminaries (Section \ref{S:preliminaries}), we reduce in Section \ref{S:reduction} the proof of Theorem \ref{T:LW} to the proof of a uniqueness result on solutions of the linearized equation \eqref{LW} (Theorem \ref{T:LW''}), which is a weaker, qualitative version of Theorem \ref{T:LW}. This section relies on the Lorentz transformation and the profile decomposition of \cite{BaGe99}. In Section \ref{S:uniqueness}, we prove the uniqueness Theorem \ref{T:LW''}. Projecting on spherical harmonics, it is sufficient to study a family of radial wave equations with a potential in odd space dimensions, which we do using the lower energy bound for the free wave equation obtained in \cite{KeLaLiSc15}. In Section \ref{S:NL_channels} we prove two rigidity theorems,  which imply Corollary \ref{Cor:rigidity_intro}, for solutions of the nonlinear wave equation in space dimension $N=3$ and $N=5$. The main tool is Theorem \ref{T:LW}. We also prove Corollary \ref{Cor:rigidity_intro2} and Proposition \ref{P:rigidity_intro3}, using  the soliton resolution for a sequence of times (from \cite{Rodriguez16}, \cite{DuJiKeMe17}) and, in the nonradial case, Theorem \ref{T:LW}.
In Appendix \ref{S:Lorentz} we recall some useful facts about the Lorentz transformation for equation \eqref{NLW}.

\section{Preliminaries}
\label{S:preliminaries}

\subsection{Duhamel formulation and Strichartz estimates}
We denote by $S_L(t)$ the linear wave evolution:
\begin{equation}
\label{SL}
S_L(t)(u_0,u_1)=\cos(t\sqrt{-\Delta})u_0+\frac{\sin(t\sqrt{-\Delta})}{\sqrt{-\Delta}}u_1,
\end{equation} 
so that the general solution (in the Duhamel sense) of 
\begin{equation}
 \label{LW_in}
\left\{
\begin{aligned}
(\partial_t^2-\Delta)u&=f\\
\vec{u}_{\restriction t=t_0}&=(u_0,u_1)\in \HHH,
\end{aligned}\right.
\end{equation} 
where $I$ is an interval and $t_0\in I$
is 
\begin{equation}
\label{Duhamel}
u(t)=S_L(t-t_0)(u_0,u_1)+\int_{t_0}^{t}S_L(t-s)(0,f(s))\,ds. 
\end{equation} 
We note that by finite speed of propagation, for any $x_0\in \RR^N$ and any $R\geq 0$, the restriction of $u$ to $\left\{(t,x)\in I\times \RR^N\;:\;|x-x_0|>|t-t_0|+R\right\}$ depends only on the restriction of $f$ to  $\left\{(t,x)\in I\times \RR^N\;:\;|x-x_0|>|t-t_0|+R\right\}$ and the restriction of $(u_0,u_1)$ to $\left\{ x\in \RR^N\;:\;|x-x_0|>R\right\}$.

We need to introduce function spaces adapted to the Cauchy theory in large space dimension. Define $
S(I):=L^{\frac{2(N+1)}{N-2}}\left( I\times \RR^N\right)$ and:
\begin{gather*}
W(I):=L^{\frac{2(N+1)}{N-1}}\left( I,\dot{B}^{\frac{1}{2}}_{\frac{2(N+1)}{N-1},2} (\RR^N) \right),\;
W'(I):= L^{\frac{2(N+1)}{N+3}}\left(I,\dot{B}^{\frac{1}{2}}_{\frac{2(N+1)}{N+3},2}(\RR^N) \right)\end{gather*}
(where $B^{\frac 12}_{p,2}$ are the usual Besov spaces  (see e.g. \cite[\S 6.3]{BerghLofstrom76BO})). 
We have the following Strichartz estimate: if $u$ satisfies \eqref{Duhamel} on $I$, with $f=f_1+f_2$, then
\begin{multline}
\label{Strichartz}
\indic_{\{N\leq 6\}}\left\|u\right\|_{L^{\frac{N+2}{N-2}}\big(I,L^{\frac{2(N+2)}{N-2}}(\RR^N)\big)}+
\sup_{t\in \RR} \|\vec{u}(t)\|_{\HHH}+\|u\|_{L^{\frac{2(N+1)}{N-2}}\left(I\times\RR^N\right)}
+\|u\|_{W(I)} \\
\lesssim \|\vec{u}(0)\|_{\HHH(I)} +\|f_1\|_{W'(I)}
+\|f_2\|_{L^1(I,L^2)}.
\end{multline}
Note that the bound of the $L^{\frac{N+2}{N-2}}\big(I,L^{\frac{2(N+2)}{N-2}}(\RR^N)\big)$ norm is only available in space dimension $N\leq 6$. We will use it for the proof of Corollary \ref{Cor:rigidity_intro} in space dimensions $3$ and $5$.

We will also need the fact that the spaces $W'(I)$ can be localized in the exterior of wave cones. Let, for $T>0$, $R\geq 0$, 
$$\Gamma_R(T):=\left\{(t,x)\;:\; 0\leq t\leq T,\; |t|>|x|+R\right\}.$$
Then there is a constant $C>0$ (independent of $T$ and $R$) such that (see \cite{BuCzLiPaZh13}, \cite[Lemma 2.3]{DuKeMe19Pc})
\begin{equation}
 \label{boundedness_indic}
 \left\|\indic_{\Gamma_R(T)}f\right\|_{W'((0,T))}\leq C\left\|f\right\|_{W'(0,T)}.
\end{equation} 
Furthermore we have the following chain rule for fractional derivatives (\cite[Remark 2.4]{DuKeMe19Pc}):
\begin{equation}
 \label{fractional_cones}
 \|\indic_{\Gamma_0(T)}F(u)\|_{W'((0,T))}\leq C \|\indic_{\Gamma_0(T)}u\|_{S((0,T))}^{\frac{4}{N-2}}\|u\|_{W((0,T))}.
\end{equation}

\subsection{Profile decomposition}
\label{SS:profile}
Let $\big\{(u_{0,n},u_{1,n})\big\}_n$ be a bounded sequence of radial functions in $\HHH$. We say that it admits a profile decomposition if for all $j\geq 1$, there exist a solution $U^j_F$ to the free wave equation with initial data in $\HHH$ and sequences of parameters $\{\lambda_{j,n}\}_n\in (0,\infty)^{\NN}$, $\{t_{j,n}\}_n\in \RR^{\NN}$, $\{x_{j,n}\}_n\in (\RR^N)^{\NN}$ such that
\begin{equation}
 \label{psdo_orth}
 j\neq k \Longrightarrow \lim_{n\to\infty}\frac{\lambda_{j,n}}{\lambda_{k,n}}+\frac{\lambda_{k,n}}{\lambda_{j,n}}+\frac{|t_{j,n}-t_{k,n}|}{\lambda_{j,n}}+\frac{|x_{j,n}-x_{k,n}|}{\lambda_{j,n}}=+\infty,
\end{equation} 
and, denoting 
\begin{gather}
 \label{rescaled_lin}
 U^j_{F,n}(t,r)=\frac{1}{\lambda_{j,n}^{\frac N2-1}}U^j_F\left( \frac{t-t_{j,n}}{\lambda_{j,n}},\frac{x-x_{j,n}}{\lambda_{j,n}} \right),\quad j\geq 1\\
 w_{n}^J(t)=S_L(t)(u_{0,n},u_{1,n})-\sum_{j=1}^J U^j_{L,n}(t),
\end{gather} 
one has 
\begin{equation}
 \label{wnJ_dispersive}
 \lim_{J\to\infty}\limsup_{n\to\infty}\|w_n^J\|_{L^{\frac{2(N+1)}{N-2}}(\RR^{N+1})}=0.
\end{equation} 
We recall (see \cite{BaGe99}, \cite{Bulut10}) that any bounded sequence in $\HHH$ has a subsequence that admits a profile decomposition.


\subsection{Wave equation with potential outside a wave cone}
If $R\geq 0$ and $(u_0,u_1)\in \HHH$, we will denote
$$\|(u_0,u_1)\|^2_{\HHH(R)}=\int_{|x|\geq R} (|\nabla u_0(x)|^2+u_1^2(x))\,dx.$$

\begin{lemma}
 \label{L:linear_approx}
 Let $N\geq 3$ and $M\in (0,\infty)$.
 There exists $C_M>0$ such that for all $R\geq 0$, for all $V\in L^{\frac{2(N+1)}{N+4}}_{\loc}\left( \RR, L^{\frac{2(N+1)}{3}}\left(\RR^N\right) \right)$ with
 \begin{equation}
  \label{G84}
  \left\|\indic_{\{|x|>R+|t|\}} V\right\|_{L^{\frac{2(N+1)}{N+4}} \left( \RR, L^{\frac{2(N+1)}{3}} \left(\RR^N\right) \right)} \leq M,
 \end{equation} 
 for all solutions $u$ of 
 \begin{equation}
  \label{G83}
  \partial_t^2u-\Delta u+Vu=f_1+f_2,\quad \vec{u}_{\restriction t=0}=(u_0,u_1)\in \HHH(R),
 \end{equation} 
 where $\indic_{\{|x|>R+|t|\}}f_1\in W'(\RR)$, $\indic_{\{|x|>R+|t|\}} f_2\in L^1\left(\RR,L^2(\RR^N)\right)$, one has:
 \begin{multline}
  \label{G90}
  \indic_{\{N\leq 6\}}\left\|u\indic_{\{|x|>R+|t|\}}\right\|_{L^{\frac{N+2}{N-2}}\left( \RR,L^{\frac{2(N+2)}{N-2}} \right)}\\
 +\left\|u\indic_{\{|x|>R+|t|\}}\right\|_{L^{\frac{2(N+1)}{N-2}}(\RR\times \RR^N)}+\sup_{t\in \RR}\left\|\indic_{\{|x|>R+|t|\}} \nabla_{t,x}u(t)\right\|_{L^2}\\ 
  \leq C_M\left( \|(u_0,u_1)\|_{\HHH(R)}+\left\| \indic_{\{|x|>R+|t|\}} f_1\right\|_{W'(\RR)}+\left\| \indic_{\{|x|>R+|t|\}} f_2\right\|_{L^1(\RR,L^2)}\right).
 \end{multline} 
\end{lemma}
\begin{proof}
 Let
$$   A:=\|(u_0,u_1)\|_{\HHH(R)}+\left\|\indic_{\{|x|>R+|t|\}} f_1\right\|_{W'(\RR)}+
\left\|\indic_{\{|x|>R+|t|\}} f_2\right\|_{L^1(\RR,L^2)}.$$
 By Strichartz inequality, for all $T>0$,
 \begin{equation*}
  \left\|\indic_{\{|x|>R+|t|\}} u\right\|_{L^{\frac{2(N+1)}{N-2}}\left( [0,T]\times \RR^N \right)}
  \lesssim \left\| \indic_{\{|x|>R+|t|\}} Vu\right\|_{L^1\left((0,T),L^2\right)}+A.
 \end{equation*}
Using H\"older inequality in the space variable, we deduce
\begin{multline*}
  \left\|\indic_{\{|x|>R+|t|\}} u\right\|_{L^{\frac{2(N+1)}{N-2}}\left( [0,T]\times \RR^N \right)}
   \\ \lesssim A
  +\int_0^{T}\left\|\indic_{\{|x|>R+|t|\}} V\right\|_{L^{\frac{2(N+1)}{3}}}\left\| \indic_{\{|x|>R+|t|\}} u\right\|_{L^{\frac{2(N+1)}{N-2}}}\,dt,
\end{multline*}
and thus, using a Gr\"onwall type inequality (Lemma 8.1 of \cite{FaXiCa11}) we obtain
\begin{equation*}
 \left\|\indic_{\{|x|>R+|t|\}} u\right\|_{L^{\frac{2(N+1)}{N-2}}\left( [0,T]\times \RR^N \right)} \leq C_M A.
\end{equation*}
Using Strichartz and H\"older's inequalities again we deduce the rest of \eqref{G90}.
  \end{proof}
As a consequence of Lemma \ref{L:linear_approx}, one can obtain an asymptotic formula, outside a wave cone, for solutions of the wave equation \eqref{G83}. For simplicity, we restrict ourselves to the linearized equation \eqref{LWl}. However we remark for further use that the same result holds when $W_{\vell}$ is replaced by any solitary wave $Q_{\vell}$ in \eqref{LWl}.
\begin{corol}
\label{Co:radiation}
 Assume $N\geq 3$. Then for all $\vell \in \RR^N$ such that $|\vell|<1$, there exists a bounded linear map
\begin{align*}
 \Phi: \HHH & \longrightarrow \left(L^2([0,+\infty)\times S^{N-1})\right)^2\\
(u_0,u_1)& \longmapsto (H_+,H_-)
\end{align*}
such that, denoting by $u$ the solution of \eqref{LWl} with initial data $(u_0,u_1)$ at $t=0$, we have
\begin{align*}
 \lim_{t\to \pm\infty} \int_{|t|}^{\infty} \int_{S^{N-1}} \left|r^{\frac{N-1}{2}} \partial_ru(t,r\theta)-H_{\pm}(r-|t|,\theta)\right|^2\,d\sigma(\theta)\,dr&=0\\
 \lim_{t\to \pm\infty} \int_{|t|}^{\infty} \int_{S^{N-1}} \left|r^{\frac{N-1}{2}} \partial_tu(t,r\theta)\mp H_{\pm}(r-|t|,\theta)\right|^2\,d\sigma(\theta)\,dr&=0\\
\lim_{t\to\pm\infty} \int_{|x|\geq |t|} \left|\spartial u(t,x)\right|^2+\frac{1}{|x|^2} |u(t,x)|^2\,dx=0.
\end{align*}
If one fixes $\eta_0\in (0,1)$, the operator norm of $\Psi$ is bounded uniformly with respect to $\vell$ such that $|\vell|\leq \eta_0$, 
\end{corol}
We note that a small variant of the proof shows the existence of a linear map $\Psi: (u_0,u_1)\mapsto G_{\pm}\in L^2_{\loc}\left( \RR\times S^N \right)$ such that, for all $A\in \RR$,
\begin{align*}
 \lim_{t\to \pm\infty} \int_{|t|+A}^{\infty} \int_{S^{N-1}} \left|r^{\frac{N-1}{2}} \partial_ru(t,r\theta)-G_{\pm}(r-|t|,\theta)\right|^2\,d\sigma(\theta)\,dr&=0\\
 \lim_{t\to \pm\infty} \int_{|t|+A}^{\infty} \int_{S^{N-1}} \left|r^{\frac{N-1}{2}} \partial_tu(t,r\theta)\mp G_{\pm}(r-|t|,\theta)\right|^2\,d\sigma(\theta)\,dr&=0\\
\lim_{t\to\pm\infty} \int_{|x|\geq |t|+A} \left|\spartial u(t,x)\right|^2+\frac{1}{|x|^2} |u(t,x)|^2\,dx=0.
\end{align*}
By Corollary \ref{Co:radiation}, the restriction of $G_{\pm}$ to $[0,+\infty)\times S^{N-1}$ belongs to $L^2([0,+\infty)\times S^{N-1})$. However, $G_{\pm}$ is not, in general, an element of $L^2(\RR\times S^{N-1}$), as shows the example of the solution $u(t,r)=e^{\omega t}\YYY(x)$, where $-\omega^2$ is the negative eigenvalue of $L_W$, and $\YYY(x)$ the corresponding eigenfunction. In this case it follows from the asymptotics of $\YYY$ (see \cite{Meshkov89}, \cite[Proposition 3.9]{DuKeMe16a}) that $G_+(\eta,\theta)=e^{-\omega \eta}V(\theta)$ for some function $V\in L^2(S^{N-1})\setminus \{0\}$ and $G_-\equiv 0$.
\begin{proof}[Proof of Corollary \ref{Co:radiation}]
We focus on the case $t\to+\infty$ and the construction of $H_+$ to lighten notations. It is known (see e.g. \cite{Friedlander62, Friedlander80} and \cite[Appendix B]{DuKeMe19}) that for all $(u_0,u_1)\in \HHH$, there exists $G_+\in L^2(\RR\times S^{N-1})$ such that, denoting by $u_F$ the solution of the free wave equation with initial data $(u_0,u_1)$ we have
\begin{align*}
 \lim_{t\to +\infty} \int_{0}^{\infty} \int_{S^{N-1}} \left|r^{\frac{N-1}{2}} \partial_ru_F(t,r\theta)-G_{+}(r-|t|,\theta)\right|^2\,d\theta\,dr&=0\\
 \lim_{t\to +\infty} \int_{0}^{\infty} \int_{S^{N-1}} \left|r^{\frac{N-1}{2}} \partial_tu_F(t,r\theta)- G_{+}(r-|t|,\theta)\right|^2\,d\theta\,dr&=0\\
\lim_{t\to +\infty} \int_{\RR^N} \left|\spartial u_F(t,x)\right|^2+\frac{1}{|x|^2} |u_F(t,x)|^2\,dx=0,
\end{align*}
and that $(u_0,u_1)\mapsto G_+$ is a one to one isometry from $\HHH$ to $L^2(\RR\times S^{N-1})$. 

To prove the corollary, it is thus sufficient to construct a bounded linear operator
\begin{align*}
\Phi:\HHH &\to \HHH \\
(u_0,u_1)&\mapsto (\tu_0,\tu_1),
\end{align*}
such that, if $\tilde{u}_F$ is the solution of the free wave equation \eqref{FW} with initial data $(\tu_0,\tu_1)$, and $u$ the solution of \eqref{LWQ} with initial data $(u_0,u_1)$, one has
\begin{equation*}
 \lim_{t\to\infty} \int_{|x|\geq |t|} \left|\nabla_{t,x}\tilde{u}_F(t,x)-\nabla_{t,x}u(t,x)\right|^2\,dx=0.
\end{equation*}
By Lemma \ref{L:linear_approx}, there exists $C>0$ (depending only on $\eta_0$) such that for all $(u_0,u_1)\in \HHH$, one has 
\begin{equation}
\label{A8}
\left\|\indic_{\{|x|\geq |t|\}} u\right\|_{L^{\frac{2(N+1)}{N-2}}\left( [0,+\infty)\times \RR^N\right)}\leq C\|(u_0,u_1)\|_{\HHH}. 
\end{equation} 
We define $\tu$ as the solution of 
\begin{equation*}
 \left\{\begin{aligned} 
 \partial_t^2\tilde{u}-\Delta\tilde{u}&=\frac{N+2}{N-2} \indic_{\{|x|\geq |t|\}} |W_{\vell}|^{\frac{4}{N-2}}u\\
 \vec{u}_{\restriction t=0}&=(u_0,u_1).
 \end{aligned}\right.
\end{equation*} 
By finite speed of propagation
\begin{equation}
 \label{A9}
 |x|>t>0\Longrightarrow u(t,x)=\tu(t,x).
\end{equation} 
Furthermore, letting 
$$(\tu_0,\tu_1)=(u_0,u_1)+\int_{0}^{+\infty} \vec{S}_L(-s)\left( 0,\frac{N+2}{N-2}\indic_{\{|x|\geq |s|\}}|W_{\vell}|^{\frac{4}{N-2}}u(s) \right)\,ds,$$
and 
$$\tu_F=S_L(t)(\tilde{u}_0,\tilde{u}_1),$$
we see that 
\begin{equation}
 \label{A10}
 \lim_{t\to\infty} \left\|\vec{\tu}_F(t)-\vec{\tu}(t)\right\|_{\HHH}=0.
\end{equation} 
Indeed,
\begin{align*}
 \tilde{u}_F(t)=&S_L(t)(u_0,u_1)+\int_0^{+\infty} S_L(t-s)\left( 0,\frac{N+2}{N-2}\indic_{\{|x|\geq |s|\}}|W_{\vell}|^{\frac{4}{N-2}}u(s) \right)\,ds\\
 &=\tilde{u}(t)+\int_t^{+\infty} S_L(t-s)\left( 0,\frac{N+2}{N-2}\indic_{\{|x|\geq |s|\}}|W_{\vell}|^{\frac{4}{N-2}}u(s) \right)\,ds,
\end{align*}
which proves \eqref{A10}. Combining \eqref{A9} and \eqref{A10}, we have proved the corollary.
\end{proof}

\section{Reduction to a uniqueness theorem}
\label{S:reduction}
In this section we reduce the proof of Theorem \ref{T:LW} to the following uniqueness theorem:
\begin{inttheo}
\label{T:LW''}
 Assume $N\geq 3$ is odd. For any solution $u$ of \eqref{LW} with initial data $(u_0,u_1)\in \HHH$, if 
\begin{equation}
\label{ext_0}
\sum_{\pm} \lim_{t\to \pm\infty}\int_{\{|x|>|t|\}} |\nabla_{t,x}u(t,x)|^2\,dx=0,
\end{equation} 
then $(u_0,u_1)\in \ZZZZ$.
\end{inttheo}
Note that Theorem \ref{T:LW''} concerns the linearized equation \eqref{LW} around $W$ and not the more general linearized equation \eqref{LWl} around the Lorentz transform of $W$.

The proof is divided into two steps.
In Subsection \ref{SS:Lor}, we use the Lorentz transformation to reduce the proof of Theorem \ref{T:LW} to the case $\ell=0$, i.e. to a statement on the linearized equation \eqref{LW}. 
In Subsection \ref{SS:qual}, we prove, using profile decomposition, that Theorem \ref{T:LW''} implies Theorem \ref{T:LW}.

\subsection{Lorentz transformation}
\label{SS:Lor}
In this subsection we prove that the following theo\-rem implies Theorem \ref{T:LW}:
\begin{inttheo}
\label{T:LW'}
 Assume $N\geq 3$ is odd. Then there exists a constant $C>0$ such that for all $(u_0,u_1)\in \HHH$,
 \begin{multline}
  \label{A2'}
 \frac{1}{C}\int |\nabla u_0(x)|^2+(u_1(x))^2\,dx \leq \left\|\pi_{\ZZZZ}(u_0,u_1)\right\|^2+ \sum_{\pm} \lim_{t\to \pm\infty} \int_{|x|\geq |t|}|\nabla_{t,x}u(t,x)|^2,
 \end{multline} 
 where $u$ is the solution of \eqref{LW} with initial data $(u_0,u_1)$.
\end{inttheo}
We first need a trace property related to the Lorentz transformation. Let $\vell \in \RR^N$ with $|\vell|<1$. Define the map 
$\TTT_{\vell}$ as follows. If $(u_0,u_1)\in \left(C^{\infty}_0(\RR^N)\right)^2$ and $u$ is the solution of \eqref{LW} with initial data $(u_0,u_1)$, then 
$$\TTT_{\vell}(u_0,u_1)=\left(\LLL_{\vell}(u),\partial_t\LLL_{\vell}(u)\right)_{\restriction t=0}.$$
Here $\LLL_{\vell}u$ is the Lorentz transform of $u$, defined in \eqref{defLl}.
Note that 
$$\TTT_{\vell}:\left(C^{\infty}_0(\RR^N)\right)^2\to\left(C^{\infty}_0(\RR^N)\right)^2.$$
\begin{lemma}
\label{L:trace}
The map $\TTT_{\vell}$ can be extended to a bounded linear isomorphism from $\HHH$ to $\HHH$. Furthermore, for all $\eta_0\in (0,1)$, there exists a constant $C>0$ such that for all $(u_0,u_1)\in \HHH$, for all $\vell\in \RR^N$ with $|\vell|\leq \eta_0$, 
\begin{equation}
 \label{LL1}
\frac{1}{C}\left\|\TTT_{\vell}(u_0,u_1)\right\|_{\HHH}\leq \|(u_0,u_1)\|_{\HHH}\leq C\left\|\TTT_{\vell}(u_0,u_1)\right\|_{\HHH}.
\end{equation} 
\end{lemma}
\begin{proof}
 In all of the proof we fix $\eta_0\in (0,1)$ and $\vell\in \RR^N$ with $|\vell|\leq \eta_0$. The letter $C$ will denote a constant, that may change from line to line, depends on $\eta_0$ but not on $\vell$. We will assume without loss of generality that 
$$\vell=\ell(1,0,\ldots,0).$$
\begin{step}[Left-hand bound]
 In this step we assume $(u_0,u_1)\in C^{\infty}_0(\RR^N)$ and prove the left-hand inequality in \eqref{LL1}. We let $\tilde{u}$ be the solution of 
\begin{equation*}
 \left\{\begin{aligned}
         \partial_t^2\tu-\Delta\tu&=\frac{N+2}{N-2}W^{\frac{4}{N-2}}\indic_{\{|x|\geq |t|\}}u\\
\vec{u}_{\restriction t=0}&=(u_0,u_1).
        \end{aligned}\right.
\end{equation*}
By finite speed of propagation, $u(t,x)=\tu(t,x)$ for $|x|>|t|$. Denoting $x'=(x_2,\ldots,x_N)$ we have
\begin{equation*}
 \TTT_{\vell}(u_0,u_1)(x)=\left( u\left(Y_{\ell}\right),\frac{1}{\sqrt{1-\ell^2}} \partial_tu\left( Y_{\ell}\right)
-\frac{\ell}{\sqrt{1-\ell^2}}\partial_{x_1}u\left( Y_{\ell}\right) \right),
\end{equation*}  
where $Y_{\ell}=\left(\frac{-\ell x_1}{\sqrt{1-\ell^2}},\frac{x_1}{\sqrt{1-\ell^2}},x'\right)$.
Since  $\left( \frac{x_1}{\sqrt{1-\ell^2}}\right)^2+(x')^2\geq \left( \frac{\ell x_1}{\sqrt{1-\ell^2}} \right)^2$, we deduce
$$ \TTT_{\vell}(u_0,u_1)=\left( \LLL_{\vell}\tu,\partial_t\LLL_{\vell}\tu \right)_{\restriction t=0}.$$
We recall Lemma 2.2 of \cite{KeMe08}: if $\partial_t^2u-\Delta u=f$, then 
$$ \left\|\left(\LLL_{\vell}(u),\partial_t\LLL_{\vell}(u)\right)_{\restriction t=0}\right\|_{\HHH}\leq C\left(\left\|(u_0,u_1)\right\|_{\HHH}+\|f\|_{L^1_tL^2_x}\right),$$
where the constant $C>0$ depends only on $\eta_0$. Note that in \cite{KeMe08}, the lemma is stated with $\eta_0=\frac{1}{4}$, but that exactly the same proof works for any $\eta_0\in (0,1)$. 

We thus obtain:
\begin{multline*}
 \left\|\TTT_{\vell}(u_0,u_1)\right\|_{\HHH} \leq C\left\|(u_0,u_1)\right\|_{\HHH}+C\left\|\indic_{\{|x|>|t|\}} W^{\frac{4}{N-2}}u\right\|_{L^1_tL^2_x}\\
\leq C\|(u_0,u_1)\|_{\HHH}+C\left\|\indic_{\{|x|>|t|\}} W^{\frac{4}{N-2}}\right\|_{L^{\frac{2(N+1)}{N+4}}_tL^{\frac{2(N+1)}{3}}_x}\|\indic_{\{|x|>|t|\}}u\|_{L^{\frac{2(N+1)}{N-2}}_{t,x}}.
\end{multline*}
Using that by Lemma \ref{L:linear_approx}
$$\|\indic_{\{|x|>|t|\}}\tu\|_{L^{\frac{2(N+1)}{N-2}}_{t,x}}=\|\indic_{\{|x|>|t|\}}u\|_{L^{\frac{2(N+1)}{N-2}}_{t,x}}\leq C\|(u_0,u_1)\|_{\HHH},$$
we deduce 
\begin{equation*}
 \left\|\TTT_{\vell}(u_0,u_1)\right\|_{\HHH} \leq C\left\|(u_0,u_1)\right\|_{\HHH}
\end{equation*} 
which concludes this step.
\end{step}
\begin{step}[Right-hand inequality]
Note that by the first step, $\TTT_{\vell}$ extends to a bounded linear operator from $\HHH$ to $\HHH$. In this step we construct the inverse $\widetilde{\TTT}_{\vell}$ of $\TTT_{\vell}$ and prove the right-hand inequality in \eqref{LL1}. We define $\widetilde{\TTT}_{\vell}$ as follows.

Let $(v_0,v_1)\in \left(C^{\infty}_0(\RR^N)\right)^2$, and consider the solution $v$ of 
\begin{equation}
\label{LWellbis}
 \left\{\begin{aligned}
         \partial_t^2v-\Delta v&=\frac{N+2}{N-2}W^{\frac{4}{N-2}}_{\vell}v\\
\vec{v}_{\restriction t=0}&=(v_0,v_1).
        \end{aligned}\right.
\end{equation}
Then $\widetilde{\TTT}_{\vell}(v_0,v_1):=\left(\LLL_{-\vell}(v),\partial_t\LLL_{-\vell}(v)\right)_{\restriction t=0}$. By the same argument as in Step 1, 
\begin{equation}
\label{LL2}
\|\widetilde{\TTT}_{\vell}(v_0,v_1)\|_{\HHH}\leq C\left\|(v_0,v_1)\right\|_{\HHH}.
\end{equation} 
Let $(u_0,u_1)\in \left(C^{\infty}_0(\RR^N)\right)^2$, and 
denote by $u$ the solution of \eqref{LW} with initial data $(u_0,u_1)$, and by $v=\LLL_{\vell}u$. Then $v$ is the solution of \eqref{LWellbis} with initial data $\TTT_{\vell}(u_0,u_1)$, which shows that 
$$\widetilde{\TTT}_{\vell}(v_0,v_1)=\left(\LLL_{-\vell}\LLL_{\vell} u,\partial_t \LLL_{-\vell}\LLL_{\vell} u\right)_{\restriction t=0} =(u_0,u_1)$$
Since $v=\LLL_{\vell} u$, we have proved
$$ \widetilde{\TTT}_{\vell}\circ\TTT_{\vell}(u_0,u_1)=(u_0,u_1).$$
Using \eqref{LL2} and a density argument, we see that $\widetilde{\TTT}_{\vell}$ can be extended to a bounded linear map from $\HHH$ to $\HHH$ that satisfies $\widetilde{\TTT}_{\vell}\circ \TTT_{\vell}=\id_{\HHH}$. The right-hand side inequality in \eqref{LL1} follows immediately.
\end{step}
\end{proof}
We next prove that Theorem \ref{T:LW'} implies Theorem \ref{T:LW}. 

Assume Theorem \ref{T:LW'}. We fix in all the proof $\eta_0\in (0,1)$, $\vell$ such that $|\vell|\leq \eta_0$. As before $C>0$ denotes a constant that might depend on $\eta_0$ but not on $\vell$. 
\setcounter{step}{0}
\begin{step}[Lorentz transformation]
In this step we prove that there exists a projection $P_{\vell}$ on $\HHH$, whose rank is exactly the dimension of $\ZZZZ_{\vell}$, and such that for all $(u_0,u_1)\in \HHH$, the solution $u$ of \eqref{LWl} satisfies 
 \begin{multline}
  \label{LL2'}
 \frac{1}{C}\int |\nabla u_0(x)|^2+(u_1(x))^2\,dx \leq \left\|P_{\vell}(u_0,u_1)\right\|^2+ \sum_{\pm} \lim_{t\to \pm\infty} \int_{|x|\geq |t|}|\nabla_{t,x}u(t,x)|^2.
 \end{multline} 
Using a density argument, it is sufficient to prove \eqref{LL2'} for $(u_0,u_1)\in \left( C^{\infty}_0(\RR^N) \right)^2$. Assuming (without loss of generality) that $\vell=(\ell,0,\ldots,0)$, and letting 
$$v(t,x)=\left(\LLL_{\vell}u\right)(t,x)=u\left( \frac{t+\ell x_1}{\sqrt{1-\ell^2}},\frac{x_1+t\ell}{\sqrt{1-\ell^2}},x' \right),$$
we see that $v$ satisfies the equation \eqref{LW}. By Theorem \ref{T:LW'}, we obtain
\begin{equation}
 \label{LL3} 
 \left\|\pi_{\ZZZZ}(v_0,v_1)\right\|^2_{\HHH} +\sum_{\pm}\lim_{t\to\pm \infty} \int_{|x|>|t|} |\nabla_{t,x}v(t,x)|^2\,dx\geq \frac{1}{C}\left\|(v_0,v_1)\right\|^2_{\HHH}.
\end{equation} 
We have $\|(u_0,u_1)\|^2_{\HHH}=\|\TTT_{\ell}(v_0,v_1)\|^2_{\HHH}$ and thus, by Lemma \ref{L:trace},
\begin{equation}
 \label{LL4}
\|(u_0,u_1)\|^2_{\HHH}\leq C\|(v_0,v_1)\|^2_{\HHH}.
\end{equation} 
Lemma \ref{L:trace} also implies
\begin{equation}
 \label{LL4'}
 \left\|\pi_{\ZZZZ}(v_0,v_1)\right\|_{\HHH}^2=\left\|\pi_{\ZZZ}\circ \TTT_{\ell}^{-1}(u_0,u_1)\right\|^2_{\HHH}
 \leq C\left\|\TTT_{\ell}\circ\pi_{\ZZZZ}\circ\TTT_{\ell}^{-1}(u_0,u_1)\right\|^2_{\HHH}.
\end{equation} 
Note that $P_{\vell}:=\TTT_{\ell}\circ\pi_{\ZZZZ}\circ\TTT_{\ell}^{-1}$ is a projection of rank $\dim \ZZZZ$. In view of \eqref{LL4} and \eqref{LL4'}, we are left with proving that $L_{u}\geq L_v$, where
$$L_u:=\sum_{\pm} \lim_{t\to\pm \infty} \int_{|x|>|t|} |\nabla_{t,x}u(t,x)|^2\,dx,$$
and similarly for $L_v$. Note that $L_u$ and $L_v$ exist by Corollary \ref{Co:radiation}. 

Let 
$\CCC(T,\ell)=\Big\{(s,y)\in \RR^{1+N}\;:\;
|y|>|s|\text{ and } -T\leq \frac{s-\ell y_1}{\sqrt{1-\ell^2}}\leq T \Big\}. $
By the change of variable
$$ (s,y)=\left( \frac{t+\ell x_1}{\sqrt{1-\ell^2}},\frac{x_1+t\ell}{\sqrt{1-\ell^2}},x' \right),\quad (t,x)=\left( \frac{s-\ell y_1}{\sqrt{1-\ell^2}},\frac{y_1-s\ell}{\sqrt{1-\ell^2}},y' \right)$$
we obtain
\begin{multline}
\label{LL5}
L_v=\lim_{T\to \infty} \frac{1}{T}\int_{-T}^{+T} \int_{|x|>|t|} |\nabla_{t,x}v(t,x)|^2\,dx\,dt\\
=\lim_{T\to\infty} \frac 1T \iint_{\CCC(T,\ell)}
\left[\frac{1+\ell^2}{1-\ell^2}\left( (\partial_su)^2+(\partial_{y_1}u)^2 \right)+\frac{4\ell}{1-\ell^2} \partial_su\partial_{y_1}u+(\partial_{y'}u)^2\right]\,ds\,dy.
\end{multline}
Hence
\begin{equation*}
L_v\leq \frac{C}{T}\limsup_{T\to+\infty} \iint_{\CCC(T,\ell)} |\nabla_{s,y} u(s,y)|^2\,ds\,dy.
 \end{equation*} 
 Since $(u_0,u_1)\in \left( C_0^{\infty}(\RR^N) \right)^2$, we know (using finite speed of propagation) that $u(s,y)=0$ for $|y|\geq s+ K_u$, where $K_u$ is a constant depending on $u$. Thus, in $\CCC(T,\ell)$, we have
$$|s|\leq T\sqrt{1-\ell^2}+\ell|y|\leq T\sqrt{1-\ell^2}+\ell|s|+\ell K_u.$$
Hence $(1-\ell)|s|\leq T\sqrt{1-\ell^2} +\ell K_u$, which implies
$$ |s|\leq c_{\ell}T +\frac{\ell}{1-\ell}K_u, \quad c_{\ell}:=\sqrt{\frac{1+\ell}{1-\ell}}>1.$$
As a consequence,
\begin{multline}
\label{bound_Lorentz}
 L_v\leq \frac{C}{T} \limsup_{T\to +\infty}\int_{-c_{\ell}T-\frac{\ell}{1-\ell}K_u}^{c_{\ell} T+\frac{\ell}{1-\ell}K_u} \int_{|y|\geq |s|} |\nabla_{s,y}u(s,y)|^2dyds\\
 =\frac{C}{T} \limsup_{T\to +\infty}\int_{-c_{\ell}T}^{c_{\ell}T} \int_{|y|\geq |s|} |\nabla_{s,y}u(s,y)|^2dy\,ds\leq C L_u.
\end{multline} 
Combining with \eqref{LL4} and \eqref{LL4'}, we deduce \eqref{LL2'}.
\end{step}
\begin{step}[Conclusion of the proof]
 We conclude the proof using elementary linear algebra. Recalling that
 \begin{equation}
 \label{LL7'}
 (z_0,z_1)\in \ZZZZ_{\vell}\Longrightarrow  \lim_{t\to+\infty}\int_{|x|\geq |t|} |\nabla_{t,x}z(t,x)|^2\,dx=0,
 \end{equation} 
 we obtain with \eqref{LL2'} that 
 \begin{equation}
  \label{LL7}
  (z_0,z_1)\in \ZZZZ_{\vell}\Longrightarrow C\left\| P_{\vell}(z_0,z_1)\right\|^2_{\HHH}\geq \left\|(z_0,z_1)\right\|^2_{\HHH}.
 \end{equation} 
 Since the dimension of $\ZZZZ_{\vell}$ and the rank of $P_{\vell}$ are equal, we deduce that the restriction of $P_{\vell}$ to $\ZZZZ_{\vell}$ is an isomorphism between $\ZZZZ_{\vell}$ and the image of $P_{\vell}$.   
 
 Let $(u_0,u_1)\in \HHH$. Let $(z_0,z_1)$ be the only element of $\ZZZZ_{\vell}$ such that 
 $$ P_{\vell}(z_0,z_1)=P_{\vell}(u_0,u_1),$$
 and let $(w_0,w_1)=(u_0,u_1)-(z_0,z_1)$, which is in the kernel of $P_{\vell}$. Let $w$ be the solution of \eqref{LWl} with initial data $(w_0,w_1)$. Then by \eqref{LL2'},
\begin{multline}
 \label{LL8}
 \left\|(w_0,w_1)\right\|^2_{\HHH} \leq C\sum_{\pm} \lim_{t\to \pm\infty} \int_{|x|>|t|} |\nabla_{t,x}w(t,x)|^2\,dx\\
 \leq C\sum_{\pm} \lim_{t\to\pm \infty} \int_{|x|>[t|} |\nabla_{t,x}u(t,x)|^2\,dx,
\end{multline}
where we have used \eqref{LL7'} to obtain the second line. 

Since $\pi_{\ZZZ_{W_{\vell}}^{\bot}}(u_0,u_1)=\pi_{\ZZZ_{W_{\vell}}^{\bot}}(w_0,w_1)$, we deduce from \eqref{LL8}:
\begin{equation*}
 \left\|\pi_{\ZZZ_{W_{\vell}}^{\bot}}(u_0,u_1)\right\|^2_{\HHH} 
 \leq C\sum_{\pm} \lim_{t\to\pm \infty} \int_{|x|>[t|} |\nabla_{t,x}u(t,x)|^2\,dx,
\end{equation*}
which concludes the proof of Theorem \ref{T:LW}.
\end{step}

\subsection{Reduction to a qualitative statement}
\label{SS:qual}
In this subsection, we prove that Theorem \ref{T:LW''} implies Theorem \ref{T:LW'} (and thus, according to the preceding subsection, Theorem \ref{T:LW}). The argument is quite general and works around any stationary solution. We will thus fix $Q\in \hdot$ such that $-\Delta Q=|Q|^{\frac{4}{N-2}}Q$ and consider the linearized equation
\begin{equation}
 \label{LWQ}
 \partial_t^2u+L_Qu=0, 
\end{equation} 
where $L_Q:= -\Delta -\frac{N+2}{N-2}|Q|^{\frac{4}{N-2}}$.
We recall (see \cite{DuKeMe13}) that $Q$ is of class $C^2$ and
\begin{equation}
\label{bound_Q}
\exists C_Q>0,\; \forall x\in \RR^N,\quad |Q(x)|\leq \frac{C_Q}{1+|x|^{N-2}}.
\end{equation} 
\begin{prop}
 \label{P:qual_to_quant}
 Assume that $N\geq 3$ is odd.
 Let $A$ be a finite dimensional subspace of $\HHH$, and $A^{\bot}$ the orthogonal complement of $A$ in $\HHH$. We assume that the following implication is true:
 \begin{multline}
  \label{QQ1}
  (u_0,u_1) \in A^{\bot} 
  \text{ and } \sum_{\pm}\lim_{t\to \pm\infty} \int_{|x|\geq |t|}|\nabla_{t,x}u(t,x)|^2=0\\
  \Longrightarrow (u_0,u_1)\equiv (0,0).
 \end{multline} 
 Then there exists a constant $C>0$ such that for all $(u_0,u_1)\in \HHH$,
 \begin{multline}
  \label{QQ2}
  \int |\nabla u_0(x)|^2+(u_1(x))^2\,dx\\
  \leq C\|\pi_A (u_0,u_1)\|^2_{\HHH}+C\sum_{\pm}\lim_{t\to \pm\infty} \int_{|x|\geq |t|}|\nabla_{t,x}u(t,x)|^2,
 \end{multline} 
 where $u$ is the solution of \eqref{LW} with initial data $(u_0,u_1)\in \HHH$.
 \end{prop}
 \begin{remark}
 Assume that Theorem \ref{T:LW''} holds. Then \eqref{QQ1} holds and so, by Proposition \ref{P:qual_to_quant}, \eqref{QQ2} holds with $Q=W$ and $A=\ZZZZ$. We claim that the conclusion of Theorem \ref{T:LW'} is true. Indeed, let $(u_0,u_1)\in \HHH$. Let $w$ be the solution of \eqref{LW} with initial data $\pi_{\ZZZZ}(u_0,u_1)$. We have 
  $$\lim_{t\to\pm\infty} \int_{|x|>|t|} |\nabla_{t,x}w(t,x)|^2\,dx=0.$$
  Using \eqref{QQ2} with $A=\ZZZZ$, we obtain, with $\tilde{w}$ the solution of \eqref{LW} with data $\pi_{\ZZZZ^{\bot}}(u_0,u_1)$, so that $u=w+\tilde{w}$, 
  \begin{multline*}
  \left\|\pi_{\ZZZZ^{\bot}}(u_0,u_1)\right\|^2_{\HHH}\leq C\sum_{\pm}\lim_{t\to \pm\infty}\int_{|x|>|t|} |\nabla_{t,x}\tilde{w}(t,x)|^2\,dx
  \\ =C\sum_{\pm}\lim_{t\to \pm\infty}\int_{|x|>|t|} |\nabla_{t,x}u(t,x)|^2\,dx,  
  \end{multline*} 
  which yields the conclusion of Theorem \ref{T:LW'}.
 \end{remark}
 We now turn to the proof of Proposition \ref{P:qual_to_quant}.
 
We recall the following equirepartition of the energy outside the wave cone for the free wave equation \eqref{FW} (see \cite{DuKeMe12}).
\begin{theoint}
 \label{T:equirepartition}
 Assume that $N\geq 3$ is odd. Let $u_F$ be a solution of \eqref{FW} with initial data $(u_0,u_1)$ at $t=0$. Then
 $$ \sum_{\pm }\lim_{t\to \pm \infty} \int_{|x|\geq |t|} |\nabla_{t,x} u_F(t,x)|^2\,dx=\int |\nabla u_0(x)|^2+u_1(x)^2\,dx$$
\end{theoint}
A key ingredient of the proof of Proposition \ref{P:qual_to_quant} is the following lemma:
\begin{lemma}
 \label{L:B1}
 Let $\{u_n\}_n$ be a sequence of solutions of \eqref{LWQ}, with initial data 
 $$\vec{u}_{n\restriction t=0}=(u_{0,n},u_{1,n})\in \HHH,$$
 such that 
 \begin{equation*}
 (u_{0,n},u_{1,n})\xrightharpoonup[n\to\infty]{}0\text{ in }\HHH.
 \end{equation*} 
 Let $u_{F,n}$ be the solution of the free wave equation \eqref{FW} 
 with initial data 
 $$\vec{u}_{F,n\restriction t=0}=(u_{0,n},u_{1,n})\in \HHH.$$
Then 
\begin{equation}
 \label{B1}
\lim_{n\to\infty}\left(\sup_{t\in \RR}\int_{|x|\geq |t|} |\nabla_{t,x}(u_n(t,x)-u_{F,n}(t,x))|\,dx\right)=0.
\end{equation} 
\end{lemma}
\begin{proof}
\setcounter{step}{0}
\begin{step}[Convergence of the potential term to $0$]
 Let 
 $$\eps_n=\left\|\indic_{\{|x|\geq |t|\}} |Q|^{\frac{4}{N-2}} u_{F,n}\right\|_{L^1(\RR,L^2)}.$$
 We first prove
 \begin{equation}
  \label{B1'}
  \lim_{n\to\infty} \eps_n=0.
 \end{equation} 
 Using the profile decomposition recalled in Subsection \ref{SS:profile}, it is sufficient to prove \eqref{B1'} assuming one of the following
\begin{equation}
 \label{B2}
 \lim_{n\to\infty} \left\|u_{F,n}\right\|_{L^{\frac{2(N+1)}{N-2}}_{t,x}}=0
\end{equation} 
or
\begin{equation}
 \label{B3}
 u_{F,n}=\frac{1}{\lambda_n^{\frac{N}{2}-1}} U_F\left( \frac{t-t_n}{\lambda_n},\frac{x-x_n}{\lambda_n} \right)
\end{equation} 
where $U_F$ is a finite energy solution of the free wave equation  \eqref{FW} and
\begin{equation}
 \label{B3'}
 \lim_{n\to\infty} \frac{|x_n|+|t_n|}{\lambda_n}+\frac{1}{\lambda_n}+\lambda_n=+\infty.
\end{equation} 
Indeed, since by the assumptions of the lemma $\{(u_{0,n},u_{1,n})\}_n$ converges weakly to $0$ in $\HHH$, there is no nonzero profile $U^j_F$ in the profile decomposition of $\{(u_{0,n},u_{1,n})\}_n$ such that the corresponding sequence of parameters $\{\lambda_n^j,t_n^j,x_n^j\}_n$ satisfies
$$\limsup_{n\to\infty} \frac{|x_n^j|+|t_n^j|}{\lambda_n^j}+\frac{1}{\lambda_n^j}+\lambda_n^j<\infty.$$
Translating $U_F$ in time and space and extracting subsequences if necessary, we can assume without loss of generality:
\begin{equation}
 \label{B4}
 \lim_{n\to\infty} \frac{t_n}{\lambda_n}\in \{-\infty,0,\infty\},\quad \lim_{n\to\infty}\lambda_n\in\{0,1,\infty\}\text{ and }\lim_{n\to\infty} \frac{|x_n|}{\lambda_n}\in \{0,\infty\}.
\end{equation} 
Using \eqref{bound_Q}, we have 
\begin{equation}
\label{integrQ}
Q^{\frac{4}{N-2}}\indic_{\{|x|\geq |t|\}}\in L^{\frac{2(N+1)}{N+4}}_tL^{\frac{2(N+1)}{3}}_x. 
\end{equation} 
By H\"older's inequality,
$$\eps_n\leq \left\|\indic_{\{|x|\geq |t|\}} |Q|^{\frac{4}{N-2}}\right\|_{L^{\frac{2(N+1)}{N+4}}_tL^{\frac{2(N+1)}{3}}_x}\left\|u_{F,n}\right\|_{L^{\frac{2(N+1)}{N-2}}_{t,x}}.$$
This proves that \eqref{B2} implies \eqref{B1'}. Next, we assume \eqref{B3} and \eqref{B3'}. By Strichartz estimate, $U_F\in L^{\frac{2(N+1)}{N-2}}_{t,x}$. By H\"older inequality and \eqref{integrQ}, we see that the map
$$U\mapsto \indic_{\{|x|\geq |t|\}}|Q|^{\frac{4}{N-2}}U$$
is continuous from $L^{\frac{2(N+1)}{N-2}}_{t,x}$ to $L^1_tL^2_x$. By density, we deduce that to prove \eqref{B1'}, it is sufficient to check that 
\begin{equation}
\label{B5}
\lim_{n\to\infty} \left\|\indic_{\{|x|\geq |t|\}} |Q|^{\frac{4}{N-2}}  \frac{1}{\lambda_n^{\frac{N}{2}-1}} U\left( \frac{t-t_n}{\lambda_n},\frac{x-x_n}{\lambda_n} \right)  \right\|_{L^1(\RR,L^2)}=0 .
\end{equation} 
whenever $U\in C_0^{\infty}(\RR^{1+N})$. For $R\gg 1$, we let 
$$A_{n,R}=\left\{(t,x)\in \RR\times \RR^N,\; \frac{1}{R}\leq \frac{|t-t_n|}{\lambda_n}\leq R,\; \frac{1}{R}\leq \frac{|x-x_n|}{\lambda_n}\leq R\right\}.$$
Note that by \eqref{B3'}, \eqref{B4}, at fixed $R$, for almost every $(t,x)$ in $\RR\times \RR^N$,
\begin{equation}
\label{CVAnR}
\lim_{n\to \infty} \indic_{A_{n,R}}(t,x)=0. 
\end{equation} 
Denoting by 
$$ U_n(t,x)=\frac{1}{\lambda_n^{\frac{N}{2}-1}} U\left( \frac{t-t_n}{\lambda_n},\frac{x-x_n}{\lambda_n} \right),$$
we see that 
\begin{multline*}
\left\|\indic_{\{|x|\geq |t|\}} |Q|^{\frac{4}{N-2}}U_n\right\|_{L^1_tL^2_x}\\
\leq \left\|\indic_{\{|x|\geq |t|\}} \indic_{A_{n,R}}|Q|^{\frac{4}{N-2}}U_n\right\|_{L^1_tL^2_x}+\left\|\indic_{\{|x|\geq |t|\}}\indic_{{}^CA_{n,R}}|Q|^{\frac{4}{N-2}}U_n\right\|_{L^1_tL^2_x}.
\end{multline*}
We have
\begin{multline*}
 \left\|\indic_{\{|x|\geq |t|\}} \indic_{A_{n,R}}|Q|^{\frac{4}{N-2}}U_n\right\|_{L^1_tL^2_x}\\
 \leq \left\|U_n\right\|_{L^{\frac{2(N+1)}{N-2}}_{t,x}}\left\|\indic_{\{|x|\geq t\}}|Q^{\frac{4}{N-2}}\indic_{A_{n,R}}\right\|_{L^{\frac{2(N+1)}{N+4}}_tL^{\frac{2(N+1)}{3}}_x},
\end{multline*}
which goes to $0$ as $n$ goes to infinity by dominated convergence and \eqref{CVAnR}.
Furthermore, for large $R$, $\indic_{{}^C A_{n,R}}U_n=0$ since $U$ is compactly supported. This concludes the proof of \eqref{B1'}.
\end{step}
\begin{step}[End of the proof]
We let $\tilde{u}_n$ be the solution of 
\begin{equation}
 \label{B6}
\partial_t^2\tilde{u}_n-\Delta\tilde{u}_n-\frac{N+2}{N-2}\indic_{\{|x|\geq |t|\}} |Q|^{\frac{4}{N-2}} \tilde{u}_n=0
\end{equation} 
with the same initial data as $u_n$ at $t=0$. By finite speed of propagation,
$$|x|>|t|\Longrightarrow \tilde{u}_n(t,x)=u_n(t,x),$$
and it is sufficient to prove \eqref{B1} with $\tilde{u}_n$ instead of $u_n$. Let $w_n=u_{F,n}-\tilde{u}_n$. Then
\begin{equation}
\label{eq_wn} 
\left\{ 
\begin{aligned}
 (\partial_t^2-\Delta)w_n&=-\frac{N+2}{N-2}\indic_{\{|x|>|t|\}}|Q|^{\frac{4}{N-2}}u_{F,n}+\frac{N+2}{N-2}\indic_{\{|x|>|t|\}}|Q|^{\frac{4}{N-2}} w_n\\
\vec{w}_{n\restriction t=0}&=(0,0).
\end{aligned}
\right.
\end{equation} 
Hence for all $T>0$, 
\begin{multline*}
 \left\| \indic_{\{|x|>|t|\}}w_n\right\|_{L^{\frac{2(N+1)}{N-2}}([0,T]\times \RR^N)}\lesssim  \eps_n+\left\||Q|^{\frac{4}{N-2}}\indic_{\{|x|>|t|\}} w_n\right\|_{L^1_t([0,T],L^2_x)}\\
\lesssim  \eps_n+\int_0^T \left\||Q|^{\frac{4}{N-2}}\indic_{\{|x|>|t|\}}\right\|_{L^{\frac{2(N+1)}{3}}_x} \left\|\indic_{\{|x|>|t|\}} w_n(t)\right\|_{L^{\frac{2(N+1)}{N-2}}_x}\,dt.
\end{multline*}
Using a Gr\"onwall type inequality (see e.g. Lemma 8.1 in \cite{FaXiCa11}) and \eqref{integrQ}, we obtain that for all $T\geq 0$, 
$$\left\|\indic_{\{|x|> |t|\}}w_n\right\|_{L^{\frac{2(N+1)}{N-2}}([0,T]\times \RR^N)} \leq C_Q \eps_n,$$
where $C_Q$ depends only on $\left\|\indic_{\{|x|>|t|\}}Q\right\|_{L^{\frac{2(N+1)}{N+4}}_tL_x^{\frac{2(N+1)}{3}}},$
and $ \eps_n$ goes to $0$ as $n$ goes to infinity according to Step 1. The arbitrariness of $T>0$ and the same argument for $T<0$ imply
$$\left\|\indic_{\{|x|> |t|\}}w_n\right\|_{L^{\frac{2(N+1)}{N-2}}(\RR_t\times\RR^{N}_x)} \leq 3C_Q \eps_n.$$
Going back to the equation \eqref{eq_wn} and using Strichartz estimates again, we deduce
$$ \sup_{t\in \RR} \|\vec{w}_n(t)\|_{\HHH} \lesssim \eps_n\underset{n\to\infty}{\longrightarrow} 0,$$
which concludes the proof of Lemma \ref{L:B1}.
\end{step}
 \end{proof}
\begin{proof}[Proof of Proposition \ref{P:qual_to_quant}]
We assume \eqref{QQ1} and prove \eqref{QQ2} by contradiction, assuming that there is a sequence $\{u_n\}_n$ of solutions of \eqref{LWQ} with initial data $(u_{0,n},u_{1,n})$ such that 
\begin{gather}
 \label{A4} \left\|(u_{0,n},u_{1,n})\right\|_{\HHH}=1\\
\label{A5}
\lim_{n\to\infty}\left( 
\left\|\pi_A(u_{0,n},u_{1,n})\right\|_{\HHH}^2+\sum_{\pm}\lim_{t\to \pm\infty} \int_{|x|>|t|}|\nabla_{t,x}u_n(t,x)|^2\right)=0.
\end{gather}
Extracting subsequences, we assume that $(u_{0,n},u_{1,n})$ has a weak limit $(u_0,u_1)$ in $\HHH$ as $n\to\infty$. We distinguish between two cases.
\begin{cas} We first consider the case where $(u_0,u_1)\equiv (0,0)$.
Let $u_{F,n}$ be the solution of the free wave equation \eqref{FW} with the same initial data as $u_n$ at $t=0$. Then 
$$\|\vec{u}_{F,n}(0)\|_{\HHH}=1$$
and, by \eqref{A4}, \eqref{A5} and Lemma \ref{L:B1},
$$ \lim_{n\to\infty} \left( \lim_{t\to\pm \infty} \int_{|x|>|t|} |\nabla_{t,x} u_{F,n}(t,x)|^2\,dx\right)=0.$$
This contradicts the equirepartition of the energy (Theorem \ref{T:equirepartition}) above. 
\end{cas}
\begin{cas}
 We next assume that $(u_0,u_1) \not\equiv (0,0)$. Using the weak convergence of $(u_{0,n},u_{1,n})$ to $(u_0,u_1)$, we obtain
\begin{equation}
 \label{A6}
 (u_0,u_1)\in A^{\bot}.
\end{equation} 
We claim
\begin{equation}
 \label{A7}
\sum_{\pm} \lim_{t\to\pm  \infty} \int_{|x|>|t|} |\nabla_{t,x}u(t,x)|^2\,dx=0.
\end{equation} 
In view of \eqref{A6}, \eqref{A7} and since $(u_0,u_1)$ is nonzero we would obtain a contradiction from \eqref{QQ1}. To conclude the proof of \eqref{QQ2}, we are thus reduced to proving \eqref{A7}. 

For all $n$, we denote by $H_{\pm,n}$ the asymptotic states given by Corollary \ref{Co:radiation} (or rather its analog for the equation \eqref{LWQ}), corresponding to the solution $u_n$ of \eqref{LWQ}. By \eqref{A5},
$$ \lim_{n\to\infty}\left\|H_{\pm,n}\right\|_{L^2([0,+\infty)\times S^{N-1})}=0.$$
Using that the linear map $\Phi$ defined in the corollary is continuous 
$$\HHH \to \left(L^2([0,+\infty)\times S^{N-1})\right)^2$$
for the strong topologies of these spaces, and thus also for the weak topologies, we obtain 
$$ H_{\pm,n}\xrightharpoonup[n\to\infty]{} H_{\pm}$$
weakly in $L^2([0,+\infty)\times S^{N-1})$. As a consequence, $H_{\pm}\equiv 0$. This means exactly that \eqref{A7} holds, concluding the proof of \eqref{QQ2}.
\end{cas}
\end{proof}

\section{Proof of the uniqueness theorem}
\label{S:uniqueness}
In this Section, we prove the uniqueness Theorem \ref{T:LW''}. We will start, in Subsection \ref{SS:radial} by projecting the equation \eqref{LW} on spherical harmonics, which reduces the proof to the proof of a similar uniqueness result for a family of radial wave equations in odd space dimensions $D\geq N$. The next four subsections are dedicated to the proof of this result.

\subsection{Reduction to a radial problem}
\label{SS:radial}
Recall that the eigenfunctions of the Laplace-Beltrami operator on $S^{N-1}$ are of the form $\Phi(\theta)$, where $\Phi$ is a homogeneous harmonic polynomial in the variables $(x_1,\ldots,x_n)$. If $\Phi$ is such a polynomial and $\nu\in \NN$ is its degree, then $-\Delta_{S^{N-1}}u=\nu(\nu+N-2)u$.
Consider a Hilbert basis $\left(\Phi_{k}\right)_{k\in \NN}$ of the Laplace-Beltrami operator on $S^{N-1}$, and let $(\nu_k)_{k\in \NN}$ be the sequence of natural integers, so that for each $k\in \NN$
$$-\Delta_{S^{N-1}} \Phi_k=\nu_k(\nu_k+N-2)\Phi_k.$$
Let $u$ be a solution of \eqref{LW} such that \eqref{ext_0} holds
We let, for $t\in \RR$ and $r>0$,
 $$u_k(t,r)= \int_{S^{N-1}} \Phi_k(\theta)u(t,r\theta)\,d\sigma(\theta).$$
We will identify the function $u$ with the corresponding radial function on $\RR^N$. We have
$$ \vec{u}_k\in C^0(\RR,\HHH)$$
and
\begin{equation}
\label{T0}
 \partial_t^2u_k-\Delta u_k+\frac{\nu_k(\nu_k+N-2)}{r^2}u_k +\frac{N+2}{N-2}W^{\frac{4}{N-2}}u_k=0.
\end{equation}
Letting $v_k=r^{-\nu_k} u_{k}$, and dropping the suffixes $k$ to lighten notations (so that $\nu=\nu_k$, $v=v_k$), we obtain:
\begin{equation}
 \label{T1}
\partial_t^2v-\Delta_D v-\frac{N+2}{N-2}W^{\frac{4}{N-2}} v=0,
\end{equation} 
with initial data
\begin{equation*}
\vec{v}_{\restriction t=0}=(v_0,v_1):=r^{-\nu_k}\vec{u}_k(0), 
\end{equation*} 
where 
$$\Delta_D=\partial_r^2+\frac{D-1}{r}\partial_r$$
is the radial part of the Laplace operator in dimension $D=N+2\nu$. Considering $v$ as a radial function on $\RR^D$, we see that 
\begin{equation}
\label{cont_v}
\vec{v} \in C^0\left(\RR,(\hdot\times L^2)(\RR^D)\right), 
\end{equation} 
and using \eqref{ext_0}, and that 
\begin{equation}
 \label{T1bis}
\lim_{t\to \pm\infty}\int_{|x|\geq |t|} \frac{1}{|x|^2} v^2(t,x)\,dx=0,
\end{equation} 
we obtain
\begin{equation}
\label{T3}
 \lim_{t\to \pm\infty}\int_{|t|}^{+\infty}(\partial_{t,r} v)^2r^{D-1}\,dr=0.
\end{equation} 
The limit \eqref{T1bis} follows from \eqref{ext_0} and the exterior Hardy's inequality (or from the asymptotic behavior of the solutions of \eqref{LW}, see Corollary \ref{Co:radiation}). 

We note also that, for $j\in \llbracket 1, N\rrbracket$, 
$$\partial_{x_j}W=-\frac{1}{N} \frac{x_j}{\left(1+\frac{|x|^2}{N(N-2)}  \right)^{\frac N2}},$$
and thus 
$$\frac{1}{|x|}\partial_{x_j}W=-\frac{1}{N} \frac{\theta_j}{\left(1+\frac{|x|^2}{N(N-2)}\right)^{\frac N2}},$$
where $\theta_j=x_j/|x|\in S^{N-1}$.
Furthermore, $\theta \mapsto \theta_j$ is a spherical harmonic of degree $1$ and $(\theta_j)_{j \in \llbracket 1, N\rrbracket}$ spans the vector space of degree $1$ spherical harmonics.
As a consequence, we see that the proof of Theorem \ref{T:LW''} reduces to the proof of the following proposition:
\begin{prop}
 \label{P:T2}
Let $N,D$ be two odd integers with $D\geq N$.
Let $v$ be a radial solution of \eqref{T1} that satisfies \eqref{cont_v} and \eqref{T3}. 
Then $\vec{v}(0)=(0,0)$ if $D\geq N+4$. Furthermore, $\vec{v}(0)$ is an element of 
\begin{equation*}
\begin{cases}
               \vect\big\{ (\Lambda W,0)\big\}& \text{ if }N=D=3,\\
            \vect\big\{ (\Lambda W,0),(0,\Lambda W)\big\}&\text{ if }N=D\geq 5,\\
            \vect\left\{\left(\left(1+\frac{|x|^2}{N(N-2)}\right)^{-\frac N2},0\right),\left(0,\left(1+\frac{|x|^2}{N(N-2)}\right)^{-\frac N2}\right)\right\}&\text{ if }D=N+2.
              \end{cases}
\end{equation*}
\end{prop}
We divide the proof of Proposition \ref{P:T2} into 4 parts. In Subsection \ref{SS:ODE}, we study the differential equation satisfied by the stationary solutions of the equation \eqref{T1}. In Subsection \ref{SS:cmpct}, we prove that the initial data of a solution of \eqref{T1} satisfying the assumptions of Proposition \ref{P:T2} coincide for large $r$ with a stationary solution.  In Subsection \ref{SS:enlargement}, we use a bound from below of the energy for compactly supported solutions of \eqref{T1} to prove that the initial data actually coincide with a stationary solution for all $r>0$. In Subsection \ref{SS:conclusion} we combine the preceding results to conclude the proof.
\subsection{Preliminaries on a differential equation}
\label{SS:ODE}
In this subsection, we fix two odd integers $N$, $D$, with $3\leq N\leq D$ and study the differential operator:
$$ P=\frac{d^2}{dr^2}+\frac{D-1}{r}\frac{d}{dr} +\frac{N+2}{N-2}W^{\frac{4}{N-2}},$$
which was introduced in the previous subsection, by projecting the operator $L_W$ on spherical harmonics. As before, $W$ is the ground state \emph{in space dimension $N$}, defined in \eqref{defW}. 
Denoting by $\ent{x}$ the integer part of a real number $x$, we prove:
\begin{prop}
 \label{P:Z1}
Let $N$, $D$ and $P$ be as above. Then there exists functions $(Z_k)_{1\leq k\leq \ent{\frac{D+2}{4}}}$ with $Z_k\in C^{\infty}\Big((0,\infty)\Big)$, with the following properties:
\begin{gather}
 \label{Z1} \left|Z_1(r)-\frac{1}{r^{D-2}}\right|\lesssim \frac{1}{r^D}, \quad r\geq 1\\
 \label{Z2} \left|Z_1'(r)+\frac{D-2}{r^{D-1}}\right|\lesssim \frac{1}{r^{D+1}}, \quad r\geq 1\\
\label{Z3} PZ_1=0,
 \end{gather}
and, for $2\leq k\leq \ent{\frac{D+2}{4}}$, there exists $c_k\in \RR\setminus\{0\}$ such that
\begin{gather}
 \label{Z4} 
 \left|Z_k(r)-\frac{c_k}{r^{D-2k}}\right|\lesssim \frac{\log r}{r^{D+2-2k}},\quad r\geq 1\\
 \label{Z5}
 \left|Z_k'(r)+\frac{c_k(D-2k)}{r^{D-2k+1}}\right|\lesssim \frac{\log r}{r^{D+3-2k}},\quad r\geq 1\\
\label{Z6}
PZ_k=Z_{k-1}.
 \end{gather}
Furthermore, 
\begin{align}
\label{Z6a}
D=N &\Longrightarrow Z_1(r)=-\frac{2}{N^{\frac{N}{2}-1} (N-2)^{\frac N2}}\Lambda W,\\
\label{Z6b}
D=N+2&\Longrightarrow Z_1(r)=\frac{1}{\left(N(N-2)\right)^{\frac N2}}\left( 1+\frac{r^2}{N(N-2)} \right)^{-\frac N2},
\end{align}
Finally, let $k\in \NN$ with $1\leq k\leq \ent{\frac{D+2}{4}}$. Then if $k\geq 2$ and $Z\in \vect\{Z_2,\ldots,Z_k\}\setminus\{0\}$, or if $D\geq N+4$ and $Z\in \vect\{Z_1,\ldots,Z_k\}\setminus\{0\}$, then
\begin{equation}
\label{Z6c}
\exists d\in \RR\setminus\{0\},\; \exists \theta\in (0,2k]\cap 2\NN,\quad Z(r)\sim \frac{d}{r^{D-\theta}},\quad r\to 0.
\end{equation}
\end{prop}
For the proof of the Proposition, we need the following two Lemmas concerning the homogeneous and inhomogeneous ODE's
\begin{equation}
 \label{H}
 Py=0
\end{equation} 
and
\begin{equation}
 \label{In}
 Py=f
\end{equation} 
on $(0,\infty)$.
\begin{lemma}
\label{L:Z1}
There exist two basis of solutions of \eqref{H}, $(y_0,z_0)$, $(y_{\infty},z_{\infty})$, where $y_0,z_0,y_{\infty}, z_{\infty}$ are smooth on $(0,\infty)$, and satisfy:
\begin{alignat}{2}
 \label{Z7}
 \left|y_0(r)-1\right|&\lesssim r^2,&\quad |y_0'(r)|&\lesssim r,\\
 \label{Z8}
 \left|z_0(r)-\frac{1}{r^{D-2}}\right|&\lesssim 1+\frac{1}{r^{D-4}},&\quad |z_0'(r)|&\lesssim \frac{1}{r^{D-1}},
\end{alignat}
for $0<r\leq 1$
and
 \begin{alignat}{2}
 \label{Z9} 
 \left|y_{\infty}(r)-1\right|&\lesssim \frac{1}{r^2}+\frac{1}{r^{D-2}},&\quad |y'_{\infty}(r)|&\lesssim \frac{1}{r^3}+\frac{1}{r^{D-1}},\\
 \label{Z10} 
 \left|z_{\infty}(r)-\frac{1}{r^{D-2}}\right|&\lesssim \frac{1}{r^{D}},&\quad \left|z'_{\infty}(r)+\frac{D-2}{r^{D-1}}\right|&\lesssim \frac{1}{r^{D+1}},
 \end{alignat}
for $r\geq 1$.
 \end{lemma}
\begin{lemma}
\label{L:Z2}
 Let $i\in\{1,\infty\}$ and $y_i$, $z_i$ be as in Lemma \ref{L:Z1}. Let $w_i$ be the Wronskian $y_iz_i'-z_iy_i'$. Then $w_i=\frac{1}{\omega_i r^{D-1}}$ for some constant $\omega_i\neq 0$. Furthermore, if $f\in C^0\big((0,\infty)\big)$, the general solution $y$ of \eqref{In} is given by 
 $$y(r)=\alpha(r)y_i(r)+\beta(r)z_i(r),\quad y'(r)=\alpha(r)y_i'(r)+\beta(r)z_i'(r),$$
 where 
 $$\alpha'(r)=-\omega_i f(r) r^{D-1}z_i(r), \quad \beta'(r)=\omega_if(r)r^{D-1}y_i(r).$$
\end{lemma}
Lemma \ref{L:Z2} follows from standard variation of parameters and we omit the proof. Lemma \ref{L:Z1} is also very classical. We give a sketch of proof for the sake of completeness.
\begin{proof}[Sketch of proof of Lemma \ref{L:Z1}]
 Writing the equation \eqref{H} 
 $$\frac{d}{dr} \left(r^{D-1}\frac{d y}{dr}\right)=-r^{D-1}\frac{N+2}{N-2}W^{\frac{4}{N-2}}y,$$
 and integrating twice, we see that to construct $y_0$, it is sufficient to find a solution of the equation
 \begin{equation}
  \label{Z11}
  y_0=1-\Phi_0(y_0),\quad \Phi_0(y)(r):= \int_0^{r}\int_0^{\rho} \left( \frac{s}{\rho} \right)^{D-1}\frac{N+2}{N-2}W^{\frac{4}{N-2}}(s)y_0(s)\,ds\,d\rho.
 \end{equation} 
 This can be done by checking that if $\eps>0$ is small, $\Phi_0$ is a contraction mapping on $C^0\left([0,\eps],\|\cdot\|_{\infty}\right)$, and then extending the solution to $(0,\infty)$ by the linear Cauchy-Lipschitz theorem. The bounds \eqref{Z7} follow easily from \eqref{Z11}.
 
 We next fix a small $\eps>0$ (such that $y_0>0$ on $(0,\eps)$), and let, for $0<r<\eps$
 $$ z_0(r)=a_0(r)y_0(r),\quad a_0=-(D-2)\int_{\eps}^r \frac{1}{s^{D-1}y_0^2(s)}ds.$$
 Noting that 
 $$a_0''+a_0'\left( \frac{2y_0'}{y_0}+\frac{D-1}{r} \right)=0,$$
 we see that $Pz_0=0$. The estimates \eqref{Z8} are easy to check.
 
 Using again the Banach fixed point theorem, we define $z_{\infty}$ as the unique solution of 
 \begin{equation}
  \label{Z12}
  z_{\infty}(r)=\frac{1}{r^{D-2}} - \Phi_{\infty}(z_{\infty})(r),
   \end{equation} 
  where
  \begin{equation*}
  \Phi_{\infty}(z)(r):=\int_{r}^{+\infty}\int_{\rho}^{+\infty} \left( \frac{s}{\rho} \right)^{D-1}\frac{N+2}{N-2}W^{\frac{4}{N-2}}(s)z_{\infty}(s)\,ds\,d\rho,
 \end{equation*}
 on 
 $$X_{A}= \left\{ y\in C^0\left( [A,+\infty) \right),\;\|y\|_{X_{A}}:=\sup_{s\geq A}|y(s)|s^{D-2}<\infty\right\},$$
 where $A$ is a large constant. Using that $W(r)^{\frac{4}{N-2}}\lesssim 1/r^4$, it is easy to check that $\Phi_{\infty}$ is a contraction mapping on $X_A$. The estimates \eqref{Z10} follow from \eqref{Z12} and the fact that $z_{\infty}\in X_A$.

Finally, we let $y_{\infty}(r)=b_{\infty}(r)z_{\infty}(r)$, where 
$$ b_{\infty}(r)=(D-2)\int_A^r \frac{1}{s^{D-1}z_{\infty}^2(s)}\,ds.$$
The equation $Py_{\infty}=0$ and the estimates \eqref{Z9} are again easy to check.
 \end{proof}
We next prove Proposition \ref{P:Z1}.
\begin{proof}
\setcounter{step}{0}
 \begin{step}[Existence and estimates for large $r$]
 We let $Z_1(r)=z_{\infty}(r)$, and see that \eqref{Z1}, \eqref{Z2} and \eqref{Z3} are satisfied. If $D\in \{3,5\}$, $2$ is larger than $\ent{\frac{D+2}{4}}$ and we are done. If $D\geq 7$, we construct $Z_k$ for $2\leq k\leq \ent{\frac{D+2}{4}}$ satisfying \eqref{Z4}, \eqref{Z5} and \eqref{Z6} by induction. 
 
 Let $k\in \NN$ such that $2\leq k\leq \ent{\frac{D+2}{4}}$. Assume that $Z_{k-1}$ is known and satisfies the desired estimates. According to Lemma \ref{L:Z2}, the formula
 $$Z_k(r)=\alpha_k(r)y_{\infty}(r)+\beta_k(r)z_{\infty}(r),$$
 where 
 \begin{equation}
  \label{Z14}
  \left\{
  \begin{aligned}
\alpha_k(r) &= \omega_{\infty}\int_r^{\infty} z_{\infty}(s) Z_{k-1}(s)s^{D-1}\,ds\\
\beta_k(r)&=\omega_{\infty}\int_1^ry_{\infty}(s)Z_{k-1}(s)s^{D-1}\,ds
  \end{aligned}
\right.
 \end{equation} 
 yields a solution of \eqref{Z6}.
 It is easy to check, using the asymptotic behaviors of $Z_{k-1}(s)$ (known by the induction hypothesis) and of $z_{\infty}(s)$, that $\alpha_k$ is well defined. It remains to check that
 \eqref{Z4} and \eqref{Z5} are satisfied.
 
We start with the case $k=2$. In view of \eqref{Z14}, 
using the asymptotic behaviour of $z_{\infty}=Z_1$ for large $r$ (see \eqref{Z10}), we have
$$ \alpha_2(r)=\omega_{\infty}\int_r^{\infty}\left(s^{3-D}+O\left( s^{1-D} \right) \right)\,ds=\frac{\omega_{\infty}}{(D-4)r^{D-4}}+O\left( \frac{1}{r^{D-2}} \right),\; r\geq 1,$$
and similarly, using also the asymptotic behaviour \eqref{Z9} of $y_{\infty}$, $\beta_2(r)=\omega_{\infty}\frac{r^2}{2}+O(\log(r))$, $r\geq 1$. As a consequence,
$$ Z_2(r)=\left( \frac{1}{D-4}+\frac 12 \right)\frac{\omega_{\infty}}{r^{D-4}}+O\left( \frac{\log r}{r^{D-2}} \right).$$
This yields \eqref{Z4} for $k=2$, with $c_2=\left( \frac{1}{D-4}+\frac 12 \right)\omega_{\infty}$.  Furthermore,
\begin{multline*}
Z_2'(r)=\alpha_2(r)y_{\infty}'(r)+\beta_2(r)z_{\infty}'(r)\\
=-\frac{(D-2)\omega_{\infty} }{2 r^{D-3}}+O\left( \frac{\log r}{r^{D-1}} \right)=
-\frac{(D-4)c_2}{r^{D-3}}+O\left( \frac{\log r}{r^{D-1}} \right). 
\end{multline*}

We next treat the case $k\geq 3$. By \eqref{Z10}, \eqref{Z14} and the induction hypothesis, for $r\geq 1$:
\begin{equation*}
\alpha_k(r)
=\frac{c_{k-1}\omega_{\infty}}{(D-2k)r^{D-2k}} + O\left(\frac{\log r}{r^{D-2k+2}}  \right).
\end{equation*}
Using also \eqref{Z9}, we obtain similarly, for $r\geq 1$,
\begin{equation*}
 \beta_k(r)=c_{k-1}\omega_{\infty} \frac{r^{2k-2}}{2k-2}+O\left( r^{2k-4}\log r \right).
\end{equation*} 
Combining the above estimates with the asymptotic formulas \eqref{Z9} and \eqref{Z10} for $y_{\infty}$ and $z_{\infty}$, we obtain, for $r\geq 1$
\begin{equation*}
 Z_k(r)=\left( \frac{1}{D-2k}+\frac{1}{2k-2} \right)\frac{c_{k-1}\omega_{\infty}}{r^{D-2k}}+O\left( \frac{\log r}{r^{D-2k+2}} \right).
\end{equation*}
This yields \eqref{Z4} with $c_k=c_{k-1}\left( \frac{1}{D-2k}+\frac{1}{2k-2} \right)\omega_{\infty}$.
Also:
\begin{equation*}
 Z_k'(r)=a_k(r)y'_{\infty}(r)+\beta_k(r)z_{\infty}'(r)=-\frac{c_{k-1}\omega_{\infty}(D-2)}{(2k-2)r^{D-2k+1}}+O\left( \frac{\log r}{r^{D-2k+3}} \right).
\end{equation*}
Since 
$(D-2k)c_k=\frac{c_{k-1}\omega_{\infty}(D-2)}{2k-2}$, this proves \eqref{Z5}.
\end{step}
\begin{step}[Asymptotic behavior at the origin] 
We note that by Lemma \ref{L:Z1}, $Z_1=z_{\infty}$ is the unique solution of $PZ_1=0$ satisfying \eqref{Z1}. 

In the case $N=D$, it follows from this uniqueness property, the equation $P \Lambda W=L_W(\Lambda W)=0$ and the asymptotic behavior of $\Lambda W(r)$ as $r\to\infty$ that $Z_1=-\frac{2}{N^{\frac{N}{2}-1}(N-2)^{\frac{N}{2}}}\Lambda W$, i.e. \eqref{Z6a} holds.

In the case $N=D+2$, the computation before Proposition \ref{P:T2} implies, since $L_W(\partial_{x_j}W)=0$,
$$ P\left( \left( 1+\frac{r^2}{N(N-2)} \right)^{-\frac N2}\right)=0,$$
which yields \eqref{Z6b}.

We next prove \eqref{Z6c} distinguishing between three cases.

We start by proving \eqref{Z6c} when $Z=Z_1$ and $D\geq N+4$. By Lemma \ref{L:Z1}, $Z_1=\gamma_1 y_0+\delta_1 z_0$ for some constants $\gamma_1,\delta_1\in \RR$. We claim that $\delta_1\neq 0$. 
Letting $\nu=\frac{D-N}{2}$ and $\Phi(\theta)$ a spherical harmonic of degree $\nu$, we see that 
$$L_W\left( r^{\nu}\Phi(\theta)Z_1 \right)=0.$$
If $\delta_1=0$, the estimates \eqref{Z7} on $y_0$ and \eqref{Z1}, \eqref{Z2} on $Z_1$ imply $r^{\nu}\Phi(\theta)Z_1(r) \in \hdot(\RR^N)$. This is impossible since
$$\{Z\in \hdot(\RR^N)\;:\; L_WZ=0\}=\vect\left\{ \Lambda W,\partial_{x_1}W,\ldots,\partial_{x_N}W\right\}.$$
Thus $\delta_1 \neq 0$ and $Z_1\sim \frac{\delta_1}{r^{D-2}}$ as $r\to 0$.

We next prove \eqref{Z6c} when $Z=Z_2$, $D=N$ or $D=N+2$.  According to Lemma \ref{L:Z2}, 
$$Z_2(r)=\gamma_2(r)y_0(r)+\delta_2(r)z_0(r),\quad Z_2'(r)=\gamma_2(r)y_0'(r)+\delta_2(r)z_0'(r)$$
where 
$$\gamma_2'(r)=-\omega_0Z_1(r)z_0(r)r^{D-1},\quad \delta_2'(r)=\omega_0 Z_1(r)y_0(r)r^{D-1},\quad r>0.$$
Since $Z_1(r)$ is, by \eqref{Z6a} or \eqref{Z6b}, bounded in the neighborhood of $0$, we see that $\gamma_2'$ and $\delta_2'$ are integrable close to $r=0$, and thus that $\gamma_2(r)$ and $\delta_2(r)$ have limits $\gamma_2(0)$ and $\delta_2(0)$ as $r\to 0$. We next prove $\delta_2(0)\neq 0$. We argue by contradiction. If $\delta_2(0)=0$, then 
$$\delta_2(r)=\omega_0\int_0^r Z_1(s)y_0(s)s^{D-1}\,ds=O\left( r^{D-1} \right)\text{ as }r\to 0.$$
As a consequence, $Z_2'$ and $Z_2$ are bounded close to $r=0$.
Since $Z_2(r)\sim \frac{1}{r^{D-4}}$, $Z_2'(r)\sim \frac{4-D}{r^{D-3}}$ as $r\to \infty$ and $\frac{D+2}{4}>2$ (i.e. $D\geq 7$), we can integrate by parts:
\begin{multline*}
\int_0^{+\infty} (Z_1(r))^2r^{D-1}dr=\int_0^{+\infty} PZ_2(r)Z_1(r)r^{D-1}dr\\
=\int_0^{+\infty} Z_2(r)PZ_1(r)r^{D-1}dr=0. 
\end{multline*}
Thus $Z_1(r)=0$ a.e., a contradiction. Thus $\delta_2(0)\neq 0$, which proves
$$Z_2(r)\sim \frac{\delta_2(0)}{r^{D-2}},\quad r\to 0,$$
yielding \eqref{Z6c} with $Z=Z_2$.

Let 
\begin{equation*}
A_k:=
\begin{cases}
  \vect\{Z_2,\ldots,Z_k\} & \text{ if }D\in \{N,N+2\}\\
  \vect\{Z_1,Z_2,\ldots,Z_k\}&\text{ if }D\geq N+4.
 \end{cases}
\end{equation*}

To conclude the proof of Proposition \ref{P:Z1}, we show that if $2\leq k\leq \ent{\frac{D+2}{4}}$, or if $k=1$ and $D\geq N+4$, then for all $Z\in A_k\setminus\{0\}$, there exists $d\in \RR\setminus \{0\}$, $0<\theta\leq 2k$ ($\theta\in 2\NN$) such that 
\begin{equation}
 \label{Z15}
 Z(r)\sim \frac{d}{r^{D-\theta}},\quad r\to 0.
\end{equation} 
We argue by induction on $k$. The cases ($k=2$, $D\in \{N,N+2\}$) and ($k=1$, $D\geq N+4$) were treated above. Fixing $k\geq 3$ (if $D\in \{N,N+2\}$) or $k\geq 2$ (if $D\geq N+4$), we assume that for all $Z\in A_{k-1}\setminus\{0\}$, there exists $d\in \RR\setminus \{0\}$, $\theta\in 2\NN$ such that $0<\theta\leq 2(k-1)$ and
\begin{equation}
 \label{Z16}
 Z(r)\sim \frac{d}{r^{D-\theta}},\quad r\to 0.
\end{equation} 
Let $Z\in A_k\setminus\{0\}$. If $Z\in A_{k-1}$, then we can use the induction hypothesis on $Z$ and we are done. If not, we let $Y=PZ$. Then $Y\in \vect\{Z_1,\ldots, Z_{k-1}\}\setminus\{0\}$. By the induction hypothesis:
\begin{equation}
\label{Ycase1}
\exists d\neq 0,\; \exists \theta\in (0,2(k-1)]\cap 2\NN,\quad Y(r)\sim \frac{d}{r^{D-\theta}},\quad r\to 0 
\end{equation} 
By Lemma \ref{L:Z2}, we have
$$Z(r)=\gamma(r)y_0(r)+\delta(r)z_0(r),$$
where
\begin{align*}
\gamma'(r)&=-\omega_0Y(r)z_0(r)r^{D-1}\sim -\frac{\omega_0 d}{r^{D-1-\theta}},\quad r\to 0\\
\delta'(r)&=\omega_0 Y(r)y_0(r)r^{D-1}\sim \frac{\omega_0 d}{r^{1-\theta}},\quad r\to 0.
\end{align*}
Since $0<\theta\leq 2k-2< \frac D2-1$, we see that
$$\gamma(r)\sim \omega_0 d\int_r^1 \frac{1}{\sigma^{D-1-\theta}}\,d\sigma\sim \frac{\omega_0 d}{(D-\theta-2)r^{D-\theta-2}}$$
and thus
$$\gamma(r)y_0(r)\sim \frac{\omega_0d}{(D-\theta-2)r^{D-\theta-2}},\quad r\to 0.$$
Furthermore, $\delta'$ is integrable at the origin, and thus $\delta(r)$ has a limit $\delta(0)$ as $r\to 0$. We distinguish between two cases. \setcounter{cas}{0}
\begin{cas}[$\delta(0)\neq 0$] Since $\theta>0$, we see that $\gamma y_0=o\left( \delta z_0 \right)$ as $r\to 0$. Thus
$$ Z(r)\sim \delta(r)z_0(r)\sim \frac{\delta(0)}{r^{D-2}},\quad r\to 0,$$
and \eqref{Z15} follows with $\theta=2$, $d=\delta(0)$.
\end{cas}
\begin{cas}[$\delta(0)=0$] Writing $\delta(r)=\int_0^{r}\delta'(s)\,ds$, we obtain
$$\delta(r)z_0(r)\sim \omega_0 d \frac{1}{\theta r^{D-2-\theta}}.$$
Combining with the estimate on $\gamma(r)$, we deduce
$$ Z(r)\sim \omega_0d\left( \frac{1}{\theta}+\frac{1}{D-\theta-2} \right)\frac{1}{r^{D-2-\theta}},$$
which yields again \eqref{Z15}, with $\theta$ replaced by $\theta+2\in 2\NN \cap (0,2k]$.
\end{cas}
\end{step}
\end{proof}
\subsection{Compact support of the initial data}
\label{SS:cmpct}
In this subsection, we prove:
\begin{prop}
 \label{P:cpct}
Let $N$ and $D$ be two odd integers such that $3\leq N\leq D$. Let  $v$ be a (radial) solution of \eqref{T1}, that satisfies \eqref{cont_v} and \eqref{T3}. Then there exist real numbers $\zeta_1,\ldots,\zeta_{\ent{\frac{D+2}{4}}},\eta_1,\ldots,\eta_{\ent{\frac{D}{4}}}$ such that the essential support of 
\begin{equation}
\label{D4} 
\vec{v}(0)-\sum_{k=1}^{\ent{\frac{D+2}{4}}} \zeta_k (Z_k,0)-\sum_{k=1}^{\ent{\frac{D}{4}}}\eta_k(0,Z_k)
\end{equation} 
is compact. 
 \end{prop}
Here $Z_1,\ldots,Z_{\ent{\frac{D+2}{4}}}$ are the functions defined in Proposition \ref{P:Z1}. 

Until the end of the current section, the underlying space dimension is $D$ and not $N$, and we will slightly change the notations, denoting by $\HHH$ the space $(\hdot\times L^2)(\RR^D)$.

We will use a result of Kenig, Lawrie, Liu and Schlag \cite{KeLaLiSc15} on the radial wave equation in odd dimension that we now describe. 
If $V$ is a vector space of radial function, we denote by $V(R)$ the vector space of the restriction to $(R,+\infty)$ of the elements in $V$. In particular $\HHH(R)$ is the Hilbert space of radial functions in $(\dot{H}^1\times L^2)(\{|x|>R\})$. In other words, if $\HHH(R)$ is the set of pairs of radial functions $(f,g)$, defined for $r>R$, such that $(f,g)\in \left(L^{\frac{2D}{D-2}}\times L^2\right)\left(\{x\in \RR^D,|x|>R\}\right)$ and
$$\|(f,g)\|^2_{\HHH(R)}:=\int_R^{+\infty}(\partial_rf)^2r^{D-1}dr+\int_R^{+\infty} g^2r^{D-1}dr$$
is finite.

If $S$ is a subspace of $\HHH(R)$, we denote by $S^{\bot}$ its orthogonal complement in $\HHH(R)$, $\pi_S$ the orthogonal projection of $S$ and $\pi_{S^{\bot}}$ the orthogonal projection on $S^{\bot}$. 
We denote by
$$B=\vect\left\{\left(\frac{1}{r^{D-2k_1}},0\right),\left(0, \frac{1}{r^{D-2k_2}}\right),\quad 1\leq k_1\leq \ent{\frac{D+2}{4}},\; 1\leq k_2\leq \ent{\frac{D}{4}}\right\}.$$
Note that $B$ is exactly the space of radial functions $(f,g)$ on $\RR^D$ such that 
$$ \exists k\in \NN \text{ s.t. } \Delta^kf=\Delta^kg=0$$
and such that for all $R>0$,
 $(f,g)_{\restriction (R,\infty)} \in \HHH(R)$.
By \cite{KeLaLiSc15},
\begin{theoint}
\label{T:KLLS}
 Every radial finite energy solution of 
 \begin{equation}
  \label{FWD}
  \partial_t^2v- \partial_r^2v-\frac{D-1}{r}\partial_r v=0 
 \end{equation} 
 satisfies, for all $R>0$,
 $$\max_{\pm} \lim_{t\to \pm\infty} \int_{R+|t|}^{+\infty} (\partial_{r,t}v(t,r))^2\,r^{D-1}\,dr\geq \frac 12
\left\|\pi_{B(R)^{\bot}}{\vec{v}(0))}\right\|_{\HHH(R)}^2.$$
 \end{theoint}

To prove Proposition \ref{P:cpct}, we fix a large positive $R$ (to be specified), and let 
$$ \widetilde{B}=\vect\left\{\left(Z_{k_1},0\right),\left(0, Z_{k_2}\right),\quad 1\leq k_1\leq \ent{\frac{D+2}{4}},\; 1\leq k_2\leq \ent{\frac{D}{4}}\right\}.$$
In view of \eqref{Z1},\ldots,\eqref{Z5}, if $R$ is large, then
$$ \widetilde{B}(R)\cap B(R)^{\bot}=\{(0,0)\}.$$
In particular, we can choose $(\zeta_k)_{1\leq k\leq \ent{\frac{D+2}{4}}}$, $(\eta_k)_{1\leq k\leq \ent{\frac{D}{4}}}$ such that
$$(w_0,w_1):=\vec{v}(0)-\sum_{k=1}^{\ent{\frac{D+2}{4}}} \zeta_k (Z_k,0)-\sum_{k=1}^{\ent{\frac{D}{4}}}\eta_k(0,Z_k)\in B(R)^{\bot}.$$
Let $w$ be the solution of \eqref{T1} with initial data $(w_0,w_1)$. Let $w_F$ be the solution of the free wave equation \eqref{FWD} in dimension $D$ with the same initial data. Note that \eqref{T3} implies 
\begin{equation}
\label{T3'}
 \lim_{t\to \pm\infty}\int_{|t|+R}^{+\infty}(\partial_{t,r} w)^2r^{D-1}\,dr=0.
\end{equation} 
Indeed, the solution of \eqref{T1} for $r\geq |t|+R$ with initial data $(Z_k,0)$ is
$$\sum_{j=1}^k\frac{t^{2k-2j}}{(2k-2j)!}Z_j,$$
and the solution with initial data $(0,Z_k)$ is 
$$\sum_{j=1}^k\frac{t^{2k+1-2j}}{(2k+1-2j)!}Z_j,$$
so that it is easy to check that any solution of \eqref{T1} for $r\geq |t|+R$ with initial data in $\widetilde{B}(R)$ satisfies \eqref{T3'}.

By Lemma \ref{L:linear_approx}, if $R\geq 1$,
\begin{equation}
 \label{T6}
 \left\|\indic_{\{r\geq R+|t|\}}w\right\|_{L^{\frac{2(D+1)}{D-2}}(\RR^{1+D})}\lesssim \|(w_0,w_1)\|_{\HHH(R)},
\end{equation} 
where the implicit constant is independent of $R \gg 1$.

Using Strichartz estimates and finite speed of propagation, we have
\begin{multline*}
\left(\int_{R+|t|}^{+\infty} (\partial_{r,t}w_F(t,r))^2\,r^{D-1}\,dr\right)^{1/2}
\\ \lesssim \left(\int_{R+|t|}^{+\infty} (\partial_{r,t}w(t,r))^2\,r^{D-1}\,dr\right)^{1/2}+\left\|W^{\frac{4}{N-2}}\indic_{\{r\geq R+|t|\}}w\right\|_{L^1(\RR,L^2(\RR^D))}, 
\end{multline*}
and thus, by H\"older inequality and \eqref{T6},
\begin{multline*}
\left(\int_{R+|t|}^{+\infty} (\partial_{r,t}w_F(t,r))^2\,r^{D-1}\,dr\right)^{1/2}
\\ \lesssim \left(\int_{R+|t|}^{+\infty} (\partial_{r,t}w(t,r))^2\,r^{D-1}\,dr\right)^{1/2}+
\frac{1}{R^2}\|(w_0,w_1)\|_{\HHH(R)},
\end{multline*}
where we have used that for $R\gg 1$,
\begin{equation*}
\left\|W^{\frac{4}{N-2}}\indic_{\{r\geq R+|t|\}}\right\|_{L^{\frac{2(D+1)}{D+4}}\left(\RR,L^{\frac{2(D+1)}{3}}\right)}\lesssim \frac{1}{R^2},
\end{equation*} 
which follows from the bound $W^{\frac{4}{N-2}}\lesssim r^{-4}$.

By Theorem \ref{T:KLLS}, and using that $(w_0,w_1)\in B(R)^{\bot}$, and \eqref{T3'}, we obtain, letting $|t|\to \infty$, 
$$ \|(w_0,w_1)\|_{\HHH(R)}\lesssim \frac{1}{R^2}\|(w_0,w_1)\|_{\HHH(R)},$$
which implies, choosing $R$ large, $(w_0,w_1)(r)=0$ for a.e. $r\geq R$.\qed
\subsection{Propagation of the boundary of the support}
\label{SS:enlargement}
In this subsection, we consider a radial time-dependent potential $V(t,r)$, defined on $\{|x|>R+|t|\}$ for some $R$ and that satisfies 
\begin{equation*}
 r>R+|t|\Longrightarrow |V(t,r)|\leq \frac{C}{r^2}
\end{equation*} 
for some constant $C$. We let 
$v$ be a  radial solution of the equation 
\begin{equation}
 \label{T1'}
\partial_t^2v-\Delta_D v-V v=0,
\end{equation} 
also defined on  $\{|x|>R+|t|\}$ with compactly supported initial data. We let
$$\rho(v_0,v_1)=\inf \left\{\eta>R\;:\; \int_{|x|\geq \eta}\left(|\nabla v_0|^2+v_1^2\right)\,dx=0\right\}.$$
Note that \eqref{T1} is of the form \eqref{T1'} with $V=\frac{N+2}{N-2}W^{\frac{4}{N-2}}$.
\begin{prop}
\label{P:support}
Let $v$ be as above and assume that $\rho_0=\rho(\vec{v}(0))$ is finite and greater than $R$ (i.e. that $\vec{v}(0)$ is compactly supported and not identically zero). Then there exist $\eps>0$ such that for all $\rho\in \big(\rho_0-\eps,\rho_0\big)$, the following holds for all $t\geq 0$ or for all $t\leq 0$:
\begin{equation}
\label{C1}
 \int_{\rho+|t|}^{+\infty} (\partial_{t,r}v(t,r))^2r^{D-1}\,dr\geq \frac{1}{8}\int_{\rho}^{+\infty} (\partial_{t,r}v(0,r))^2r^{D-1}\,dr, 
\end{equation} 
\end{prop}
\begin{remark}
 The proposition implies that the support of $v$ grows at velocity one, in at least one time direction:
 $$ \forall t\geq 0\text{ or } \forall t\leq 0,\quad \rho(\vec{v}(t))=\rho_0+|t|,$$
 however we will need the stronger, qualitative version \eqref{C1} of this fact.
\end{remark}
\begin{proof}
Using finite speed of propagation and a standard density argument, we can assume that $v$ is smooth in the set $\{r>|t|\}$.
The idea is to approximate $v$ by the solution of the equation $(\partial_t^2-\partial_r^2)(r^{\frac{D-1}{2}}v_{\app})=0$ with the same initial data. We thus let, for $r\geq |t|+\rho_0-\eps$.
 \begin{equation}
  \label{def_uapp1}
 v_{\app}(t,r)=\frac{1}{r^{\frac{D-1}{2}}}\left(\varphi(t+r)-\varphi(t-r)\right),
 \end{equation} 
 where $\varphi$ is chosen so that $\vec{v}_{\app}(0,r)=\vec{v}(0,r)$, that is
 $$ \varphi(s)=\frac{1}{2}|s|^{\frac{D-1}{2}} v(0,|s|)\sgn(s)+\frac 12 \int_0^{|s|} \rho^{\frac{D-1}{2}} \partial_tv(0,\rho)\,d\rho.$$
In other words, for $r\geq R+|t|$,
\begin{multline}
\label{def_uapp}
v_{\app}(t,r)=
\frac{1}{2r^{\frac{D-1}{2}}}\\
\times \left( (r+t)^{\frac{D-1}{2}}v(0,t+r)+(r-t)^{\frac{D-1}{2}}v(0,r-t)+\int_{r-t}^{r+t}\rho^{\frac{D-1}{2}} \partial_tv(0,\rho)\,d\rho  \right).
\end{multline}
Then
\begin{equation}
\label{C3}
 (\partial_t^2-\Delta_D)v_{\app}=\left( \frac{D^2-4D+3}{4r^2} \right)v_{\app},
\end{equation} 
and thus
\begin{equation*}
 (\partial_t^2-\Delta_D)(v-v_{\app})=V(v-v_{\app})+\tV v_{\app}, 
\end{equation*}
where
\begin{equation}
\tV:=\frac{D^2-4D+3}{4r^2}+V.
\end{equation}
By Strichartz estimates and finite speed of propagation, for all $T\geq 0$ and all $\rho \in [\rho_0-\eps,\rho_0]$, 
\begin{multline}
\label{propa1}
\sup_{t\in [0,T]} \left\|\indic_{\{|x|\geq \rho+|t|\}}\nabla_{t,x}(v(t)-v_{\app}(t))\right\|_{L^2(\RR^D)}\\
+\left\|\indic_{\{|x|\geq \rho+|t|\}}(v-v_{\app})\right\|_{L^{\frac{2(D+1)}{D-2}}\left( [0,T]\times \RR^D \right)}\\
\leq C\left(\left\|\indic_{\{|x|\geq \rho+|t|\}}\tV v_{\app}\right\|_{L^1([0,T],L^2)}+\left\| \indic_{\{r \geq \rho+|t|\}}V(v-v_{\app})\right\|_{L^1\left([0,T],L^2\right)}\right).
\end{multline}
Let 
$$\beta(\eps)=\left\| \indic_{\{\rho_0-\eps+|t|\leq r\leq \rho_0+|t|\}} (|V|+|\tV|)\right\|_{L^{\frac{2(D+1)}{D+4}}\left( (0,\infty),L^{\frac{2(D+1)}{3}} \right)},$$
By H\"older inequality and finite speed of propagation, the right-hand side of \eqref{propa1} is bounded by 
$$\beta(\eps)\left(\left\|\indic_{\{r\geq \rho+|t|\}}(v-v_{\app})\right\|_{L^{\frac{2(D+1)}{D-2}}}+
\left\|\indic_{\{r\geq \rho+|t|\}}v_{\app}\right\|_{L^{\frac{2(D+1)}{D-2}}}\right).$$
We claim that $\beta(\eps)$ is finite and that 
\begin{equation}
 \label{beta_0}
\lim_{\eps \to 0}\beta(\eps)=0.
\end{equation} 
Indeed, by direct computation, using that $\rho_0-\eps\leq \rho\leq \rho_0$ and $|V(r)|+|\tV(r)|\lesssim r^{-2}$,
\begin{multline*}
\beta(\eps)^{\frac{2(D+1)}{D+4}}
 \lesssim \int_0^{+\infty} \left( \int_{\rho_0-\eps+|t|}^{\rho_0+|t|}\left( \frac{1}{r^2} \right)^{\frac{2(D+1)}{3}}r^{D-1}\,dr \right)^\frac{3}{D+4}\,dt\\
 \lesssim \int_0^{\infty}\left(\rho_0+t\right)^{-1}\left( \left( \frac{\rho_0-\eps+t}{\rho_0+t} \right)^{-\frac{D+4}{3}}-1\right)^{\frac{3}{D+4}}\,dt.
\end{multline*}
Noting that there exists a constant $C>0$, depending only on $D$, such that for all $\eps\leq \rho_0/2$,
$$\left|\left( \frac{\rho_0-\eps+|t|}{\rho_0+|t|}\right)^{-\frac{D+4}{3}}-1 \right|= \left|\left( 1-\frac{\eps}{\rho_0+|t|} \right)^{-\frac{D+4}{3}}-1\right|\leq C\frac{\eps}{\rho_0+|t|},$$
we deduce 
$$ \beta(\eps)\lesssim \eps^{\frac{3}{2(D+1)}},$$
(the implicit constant depending only on $D$ and $\rho_0$), which proves \eqref{beta_0}.

Choosing $\eps$ small, we deduce from \eqref{propa1}
\begin{multline}
\label{propa2}
\sup_{t\in [0,T]} \left\|\indic_{\{|x|\geq \rho+|t|\}}\nabla_{t,x}(v(t)-v_{\app}(t))\right\|_{L^2}\\
+\left\|\indic_{\{|x|\geq \rho+|t|\}}(v-v_{\app})\right\|_{L^{\frac{2(D+1)}{D-2}}\left( [0,T]\times \RR^D \right)}
\\
\leq C\beta(\eps)\left\|\indic_{\{|x|\geq \rho+|t|\}}v_{\app}\right\|_{L^{\frac{2(D+1)}{D-2}}([0,T]\times \RR^D)}. 
\end{multline} 
Using Strichartz and H\"older inequalities in a similar way, we obtain, in view of the equation \eqref{C3} on $v_{\app}$ that for all $T>0$,
\begin{multline}
 \label{C5}
\left\|v_{\app}\indic_{\{r\geq \rho+|t|\}}\right\|_{L^{\frac{2(D+1)}{D-2}}([0,T]\times \RR^D)} +\sup_{t\in [0,T]} \left\|\indic_{\{|x|\geq \rho+|t|\}}\nabla_{t,x}v_{\app}(t)\right\|_{L^2}\\
\leq C\left\|(\nabla v_0,v_1)\indic_{\{r\geq \rho\}}\right\|_{L^2}.
\end{multline} 
Combining \eqref{propa2} and \eqref{C5}, we obtain that for all $t\geq 0$,
\begin{multline}
 \label{C6}
\left\|\indic_{\{|x|\geq \rho+|t|\}}\nabla_{t,x}v(t)\right\|_{L^2} \\ 
\geq  
\left\|\indic_{\{|x|\geq \rho+|t|\}}\nabla_{t,x}v_{\app}(t)\right\|_{L^2}-C\beta(\eps) \left( \|(\nabla v_0,v_1) \indic_{\{r\geq \rho+|t|\}}\|_{L^2} \right).
\end{multline} 
The same proof works for negative times, showing that \eqref{C6} holds for all $t\in \RR$. To conclude the proof of \eqref{C1}, we must show that the following holds for all $t\geq 0$ or for all $t\leq 0$:
\begin{equation}
\label{C7}
\int_{\rho+|t|}^{+\infty} |\partial_{t,r}v_{\app}(t,r)|^2r^{D-1}\,dr\geq \frac{1}{4}\int_{\rho}^{+\infty}\left((\partial_rv_0)^2+v_1^2\right)r^{D-1}\,dr.
\end{equation} 
By the definition \eqref{def_uapp1} of $v_{\app}$, and since $\vec{v}_{\app}(0,r)=(v_0,v_1)(r)$,
\begin{equation}
\label{C8}
\int_{\rho+|t|}^{+\infty} \left|\partial_{t,r}\left(r^{\frac{D-1}{2}}v_{\app}(t,r)\right)\right|^2\,dr\geq \frac{1}{2}\int_{\rho}^{+\infty}\left(\partial_r\left(r^{\frac{D-1}{2}}v_0\right)\right)^2+\left(r^{\frac{D-1}{2}}v_1\right)^2\,dr.
\end{equation} 
Furthermore, $\partial_{r}\left(r^{\frac{D-1}{2}}v_{\app}(t,r)\right)=\frac{D-1}{2}r^{\frac{D-3}{2}}v_{\app}+r^{\frac{D-1}{2}}\partial_rv_{\app}$, and thus there exists a constant $C>0$ (depending only on $D$) such that for all $t\in \RR$,
\begin{multline*}
 0.9 \int_{\rho+|t|}^{+\infty} (\partial_r v_{\app}(t,r))^2r^{D-1}dr-C\int_{\rho+|t|}^{+\infty}v_{\app}^2r^{D-3}\,dr\\
 \leq\int_{\rho+|t|}^{+\infty} \left|\partial_r\left(r^{\frac{D-1}{2}}v_{\app}(t,r) \right)\right|^2\,dr\\
 \leq 1.1 \int_{\rho+|t|}^{+\infty} (\partial_r v_{\app}(t,r))^2r^{D-1}dr+C\int_{\rho+|t|}^{+\infty}v_{\app}^2r^{D-3}\,dr.
\end{multline*}
Furthermore, since $v_{\app}(t,r)=0$ for $r\geq \rho_0+|t|$, one has, for $\rho+|t|\leq r\leq \rho_0+|t|$, 
\begin{align*}
\left|r^{\frac{D-1}{2}}v_{\app}(t,r)\right|&=\left|\int_r^{\rho_0+|t|} \frac{\partial}{\partial s}\left( s^{\frac{D-1}{2}} v_{\app}(t,s) \right)\,ds\right|\\
&\leq \sqrt{\eps}\sqrt{\int_{\rho+|t|}^{+\infty} \left|\partial_s\left( s^{\frac{D-1}{2}} v_{\app}(t,s) \right)\right|^2\,ds}.
\end{align*}
Hence
\begin{multline*}
\int_{\rho+|t|}^{+\infty}r^{D-3}v_{\app}^2(t,r)\,dr=\int_{\rho+|t|}^{\rho_0+|t|}r^{D-3}v_{\app}^2(t,r)\,dr\\
\leq \frac{\eps^2}{\rho}\int_{\rho+|t|}^{+\infty} \left|\partial_r\left( r^{\frac{D-1}{2}}v_{\app}(t,r) \right)\right|^2\,dr, 
\end{multline*}
which yields, for $\eps$ small enough,
\begin{multline*}
\frac{1}{2}\int_{\rho+|t|}^{+\infty} |\partial_r v_{\app}(t,r)|^2r^{D-1}\,dr\leq \int_{\rho+|t|}^{+\infty} \left|\partial_r\left(r^{\frac{D-1}{2}}v_{\app}(t,r) \right)\right|^2\,dr\\
\leq 2\int_{\rho+|t|}^{+\infty} |\partial_r v_{\app}(t,r)|^2r^{D-1}\,dr.
\end{multline*}
Combining with \eqref{C8}, we obtain \eqref{C7}, which concludes the proof.
\end{proof}

\subsection{Conclusion of the proof}
\label{SS:conclusion}
We now conclude the proof of Proposition \ref{P:T2} (and thus of Theorem \ref{T:LW''}).

Let $v$ be a solution as in Proposition \ref{P:T2}. By Proposition \ref{P:cpct}, there exist real numbers $\zeta_1,\ldots,\zeta_{\ent{\frac{D+2}{4}}},\eta_1,\ldots,\eta_{\ent{\frac{D}{4}}}$ such that the essential support of 
\begin{equation*}
(w_0,w_1):=\vec{v}(0)-\sum_{k=1}^{\ent{\frac{D+2}{4}}} \zeta_k (Z_k,0)-\sum_{k=1}^{\ent{\frac{D}{4}}}\eta_k(0,Z_k)
\end{equation*} 
is compact. We will prove that $(w_0,w_1)\equiv 0$ almost everywhere by contradiction. If not, we let $R>0$ be such that 
$$\int_{R}^{+\infty} (\partial_rw_0)^2+w_1^2dr>0,$$
and consider the solution $w$ of
\begin{equation*}
\partial_t^2w-\Delta_D w-\frac{N+2}{N-2}W^{\frac{4}{N-2}} w=0,\quad \vec{v}_{\restriction t=0}=(w_0,w_1),
\end{equation*}
which is well defined for $r>R+|t|$ and satisfies
\begin{equation}
 \label{no_channel_w}
 \sum_{\pm}\lim_{t\to\pm\infty}\int_{R+|t|}^{+\infty} (\partial_{t,r}w(t,r))^2r^{D-1}\,dr=0.
\end{equation} 
Letting $\rho_0=\rho(w_0,w_1)>R$, and $\rho<\rho_0$ close to $\rho_0$, we obtain by Proposition \ref{P:support}:
\begin{equation*}
 \int_{\rho+|t|}^{+\infty} (\partial_{t,r}w(t,r))^2r^{D-1}\,dr\geq \frac{1}{8}\int_{\rho}^{+\infty} (\partial_{t,r}w(0,r))^2r^{D-1}\,dr>0, 
\end{equation*} 
which contradicts \eqref{no_channel_w}. As a conclusion, $(w_0,w_1)(r)=0$ for almost every $r>0$, and thus, denoting by $(v_0,v_1)=\vec{v}(0)$,
$$ (v_0,v_1)=\sum_{k_0=1}^{\ent{\frac{D+2}{4}}} \zeta_{k_0} (Z_{k_0},0)+\sum_{k_1=1}^{\ent{\frac{D}{4}}}\eta_{k_1}(0,Z_{k_1})\in \HHH.$$
We will conclude the proof with Proposition \ref{P:Z1}.
If $D\geq N+4$, we prove by contradiction that $\zeta_{k_0}=0$, $\eta_{k_1}=0$ for all $1\leq k_0\leq \ent{\frac{D+2}{4}}$ and $1\leq k_1\leq \ent{\frac{D}{4}}$. Assume that one of the $\zeta_{k_0}$ is nonzero. Then by Proposition \ref{P:Z1},
$$v_0=\sum_{k_0=1}^{\ent{\frac{D+2}{4}}} \zeta_{k_0} Z_{k_0}\sim \frac{d_0}{r^{D-\theta_0}}, \quad r\to 0$$
for some real number $d_0\neq 0$ and integer $\theta_0$ such that $\theta_0\leq 2\ent{\frac{D+2}{4}}$. As a consequence, $\int_0^1 \frac{1}{r^2}v_0^2(r)\,r^{D-1}dr=\infty$, a contradiction with Hardy's inequality and the fact that $v_0\in \hdot(\RR^D)$. Similarly, if there exists $k_1$ such that $\eta_{k_1}\neq 0$, then by Proposition \ref{P:Z1}
$$v_1=\sum_{k_1=1}^{\ent{\frac{D}{4}}} \eta_{k_1} Z_{k_1}\sim \frac{d_1}{r^{D-\theta_1}}, \quad r\to 0$$
for some real number $d_1\neq 0$ and integer $\theta_1$ with $\theta_1\leq 2\ent{\frac{D}{4}}$, and one checks that $v_1\notin L^2(\RR^D)$, a contradiction again. 

It remains to treat the cases when $D=N$ and $D=N+2$. Assume $D=N\geq 5$. We note that in this case $Z_1\in H^1(\RR^D)$. Thus 
$$\sum_{k_0=2}^{\ent{\frac{D+2}{4}}} \zeta_{k_0} (Z_{k_0},0)+\sum_{k_1=2}^{\ent{\frac{D}{4}}}\eta_{k_1}(0,Z_{k_1})\in \HHH,$$
and the same argument as before proves that $\zeta_{k_0}=0$, $\eta_{k_1}=0$ for $2\leq k_0\leq \ent{\frac{D+2}{4}}$, $2\leq k_1\leq \ent{\frac{D}{4}}$. Thus 
$$ (v_0,v_1)\in \vect\Big((Z_1,0),(0,Z_1)\Big)=\vect\Big((\Lambda W,0),(0,\Lambda W)\Big)$$
(see \eqref{Z6a}). The cases $D=N=3$ and $D=N+2$ are similar and we omit the proofs. \qed

\section{Channels of energy for the nonlinear equation close to $W$}
\label{S:NL_channels}
In this section we prove Corollary \ref{Cor:rigidity_intro}. We start by giving two statements, one in space dimension $3$ and the other in space dimension $5$, that imply Corollary \ref{Cor:rigidity_intro}.
Let $\MMM_W$ be the set of all $$
\left(\frac{1}{\lambda^{\frac{N-2}{2}}}W_{\vell}\left( 0,\frac{\cdot-X}{\lambda}\right),\frac{1}{\lambda^{\frac{N}{2}}}\partial_t W_{\vell}\left( 0,\frac{\cdot-X}{\lambda}\right) \right),$$
where
$$X\in \RR^N,\;\lambda>0,\vell\in \RR^N,\; |\vell|<1.$$
We denote by $d$ the distance between a subset $A$ of $\HHH$ and an element of $\HHH$:
\begin{equation}
 \label{def_dw}
 d\left(A,(u_0,u_1)\right)=\inf\Big\{ \left\|(u_0,u_1)-(\varphi_0,\varphi_1)\right\|_{\HHH},\; (\varphi_0,\varphi_1)\in A\Big\}.
\end{equation} 
\begin{theoint}
\label{T:NL3}
Assume $N=3$. Let $\eta_0\in (0,1)$. There exist constants $\eps_0=\eps_0(\eta_0)>0$, $C=C(\eta_0)>0$ such that, for any $\vell\in \RR^3$ with $|\vell|\leq \eta_0$, for all $(u_0,u_1)\in \HHH\setminus \MMM_W$ such that 
\begin{equation}
 \label{dist_u_W}
 \left\|(u_0,u_1)-\vec{W}_{\vell}(0)\right\|_{\HHH}=\eps\leq \eps_0,
\end{equation} 
there exists $X\in \RR^3$ with $|X|\lesssim \eps$ such that the solution $u$ of \eqref{NLW} on $\{|x-X|>|t|\}$ is global in time and satisfies
\begin{equation}
  \label{NL0}
  \sum_{\pm}\lim_{t\to \pm\infty} \int_{|x-X|\geq |t|}\left|\nabla_{t,x}u(t,x)\right|^2\,dx\geq \frac{1}{C}d\left(\MMM_W,(u_0,u_1)\right)^2. 
 \end{equation} 
\end{theoint} 
In the case $N=5$, we have a weaker statement
\begin{theoint}
\label{T:NL5}
Assume $N=5$. Let $\eta_0\in (0,1)$. There exists a constant $\eps_0=\eps_0(\eta_0)>0$ such that, for any $\vell\in \RR^5$ with $|\vell|\leq \eta_0$, for all $(u_0,u_1)\in \HHH\setminus \MMM_W$ such that \eqref{dist_u_W} holds, then there exists $X\in \RR^5$ with $|X|\lesssim \eps$ such that for any small $\tau_0>0$, the solution $u$ of \eqref{NLW} for $\{|x-X|>|t|-\tau_0\}$ is global in time and satisfies:
\begin{gather}
\label{NL0_5}
  \sum_{\pm}\lim_{t\to \pm\infty} \int_{|x-X|\geq |t|-\tau_0}\left|\nabla_{t,x}u(t,x)\right|^2\,dx>0. 
 \end{gather} 
\end{theoint} 
Recall from \eqref{defZZZZ} the definition of $\ZZZZ$. Let
$$\widetilde{\ZZZZ}=\vect\Big\{(\Lambda W,0),\; (\partial_{x_j}W,0),\; (0,\partial_{x_j}W),\; j\in \llbracket 1,N\rrbracket \Big\}.$$
Note that $\widetilde{\ZZZZ}\subset \ZZZZ$ (indeed these sets are equal if $N\in \{3,4\}$, the inclusion is strict if $N\geq 5$). Thus if 
$(v_0,v_1)\in \widetilde{\ZZZZ}$, the corresponding solution of the linearized wave equation \eqref{LW} is $v_0+tv_1$. We denote by:
$$\widetilde{\ZZZZ}_{\vell}=\left\{\left(\LLL_{\ell}(v_0(x)+tv_1(x)),\partial_t \LLL_{\ell}(v_0(x)+tv_1(x))\right)_{\restriction t=0},\; (v_0,v_1)\in \widetilde{\ZZZZ}\right\}.$$
Note again that 
$$\widetilde{\ZZZZ}_{\vell}=\ZZZZ_{\vell}\iff N\in \{3,4\}.$$
We next give preliminaries on local and global well-posedness (Subsection \ref{SS:WP}) and the choice of the parameters $\ell$, $\lambda$ and $X$ so that certain orthogonality conditions are satisfied (Subsection \ref{SS:ortho}). In Subsection \ref{SS:NL3} we prove Theorem \ref{T:NL3} and 
in Subsection \ref{SS:NL5} we prove Theorem \ref{T:NL5}.  Subsection \ref{SS:rig2} is dedicated to the proof of Corollary \ref{Cor:rigidity_intro2}.
\subsection{Cauchy theory outside wave cones}
\label{SS:WP}
In this subsection we assume $N\in \{3,4,5\}$ and recall some standard facts about local and global well-posedness for the nonlinear equation \eqref{NLW}. We also extend the Cauchy theory to define  solutions of \eqref{NLW} outside wave cones. We refer to \cite{DuKeMe19Pc} for the extension outside wave cones for radial data in the case $N\geq 6$.

\begin{defi}
\label{D:solution}
 Let $I$ be an interval with $t_0\in I$, $(u_0,u_1)\in \HHH$. We call a solution of \eqref{NLW} on $I\times \RR^N$, with initial data
\begin{equation}
 \label{S0}
\vec{u}_{\restriction t=t_0}=(u_0,u_1)
\end{equation} 
a function $u\in C^0(I,\hdot)$ such that $\partial_t u\in C^0(I,L^2)$ and
\begin{equation}
 \label{S1}
\forall t\in I,\quad u(t)=S_L(t-t_0)(u_0,u_1)+\int_{t_0}^tS_L(s-t_0)F(u(s))\,ds,
\end{equation} 
and such that $u\in L^{\frac{N+2}{N-2}}\left(J,L^{\frac{2(N+2)}{N-2}}(\RR^N)\right)$ for all compact intervals $J\subset I$.
\end{defi}
\begin{remark}
\label{R:solution}
 The distribution $u$ on $I\times \RR^N$ is a solution of \eqref{NLW}, \eqref{S0} if and only if $u\in L^{\cinq}_{\loc}(I,L^{\dix})$, $\vec{u}\in C^0(I,\HHH)$, \eqref{S0} holds and $u$ satisfies 
$\partial_t^2u-\Delta u=|u|^{\frac{4}{N-2}}u$ in the distributional sense on $I\times \RR^N$ (see Lemma 2.5 of \cite{DuKeMe16a}).
\end{remark}
\begin{remark}
\label{R:sol_limit}
 If $\vec{u}\in C^0(I,\HHH)$ and there exists a sequence $\{\vec{u}_k\}_k$ of solutions of \eqref{NLW}
 such that
$$\lim_{k\to\infty} \sup_{t\in I}\left\|\vec{u}(t)-\vec{u}_k(t)\right\|_{\HHH}=0$$
and 
$$\sup_k\|u_k\|_{L^{\frac{2(N+1)}{N-2}}\big(I\times \RR^N\big)}<\infty,$$
then $u$ is a solution of \eqref{NLW} (see \cite[Remark 2.14]{KeMe08}). 
\end{remark}

It is known (see \cite{GiVe95}, \cite{KeMe08}, \cite{BuCzLiPaZh13} and \cite{DuKeMe16a}), that for all initial data $(u_0,u_1)$, there is a unique maximal solution $u$ defined on a maximal interval $(T_-,T_+)$. Furthermore the following blow-up criterion is satisfied:
$$T_+<\infty\Longrightarrow \|u\|_{L^{\frac{2(N+1)}{N-2}}([t_0,T_+)\times \RR^N)}=\infty,$$
or equivalently 
$$T_+<\infty\Longrightarrow \|u\|_{L^{\frac{N+2}{N-2}}\left([t_0,T_+),L^{\frac{2(N+2)}{N-2}}(\RR^N)\right) }=\infty.$$

We will also need the notion of a solution of \eqref{NLW} outside a wave cone:
\begin{defi}
\label{D:sol_cone}
 Let $I$ be an interval with $t_0\in I$, $R\geq 0$ and $x_0\in \RR^N$. Let $(u_0,u_1)\in \HHH$. A solution $u$ of \eqref{NLW} on $\{|x-x_0|>[t-t_0|,\; t_0\in I\}$ with initial data $(u_0,u_1)$ at $t=t_0$ is the restriction to $\{|x-x_0|>[t-t_0|,\; t_0\in I\} $ of a solution $\tilde{u}$, with $\vec{\tu}\in C^0(I,\HHH)$  to the equation:
 \begin{equation}
  \label{NLW_trunc}
\partial_t^2\tu-\Delta \tu=|\tu|^{\frac{4}{N-2}}\tu\indic_{\{|x-x_0|>R+|t-t_0|\}},
\end{equation} 
with an initial data
\begin{equation}
 \label{ID_trunc}
 \vec{\tu}_{\restriction t=t_0}=(u_0,u_1),
\end{equation}
\end{defi}
The Cauchy theory in \cite{GiVe95}, \cite{KeMe08}, \cite{DuKeMe16a} adapts easily to the case of solutions outside wave cones. We give the statements and omit most proofs.
\begin{prop}[Small data well-posedness]
 \label{P:Small_data_cones}
 There exists $\eps_0>0$ with the following property.
 Let $(u_0,u_1)$, $t_0$, $x_0$, $R$ be as above. Let $u_L(t)=S_L(t)(u_0,u_1)$. Assume   
 $$\left\|\indic_{\{|x-x_0|\geq |t-t_0|+R\}} u_L\right\|_{L^{\frac{N+2}{N-2}}\big(I,L^{\frac{2(N+2)}{N-2}}\big)}=\eps\leq \eps_0.$$
 Then there exists a unique solution $u$ to \eqref{NLW} on $\{|x-x_0|>|t-t_0|,\;t\in I\}$. Furthermore 
 $$\sup_{t\in I}\left\|\indic_{\{|x-x_0|>|t-t_0|+R\}}\nabla_{t,x}(u-u_L)\right\|_{L^2}\lesssim \eps^{\frac{N+2}{N-2}}.$$
\end{prop}
Gluing the preceding local solutions, we obtain, for any initial data $(u_0,u_1)$, a maximal solution defined on a maximal domain $\{(t,x)\in (\widetilde{T}_-,\widetilde{T}_+)\times \RR^N\;\:\; |x-x_0|\geq |t-t_0|+R\}$, where $\widetilde{T}_{\pm}=\widetilde{T}_{\pm}(x_0,R)$, and that satisfies the following blow-up criterion:
$$\widetilde{T}_{+}<\infty\Longrightarrow \|\indic_{\{|x-x_0|>|t-t_0|\}}u\|_{L^{\frac{N+2}{N-2}}\left((0,\widetilde{T}_+),L^{\frac{2(N+2)}{N-2}}\right)}=+\infty
$$
(and similarly in the past). It is easy to check that if $R>R'$ then $\widetilde{T}_+(x_0,R)\geq \widetilde{T}_+(x_0,R')$, and also that for all $R\geq 0$, 
$\widetilde{T}_+(x_0,R)\geq T_+$, where $T_+$ is the maximal time of existence for the solution on the whole space (see above). We also have:
\begin{lemma}
\label{L:scattering}
If $\indic_{\{|x-x_0|\geq |t-t_0|+R\}}u\in L^{\frac{N+2}{N-2}}\left((0,\widetilde{T}_+),L^{\frac{2(N+2)}{N-2}}(\RR^N)\right)$, then $\widetilde{T}_+=+\infty$ and $u$ scatters to a linear solution on $\{|x-x_0|>R+|t-t_0|\}$, in the sense that there exists a solution $v_L$ of the linear wave equation on $\RR\times \RR^N$ such that 
$$\lim_{t\to +\infty} \left\|\indic_{\{|x-x_0|>|t-t_0|+R\}} |\nabla_{t,x}(u-v_L)(t)|\right\|_{L^2}=0.$$
\end{lemma}

\begin{lemma}
Let $I$ be an interval with $t_0\in I$, 
and $f\in L^1(I,L^2)$. Let $u$ be defined by \eqref{Duhamel} for $t\in I$,  and assume that 
 $$f=|u|^{\frac{4}{N-2}}u \text{ a.e. for } |x-x_0|>|t-t_0|,\; t\in I.$$
 Then $I\subset (\widetilde{T}_-\widetilde{T}_+)$ and the restriction of $u$ to the set $\{|x-x_0|>|t-t_0|+R,\;t\in I\}$ coincides with the solution of \eqref{NLW} on this set, as defined in Definition \ref{D:sol_cone}.
\end{lemma}

\subsection{Orthogonality conditions}
\label{SS:ortho}
If $\Theta=(X,\lambda,\vell)\in \RR^N\times (0,\infty)\times B_{\RR^N}$, where $B_{\RR^N}=\left\{X\in \RR^3\;:\; |X|<1\right\}$, we denote
 $$W_{\Theta}(t,x)= \frac{1}{\lambda^{\frac{N}{2}-1}}W_{\vell}\left( \frac{t}{\lambda},\frac{x-X}{\lambda} \right).$$
We claim:
 \begin{claim}
 \label{Cl:modulation}
Assume $N\geq 3$. There exists a small constant $\eps_0$, and, for all $\eta_0\in (0,1)$, a constant $\eta_1\in (0,1)$ with the following property.
 Let $(u_0,u_1)$ be such that \eqref{dist_u_W} holds for some $\vell_0$ with $|\vell_0|\leq \eta_0$.
 Then there exists $\Theta_*=(X_*,\lambda_*,\vell_*)\in \RR^N\times (0,\infty)\times B_{\RR^N}$, with $|\vell_*|<\eta_1$, such that 
 $$ d\left((u_0,u_1),\MMM_W\right)=\left\|(u_0,u_1)-\vec{W}_{\Theta_*}(0)\right\|_{\HHH}$$
 and
 \begin{equation}
 \label{good_ortho}
 \left(\lambda_*^{\frac{N-2}{2}}u_0\left(X_*+\lambda_*\cdot \right),\lambda_*^{\frac{N}{2}}u_1\left(X_*+\lambda_*\cdot  \right)\right)-(\vec{W}_{\vell_*},0) \in \widetilde{\ZZZZ}_{\vell_*}^{\bot}.
    \end{equation} 
  \end{claim}
 \begin{proof}
 Let 
  $$\delta_W:\Theta\mapsto \left\|\vec{W}_{\Theta}(0)-(u_0,u_1)\right\|_{\HHH}^2.$$
  We first claim that there exists a constant $\eta_1\in (0,1)$, depending only on $\eta_0$, such that 
  \begin{equation}
   \label{def_eta1}
   |\vell|\geq \eta_1\Longrightarrow \delta_W(\Theta)\geq \|W\|^2_{\hdot}.
  \end{equation} 
Indeed, by the triangle inequality and \eqref{dist_u_W},
$$\sqrt{\delta_W(\Theta)}\geq \left\|\vec{W}_{\Theta}(0)\right\|_{\HHH}-\left\|\vec{W}_{\vell_0}\right\|_{\HHH}-\eps_0.$$
Since 
$$\left\|\vec{W}_{\Theta}(0)\right\|_{\HHH}=\left(  \frac{N+(2-N)|\vell|^2}{N\sqrt{1-|\vell|^2}}\right)^{1/2}\|W\|_{\hdot}\underset{|\vell|\to 1}{\longrightarrow}+\infty,$$
we obtain the existence of $\eta_1$ such that \eqref{def_eta1} holds.
  
 We next claim that there exists $M>0$ (depending on $(u_0,u_1)$) such that
 \begin{equation}
  \label{NL1}
  |X|+|\log(\lambda)|\geq M\Longrightarrow \delta_W(\Theta)\geq \frac{1}{N} \|W\|^2_{\dot{H}^1}.
 \end{equation}
 
 If not, there exists a sequence $(\Theta_n)_n=\Big( (X_n,\lambda_n,\vell_n) \Big)_n$, such that
 \begin{equation*}
  \lim_{n\to\infty}|X_n|+|\log(\lambda_n)|=+\infty
 \end{equation*}
 and for all $n$, 
 $$\delta_W(\Theta_n)\leq \frac{1}{N}\|W\|^2_{\hdot}.$$
 By \eqref{def_eta1} we can assume $|\vell_n|\leq \eta_1$ and thus, extracting subsequences if necessary, that there exists $\vell$ such that $|\vell|\leq \eta_1$ and 
 $$\lim_{n\to\infty}\vell_n=\vell.$$
Since $\vec{W}_{\Theta_n}(0)$ converges weakly to $(0,0)$ in $\HHH$, we obtain
$$\lim_{n\to\infty} \delta_W(\Theta_n)=\left\|(u_0,u_1)\right\|^2_{\HHH}+\left\|\vec{W}_{\vell}(0)\right\|^2_{\HHH}.$$
Using
$$ \left\|\vec{W}_{\vell}(0)\right\|^2_{\HHH}=\frac{N+(2-N)|\vell|^2}{N\sqrt{1-|\vell|^2}}\|W\|_{\hdot}^2 >   \frac 1N \|W\|^2_{\hdot},$$
we deduce a contradiction,
concluding the proof of \eqref{NL1}.
 
 In view of \eqref{NL1} and since by the assumptions of the claim there exists $\Theta$ such that $\delta_W(\Theta)\leq \eps_0$, we see (taking $\eps_0<\frac{1}{N}\|W\|^2_{\hdot}$) that there exists $\Theta_{*}$ such that $\delta_W$ attains a global minimum at $\Theta_{*}$. 
 At $\Theta=\Theta_{*}$, the function $\delta_W$ has a critical point. Differentiating, we obtain that $\vec{W}_{\Theta_*}(0)-(u_0,u_1)$
 is orthogonal, in $\HHH$, to
 $$\left(\frac{\partial}{\partial X_j} W_{\Theta}(0)\right)_{\restriction\Theta=\Theta_{*}},\quad \left(\frac{\partial}{\partial \lambda} W_{\Theta}(0)\right)_{\restriction\Theta=\Theta_{*}},\quad \left(\frac{\partial}{\partial \ell_j} W_{\Theta}(0) \right)_{\restriction\Theta=\Theta_{*}}.$$
Rescaling and translating in space, we obtain \eqref{good_ortho}.
 \end{proof}
\subsection{Proof in space dimension $3$}
\label{SS:NL3}
\begin{proof}[Proof of Theorem \ref{T:NL3}]
In all of the proof $\eta_0$ is fixed in $(0,1)$, and the constant $C>0$ (that may change from line to line) is allowed to depend on $\eta_0$.

We use Claim \ref{Cl:modulation}. Without loss of generality (rescaling and translating $u$), we can assume $\lambda_*=1$, $X_*=0$. We let $\vell=\vell_*$ to lighten notations. Let 
$$(h_0,h_1)=(u_0,u_1)-\vec{W}_{\vell}(0).$$

Let $h$ be the solution of 
 \begin{equation*}
 \left\{
 \begin{aligned}
 \partial_t^2h-\Delta h&=\left( (W_{\vell}+h)^5-W_{\vell}^5 \right)\indic_{|x|\geq |t|}\\
 h_{\restriction t=0}&=(h_0,h_1).
\end{aligned}\right.  
 \end{equation*} 
 Then
 \begin{equation*} 
 \left|\partial_t^2h-\Delta h-5\indic_{|x|\geq |t|}W_{\vell}^4h\right|\leq C\left( |h|^5+W_{\vell}^3h^2 \right)\indic_{|x|\geq |t|}.  
 \end{equation*} 
 Since by explicit computation, 
 $\left\|\indic_{\{|x|\geq |t|\}}W_{\vell}\right\|_{L^5_tL^{10}_x}\leq C,$
 we deduce, using Lemma \ref{L:linear_approx} and the smallness of $\|(h_0,h_1)\|_{\HHH}$, that for all $T>0$,
\begin{equation}
 \label{NL4}
  \sup_{|t|\leq T}\left\|\vec{h}(t)\right\|_{\HHH}+\|h\|_{L^8([-T,T],L^8_x)}+\|h\|_{L^5([-T,+T],L^{10}_x)}\leq C\|(h_0,h_1)\|_{\HHH}.
\end{equation} 
Let $h_{L}$ be the solution of the linearized equation:
\begin{equation}
 \label{NL5}
 \left\{
 \begin{aligned}
 \partial_t^2h_L-\Delta h_L&=5W_{\vell}^4 h_L\\
 h_{L\restriction t=0}&=(h_0,h_1),
\end{aligned}\right.   
\end{equation} 
and $\tlh_{L}$ be the solution of the truncated linearized equation:
\begin{equation}
 \label{NL6}
 \left\{
 \begin{aligned}
 \partial_t^2\tlh_L-\Delta \tlh_L&=5W^4_{\vell} \tlh_L\indic_{|x|\geq |t|}\\
 \tlh_{L\restriction t=0}&=(h_0,h_1).
\end{aligned}\right.   
\end{equation} 
Then by Theorem \ref{T:LW} and \eqref{good_ortho},
$$\sum_{\pm}\lim_{t\to\pm \infty} \int_{|x|>|t|} \left|\nabla_{t,x}h_{L}(t,x)\right\|^2\,dx\geq \frac 1C \|(h_0,h_1)\|^2_{\HHH}.$$
Since $\tlh_L(t,x)=h_L(t,x)$ for $|x|>|t|$ by finite speed of propagation, we deduce
$$\sum_{\pm}\lim_{t\to\pm \infty} \int_{|x|>|t|} \left|\nabla_{t,x}\tlh_{L}(t,x)\right|^2\,dx\geq \frac 1C \|(h_0,h_1)\|^2_{\HHH}.$$
Next, notice that 
$$\left|\partial_t^2(h-\tlh_L)-\Delta(h-\tlh_L)-5W^4_{\vell}(h-\tlh_L)\indic_{\{|x|>|t|\}}\right|\leq C\left( |h|^5+W_{\vell}^3h^2 \right)\indic_{|x|\geq |t|}.$$
Using Lemma \ref{L:linear_approx} as before, and since $(\vec{h}-\vec{\tilde{h}}_L)(0)=0$, we obtain in view of \eqref{NL4},
$$\sup_{t\in \RR} \left\|\vec{h}(t)-\vec{\tlh}_L(t)\right\|_{\HHH}\leq C\left\|(h_0,h_1)\right\|^2_{\HHH}.$$
Hence, taking $\eps_0$ smaller if necessary,
\begin{equation}
 \label{NL8}
 \sum_{\pm}\lim_{t\to\pm \infty} \int_{|x|>|t|} \left|\nabla_{t,x}h(t,x)\right|^2\,dx\geq \frac 1C \|(h_0,h_1)\|^2_{\HHH}.
\end{equation} 
Letting $\tilde{u}=W_{\vell}+h$, we see that \eqref{NL8} implies \eqref{NL0}, since by the choice of $(h_0,h_1)$,
$$\|(h_0,h_1)\|_{\HHH}=d\left(\MMM_W,(u_0,u_1)\right).$$
Furthermore, $\tu$ satisfies the equation:
\begin{align*}
\partial_t^2\tu-\Delta\tu&=W_{\vell}^5+\left((W_{\vell}+h)^5-W_{\vell}^5 \right)\indic_{\{|x|>|t|\}}\\
&=\tilde{u}^5\indic_{\{|x|>|t|\}}+\tilde{u}W_{\vell}^5\indic_{\{|x|<|t|\}},
\end{align*}
which concludes the proof. Recall that  we have translated $u$ in space in order to assume $X_*=0$ in the conclusion of Claim \ref{Cl:modulation}, hence the necessity of the parameter $X$ in the conclusion of the Theorem (indeed, one can take $X=X_{*}$).
\end{proof}
\subsection{Proof in space dimension $5$}
\label{SS:NL5}
This subsection is dedicated to the proof of Theorem \ref{T:NL5}. The main difficulty compared to the case $N=3$ is the existence of an additional component $(0,\Lambda W)$ in $\ZZZZ$, so that $\widetilde{\ZZZZ}$ is a strict subspace of $\ZZZZ$. To deal with this additional direction, we apply Theorem \ref{T:LW} not only at $t=0$, but for all initial times close to $t=0$. As a drawback, Theorem \ref{T:NL5} is slightly weaker than Theorem \ref{T:NL3}.

We first prove:
\begin{lemma}
\label{L:PP1}
Assume $N=5$. There exists $\eps>0$ with the following property.
Let $u,v$ be two global solutions of \eqref{NLW} on $\{|x|>|t|\}$ with initial data $(u_0,u_1)$ and $(v_0,v_1)$ respectively, and such that 
\begin{gather}
 \label{L:PP1_H1} \|(u_0,u_1)-(W,0)\|_{\HHH}\leq \eps,\quad \|(v_0,v_1)-(W,0)\|_{\HHH}\leq \eps\\
 \label{L:PP1_H2} \sum_{\pm} \lim_{t\to \pm\infty}  \left(\int_{\{|x|\geq |t|\}} |\nabla_{t,x}u|^2\,dx+\int_{\{|x|\geq |t|\}} |\nabla_{t,x}v|^2\,dx\right)=0\\
 \label{L:PP1_H3} E(u_0,u_1)=E(v_0,v_1)\\
\label{L:PP1_H3'} P(u_0,u_1)=P(v_0,v_1)=0\\
 \label{L:PP1_H4} \int \nabla u_0\cdot \nabla \Lambda W=\int \nabla v_0\cdot \nabla \Lambda W=0\\ 
 \label{L:PP1_H4'} 
 \forall j\in\llbracket 1,N\rrbracket,\; \int \nabla u_0\cdot \nabla \partial_{x_j}W=\int \nabla v_0\cdot\nabla\partial_{x_j}W=0.
\end{gather}
Then $(u_0,u_1)=(v_0,v_1)$ or $(u_0,u_1)=(v_0,-v_1)$.
\end{lemma}
\begin{proof}
\setcounter{step}{0}
 \begin{step}[Estimates on $u$ and $v$]
 \label{St:uv}
  We let 
  \begin{gather*}
  \beta_0:=\frac{1}{\|\Lambda W\|_{L^2}}\int \Lambda W\,u_1,\quad \gamma_0:=\frac{1}{\|\Lambda W\|_{L^2}}\int \Lambda W\,v_1\\
  \beta_j:=\frac{1}{\|\partial_{x_j} W\|_{L^2}}\int \partial_{x_j}W\,u_1,\quad \gamma_j:=\frac{1}{\|\partial_{x_j}W\|_{L^2}}\int \partial_{x_j}W\,v_1,\quad j\in \llbracket 1,5 \rrbracket.
  \end{gather*}
  Replacing $u(t,x)$ by $u(-t,x)$ or/and $v(t,x)$ by $v(-t,x)$ if necessary, we assume that $\beta_0$ and $\gamma_0$ are non-negative. 
  We let
 \begin{align*}
 (f_0,f_1)&=(u_0,u_1)-(W,0)-\frac{\beta_0}{\|\Lambda W\|_{L^2}}(0,\Lambda W)-\sum_{j=1}^5\frac{\beta_j}{\|\partial_{x_j}W\|_{L^2}}(0,\partial_{x_j}W)\\
 (g_0,g_1)&=(v_0,v_1)-(W,0)-\frac{\gamma_0}{\|\Lambda W\|_{L^2}}(0,\Lambda W)-\sum_{j=1}^5\frac{\gamma_j}{\|\partial_{x_j}W\|_{L^2}}(0,\partial_{x_j}W).
 \end{align*}
In this step, we prove
\begin{gather}
 \label{PP3}
 \left\|(u-W)\indic_{\{|x|>|t|\}}\right\|_{L^{\frac{7}{3}}_tL^{\frac{14}{3}}_x}\lesssim \eps,\quad
 \|(f_0,f_1)\|_{\HHH}\lesssim \beta_0^2,\quad \sum_{j=1}^5 |\beta_j|\lesssim \beta_0^3\\
 \label{PP3'}
 \left\|(v-W)\indic_{\{|x|>|t|\}}\right\|_{L^{\frac{7}{3}}_tL^{\frac{14}{3}}_x}\lesssim \eps,\quad
 \|(g_0,g_1)\|_{\HHH}\lesssim \gamma_0^2,\quad \sum_{j=1}^5 |\gamma_j|\lesssim \gamma_0^3.
\end{gather} 
The proofs of \eqref{PP3} and \eqref{PP3'} are the same, and we will only prove \eqref{PP3}. 

Let $a:=u-W$. We have 
\begin{equation*}
 |\partial_t^2a+L_Wa|\lesssim W^{\frac 13}a^2+|a|^{\frac 73},
\end{equation*}
By Lemma \ref{L:linear_approx} and H\"older estimates, using that $\|a(0)\|_{\HHH}$ is small, we obtain
\begin{multline*}
\forall T>0,\quad 
 \left\|a\indic_{\{|x|>|t|\}}\right\|_{L^{\frac 73}\left(((-T,+T),L^{\frac{14}{3}}\right)}\\
 \lesssim \|\vec{a}(0)\|_{\HHH}+
 \left\|a\indic_{\{|x|>|t|\}}\right\|_{L^{\frac 73}\left(((-T,+T),L^{\frac{14}{3}}\right)}^2+\left\|a\indic_{\{|x|>|t|\}}\right\|_{L^{\frac 73}\left(((-T,+T),L^{\frac{14}{3}}\right)}^{\frac 73}.
 \end{multline*}
Since $\|a(0)\|_{\HHH}$ is small, we deduce:
$$\left\|a\indic_{\{|x|>|t|\}}\right\|_{L^{\frac 73}_tL_x^{\frac{14}{3}}}\lesssim \|\vec{a}(0)\|_{\HHH},$$
which implies the first inequality in \eqref{PP3}.
Let $a_L$ be the solution of 
$$\partial_t^2a_L+L_W a_L=0,\quad \vec{a}_L(0)=(u_0,u_1)-(W,0).$$
We have 
\begin{equation}
 \label{PP4}
 \left|(\partial_t^2+L_W)(a-a_L)\right|\lesssim |a|^2W^{\frac 13}+|a|^{\frac 73},\quad |x|>|t|.
\end{equation} 
By Lemma \ref{L:linear_approx}, since $\vec{a}(0)-\vec{a}_L(0)=(0,0)$,
$$ \sup_{t\in \RR}\left\|\indic_{\{|x|>|t|\}}(a-a_L)\right\|_{\HHH}\lesssim \left\|\left(|a|^2W^{\frac 13}+|a|^{\frac 73}\right)\indic_{\{|x|>|t|\}}\right\|_{L^{1}_tL^2_x}\lesssim \|\vec{a}(0)\|_{\HHH}^2,$$
and hence, by assumption \eqref{L:PP1_H2},
$$\limsup_{t\to \pm\infty} \left( \int_{|x|>|t|} \left|\nabla_{t,x}a_L\right|^2 \right)^{1/2}\lesssim \left\|(u_0,u_1)-(W,0)\right\|^2_{\HHH}.$$
Combining with Theorem \ref{T:LW} and the orthogonality assumptions \eqref{L:PP1_H4}, \eqref{L:PP1_H4'}, we deduce
\begin{equation*}
 \|(f_0,f_1)\|_{\HHH}\lesssim \|(f_0,f_1)\|_{\HHH}^2+\sum_{j=0}^5\beta_j^2,
\end{equation*} 
and thus
 \begin{equation}
\label{PP4'}
 \|(f_0,f_1)\|_{\HHH}\lesssim \sum_{j=0}^5\beta_j^2,
\end{equation} 
Furthermore since $P(u_0,u_1)=0$ by assumption \eqref{L:PP1_H3'},
$$\int \left( \nabla W+\nabla f_0 \right)\left( f_1+\frac{\beta_0}{\|\Lambda W\|_{L^2}}\Lambda W+\sum_{j=1}^5\frac{\beta_j}{\|\partial_{x_j}W\|_{L^2}}\partial_{x_j}W \right)=0.$$
Since $W$ is radial, we have $\int \nabla W\Lambda W=0$, $\int \partial_{x_j}W\partial_{x_k}W=0$ ($j\neq k$) and we obtain for $k\in\llbracket 1, 5\rrbracket$ (using that $\int \partial_{x_k}Wf_1=0$),
\begin{multline}
\label{momentum}
\beta_{k}\|\partial_{x_k}W\|_{L^2}+\int \partial_{x_k}f_0f_1\\
+\frac{\beta_0}{\|\Lambda W\|_{L^2}}\int \partial_{x_k}f_0\Lambda W +\sum_{j=1}^5 \frac{\beta_j}{\|\partial_{x_j}W\|_{L^2}}\int \partial_{x_k}f_0\partial_{x_j}W=0. 
\end{multline} 
Thus by \eqref{PP4'}
$$\sum_{k=1}^5 |\beta_k|\lesssim \left( \sum_{j=0}^5|\beta_j| \right)^3,$$
which yields
$$\sum_{k=1}^5 |\beta_k|\lesssim \beta_0^3.$$
Going back to \eqref{PP4'}, we deduce \eqref{PP3}.
 \end{step}
\begin{step}[Estimate on $u-v$]
\label{St:h}
We let 
\begin{gather*}
h:=u-v, \; (h_0,h_1):=\vec{h}(0)=(u_0,u_1)-(v_0,v_1),\\
\alpha_0:=\frac{1}{\|\Lambda W\|_{L^2}}\int h_1\Lambda W=\beta_0-\gamma_0,\quad \alpha_j:=\frac{1}{\|\partial_{x_j}W\|_{L^2}}\int h_1\partial_{x_j}W=\beta_j-\gamma_j,\\
h_1^{\bot}=h_1-\frac{\alpha_0}{\|\Lambda W\|_{L^2}} \Lambda W-\sum_{j=1}^5\frac{\alpha_j}{\|\partial_{x_j} W\|_{L^2}}\partial_{x_j}W=f_1-g_1.
\end{gather*}
 In this step, we prove
 \begin{equation}
  \label{PP5}
  \sum_{k=1}^{5}|\alpha_k|+
 \left\|(h_0,h_1^{\bot})\right\|_{\HHH}\lesssim \eps\alpha_0.
  \end{equation} 
Let $F(u)=|u|^{\frac 43}u$. We have 
\begin{equation}
 \label{PP6}
 \partial_t^2h+L_W h=F(u)-F(v)-F'(W)h.
\end{equation} 
We claim 
\begin{multline}
 \label{PP7}
 \left|F(u)-F(v)-F'(W)h\right|\\
 \lesssim |h|\left( |v-W|^{\frac 43}+|W|^{\frac 13}|v-W|+|W|^{\frac 13}|h|+|h|^{\frac 43} \right).
\end{multline}
Indeed, using elementary estimates on the function $F$, we obtain, since $h=u-v$,
\begin{multline*}
 |F(u)-F(v)-F'(v)h|\lesssim |v|^{\frac 13}|h|^2+|h|^{\frac 73}\lesssim W^{\frac{1}{3}}h^2+|v-W|^{\frac 13}h^2+|h|^{\frac 73}\\
 \lesssim W^{\frac 13}h^2+|v-W|^{\frac 43}|h|+|h|^{\frac 73},
\end{multline*}
where we have used that by Young's inequality, 
$|v-W|^{\frac 13}h^2\lesssim |v-W|^{\frac 43}|h|+|h|^{\frac{7}{3}}$.
Moreover,
$$|F'(v)-F'(W)|\lesssim |v-W|\,|W|^{\frac 13}+|v-W|^{\frac 43},$$
which yields \eqref{PP7}.

 By the first inequality in \eqref{PP3'}, and the fact that $W\indic_{\{|x|>|t|\}}\in L^{\frac 73}_tL^{\frac{14}{3}}_x$, we first obtain from \eqref{PP6}, \eqref{PP7} and Strichartz estimates
 $$\left\|h\ind\right\|_{L^{\frac 73}_tL^{\frac{14}{3}}_x}\lesssim \left\|(h_0,h_1)\right\|_{\HHH}.$$
 Letting $h_L$ be the solution of 
 $$(\partial_t^2+L_W)h_L=0,\quad \vec{h}_L(0)=(h_0,h_1),$$
we deduce, again from \eqref{PP6}, \eqref{PP7} and Lemma \ref{L:linear_approx},
 $$\sup_{t\in \RR} \int_{|x|>|t|} \left|\nabla_{t,x}(h-h_L)\right|^2\,dx\lesssim \eps^2\left\|(h_0,h_1)\right\|^2_{\HHH}.$$
 Using the assumption \eqref{L:PP1_H2}, we obtain
 $$\sum_{\pm}\lim_{t\to \pm\infty} \int_{|x|>|t|}|\nabla_{t,x}h_L(t,x)|^2\,dx\lesssim \eps^2\left\|(h_0,h_1)\right\|^2_{\HHH}.$$
 By Theorem \ref{T:LW}, and the orthogonality conditions \eqref{L:PP1_H4}, \eqref{L:PP1_H4'},
 \begin{equation}
 \label{PP7.1}
 \left\|h_1^{\bot}\right\|_{L^2}+\left\|h_0\right\|_{\hdot}\lesssim \eps\|(h_0,h_1)\|_{\HHH}=\eps\sqrt{\left\|(h_0,h_1^{\bot})\right\|^2_{\HHH}+\sum_{j=0}^5\alpha_j^2}. 
 \end{equation} 
 Expanding the momentums in the equality $P(u_0,u_1)-P(v_0,v_1)=0$ as in \eqref{momentum}, we obtain that for all $k$ in $\llbracket 1,5\rrbracket$, 
 \begin{multline*}
  0=(\beta_k-\gamma_k)\|\partial_{x_k}W\|_{L^2}+\int \partial_{x_k}f_0f_1-\partial_{x_k}g_0g_1\\
  +\frac{\beta_0}{\|\Lambda W\|_{L^2}}\int \partial_{x_k}f_0 \Lambda W-\frac{\gamma_0}{\|\Lambda W\|_{L^2}}\int \partial_{x_k}g_0\Lambda W\\+\sum_{j=1}^5 \left( \frac{\beta_j}{\|\partial_{x_j}W\|_{L^2}} \int \partial_{x_k}f_0\partial_{x_j}W-\frac{\gamma_j}{\|\partial_{x_j}W\|_{L^2}} \int \partial_{x_k}g_0\partial_{x_j}W \right).
 \end{multline*}
Using that by the assumption \eqref{L:PP1_H1},
$$\|\nabla f_0\|_{L^2}+\|\nabla g_0\|_{L^2}+\|f_1\|_{L^2}+\|g_1\|_{L^2}+\sum_{j=0}^5 |\beta_j|+|\gamma_j|\lesssim \eps,$$
we deduce
$$ \sum_{k=1}^5|\alpha_k|=\sum_{k=1}^5|\beta_k-\gamma_k|\lesssim \eps\|(h_0,h_1)\|_{L^2}.$$
Since 
$$\|(h_0,h_1)\|^2_{\HHH}=\|(h_0,h_1^{\bot}\|_{\HHH}^2+\sum_{j=0}^5 \alpha_j^2,$$
we obtain 
$$\sum_{j=1}^5 |\alpha_j|\lesssim \eps\left( \|(h_0,h_1^{\bot}) \|_{\HHH}+|\alpha_0|\right).$$
Combining with \eqref{PP7.1}, we obtain \eqref{PP5}.
 \end{step}
\begin{step}[Estimates on the energy and conclusion of the proof]
 \label{St:EP}
 We claim:
 \begin{gather}
 \label{PP8'} \left|\alpha_0(\beta_0+\gamma_0)\right|\lesssim \eps|\alpha_0| (\beta_0^2+\gamma_0^2).
 \end{gather}
By the assumption \eqref{L:PP1_H3}, $E(u_0,u_1)=E(v_0,v_1)$. Noting that
\begin{align*}
E(u_0,u_1)&=E(u_0,0)+\frac 12\beta_0^2+\frac 12\sum_{j=1}^5 \beta_j^2+\frac 12\|f_1\|^2_{L^2}\\
E(v_0,v_1)&=E(v_0,0)+\frac 12\gamma_0^2+\frac 12\sum_{j=1}^5 \gamma_j^2+\frac 12\|g_1\|_{L^2}^2, 
\end{align*}
and $\alpha_j=\beta_j-\gamma_j$, for $j\in \llbracket 0,5\rrbracket$, the estimate \eqref{PP8'} will follow from the following inequalities:
 \begin{gather}
 \label{PP10}
 \forall j\in \llbracket 1,5\rrbracket,\quad 
 \left|\beta_j^2-\gamma_j^2\right|\lesssim \eps |\alpha_0|(\beta_0^2+\gamma_0^2)\\
  \label{PP8}
  \left|E(u_0,0)-E(v_0,0)\right|\lesssim \eps |\alpha_0|(\beta_0^2+\gamma_0^2)\\
  \label{PP9}
  \left|\|f_1\|_{L^2}^2-\|g_1\|_{L^2}^2\right|\lesssim \eps |\alpha_0|(\beta_0^2+\gamma_0^2).
 \end{gather}
 By Step  \ref{St:h}, for $j\in \llbracket 1,5\rrbracket$
 $$ \left|\beta_j^2-\gamma_j^2\right|\lesssim |\alpha_j| \left( |\beta_j|+|\gamma_j| \right)\lesssim \eps|\alpha_0|\left( |\beta_j|+|\gamma_j| \right),$$
 and we deduce from Step \ref{St:uv}, 
 $$ \left|\beta_j^2-\gamma_j^2\right|\lesssim \eps |\alpha_0|\left( \beta_0^2+\gamma_0^2 \right).$$
 This yields \eqref{PP10}.
 
 To prove \eqref{PP8}, we expand
 $$E(u_0,0)-E(v_0,0)=E(W+f_0,0)-E(W+g_0,0),$$
 and see that the linear terms in $f_0,g_0$ are zero (since $(W,0)$ is a critical point for $E$). As a consequence (using the Sobolev inequality to control the terms coming from $\int |W+f_0|^{\frac{10}{3}}-\int |W+g_0|^{\frac{10}{3}}$, we obtain
 $$\left|E(u_0,0)-E(v_0,0)\right|\lesssim \|f_0-g_0\|_{\HHH}\left( \|f_0\|_{\HHH}+\|g_0\|_{\HHH} \right)\lesssim \left\|h_0\right\|_{\HHH}\left( \left\|f_0\right\|_{\HHH}+\|g_0\|_{\HHH} \right),$$
 and \eqref{PP8} follows from the estimates of Steps \ref{St:uv} and \ref{St:h}. The proof of \eqref{PP9} is similar, using that $h_1^{\bot}=f_1-g_1$ and thus that
 $$\left|\|f_1\|^2_{L^2}-\|g_1\|^2_{L^2}\right\|\leq \|h_1^{\bot}\|_{L^2}\left( \|f_1\|_{L^2}+\|g_1\|_{L^2} \right).$$

 In view of \eqref{PP8'},
since $\beta_0$ and $\gamma_0$ are non-negative, we obtain $\alpha_0=0$ or $\beta_0=\gamma_0=0$ (which implies $\alpha_0=0$). By \eqref{PP5} $(h_0,h_1)=(0,0)$, i.e. $(u_0,u_1)=(v_0,v_1)$. Recalling that we might have changed $u_1$ into $-u_1$ and $v_1$ into $-v_1$ (see the beginning of Step \ref{St:uv}), we obtain the conclusion of the lemma.
\end{step}
\end{proof}
\begin{proof}[Proof of Theorem \ref{T:NL5}]
Recall that for all $\vell$ with $|\vell|<1$, we have
$$ P\left(\vec{W}_{\vell}(0)\right)=-\frac{\vell}{\sqrt{1-|\vell|^2}}E(W,0) ,\quad  E\left(\vec{W}_{\vell}(0)\right)=\frac{1}{\sqrt{1-|\vell|^2}}E(W,0),$$
Since $(u_0,u_1)$ is close to $\vec{W}_{\vell}(0)$ for some $|\vell|\leq \eta_0$, we deduce that $|P(u_0,u_1)|<E(u_0,u_1)$.
Let 
$$ \vell_0=\frac{P(u_0,u_1)}{E(u_0,u_1)}.$$
By Lemma \ref{L:Lorentz} in the appendix, $P(\vu_{\vell_0}(0))=0$. By Lemma \ref{L:Lorentz_perturb} and the assumption\eqref{dist_u_W}, there exists $\vell'$ with $|\vell'|<1$ such that 
$$\left\|\vec{u}_{\vell_0}(0)-\vec{W}_{\vell'}(0)\right \|_{\HHH}\lesssim \eps$$
(where the implicit constant depends only on $\eta_0$).
This yields
$$\left|P(\vec{u}_{\vell_0}(0))-P(\vec{W}_{\vell'}(0)\right|\lesssim \eps,$$
and thus, since $P(\vec{u}_{\vell_0})=0$, 
$$\frac{|\vell'|}{\sqrt{1-|\vell'|^2}}\lesssim \eps.$$
As a consequence, $|\vell'|\lesssim \eps$. Using that by direct computation, $\left\|(W,0)-\vec{W}_{\vell'}(0)\right\|_{\HHH}\lesssim |\vell'|$ for small $\vell'$, we obtain
$$\|\vec{u}_{\vell_0}(0)-(W,0)) \|_{\HHH}\lesssim \eps.$$

Using the same arguments as in the end of Subsection \ref{SS:Lor} (see in particular \eqref{bound_Lorentz}), one can also check that the Lorentz transformation preserves the assumption \eqref{NL0_5} (with a smaller $\tau_0$). 

We are thus reduced to the case where
 \begin{equation}
  \label{zero_momentum}
  P(u_0,u_1)=0
 \end{equation}
 and $\|(u_0,u_1)-(W,0)\|_{\HHH}\lesssim \eps$.
 Since $\vec{u}$ is a $\HHH$-valued continuous function, we deduce that for small $t$,
 $$\|\vec{u}(t)-(W,0)\|_{\HHH}\lesssim \eps.$$
We also see that the assumption \eqref{NL0_5} implies that for small $t_0$
$$\sum_{\pm} \lim_{t\to\pm\infty} \int_{|x|\geq |t|-\tau_0/2} |\nabla_{t,x}(u(t+t_0,x))|^2\,dx=0.$$
By a standard application of the implicit function theorem, we can find $\lambda(t)>0$ close to $1$ and a small $x(t)\in \RR^5$ such that $u_{(\lambda(t))}(\cdot-x(t))$ satisfies the orthogonality assumptions \eqref{L:PP1_H4} and \eqref{L:PP1_H4'} of Lemma \ref{L:PP1}.
 By Lemma \ref{L:PP1}, for small $t$, 
 $\vec{u}(t)$ is equal to $\pm (u_0,u_1)$ up to translation and scaling. Using the continuity of the flow, we see that $\vec{u}(t)$ must be close to $(u_0,u_1)$ for small $t$.

We are thus reduced to proving the following:
\begin{claim}
\label{C:self_similar}
 Let $u$ be a solution of \eqref{NLW} defined on a interval $I$ containing $0$, such that $P(u_0,u_1)=0$ and $(u_0,u_1)$ is close to $W$. Assume
 \begin{multline}
\label{self_similarity}
 \forall \tau\in  I,\quad \exists x(\tau)\in \RR^N,\;\exists \lambda(\tau)\in (0,\infty),\\ \vec{u}(\tau,x)=\left( \frac{1}{\lambda(\tau)^{\frac{N}{2}-1}}u_0\left( \frac{x+x(\tau)}{\lambda(\tau)} \right), \frac{1}{\lambda(\tau)^{\frac{N}{2}}}u_1\left( \frac{x+x(\tau)}{\lambda(\tau)} \right)\right).
 \end{multline} 
 Then 
 $$\exists \lambda>0,\; \exists X\in \RR^N,\quad (u_0(x),u_1(x))=\left(\frac{1}{\lambda^{\frac{N}{2}-1}}W\left( \frac{x}{\lambda} \right),0\right).$$
\end{claim}
\end{proof}

\begin{proof}[Proof of Claim \ref{C:self_similar}]
By an easy induction, \eqref{self_similarity} holds for all $\tau$ in the maximal interval of definition of $u$. As a consequence, $u$ has the compactness property and remains close to $W$ in the energy space, up to translation and scaling. By \cite{DuKeMe16a}, and since $P(u_0,u_1)=0$, we deduce $(u_0,u_1)=(W,0)$.

 Let us mention that it is possible to prove a more general version of the claim, omitting the assumptions ``$P(u_0,u_1)=0$'' and ``$(u_0,u_1)$ is close to $W$''. In this case the conclusion is that $u$ is a general solitary wave.
 \end{proof}
\subsection{Channels of energy below twice the energy of the ground state}
\label{SS:rig2}
In this subsection we prove Corollary \ref{Cor:rigidity_intro2} and Proposition \ref{P:rigidity_intro3}. 
\begin{proof}[Proof of Corollary \ref{Cor:rigidity_intro2}]
We argue by contradiction. Let $u$ satisfy the assumptions of the corollary, and assume furthermore:
\begin{equation}
 \label{no_channels2}
 \forall \tau_0>0,\quad \sum_{\pm} \lim_{t\to\pm\infty}\int_{|x|>|t|-\tau_0} |\nabla_{t,x}u(t,x)|^2\,dx=0.
\end{equation} 

According to \cite{DuJiKeMe17} if $N\in \{3,5\}$ there exists a sequence of times $\{t_n\}_n$ with 
$$\lim_{n\to +\infty} t_n=+\infty,$$
an integer $J\ge 0$, scales $\lambda_n^j$ with $\lambda_n^j>0$ and $\lim_{n\to\infty}\lambda^j_n/t_n=0$, positions $x_n^j\in \RR^d$, with $\vell_j=\lim_{n\to\infty}\frac{x_n^j}{t_n}$ well defined, and traveling waves $Q_{\vell_j}^j$, for $1\leq j\leq J$, such that 
\begin{equation}
\label{decomposition}
\vec{u}(t_n)=\sum_{j=1}^{J}\,\left((\lambda_n^j)^{-\frac{N}{2}+1}\, Q_{\vell_j}^j\left(\frac{x-x_n^j}{\lambda_n^j},\,0\right),\,(\lambda_x^j)^{-\frac{N}{2}}\, \partial_tQ_{\vell_j}^j\left(\frac{x-x_n^j}{\lambda_n^j},\,0\right)\right)+o(1),
\end{equation}
in $\HHH$, as $n\to\infty$.
In addition, the parameters $\lambda_n^j,\,x_n^j$ satisfy the pseudo-orthogonality condition 
\begin{equation}
\label{ortho}
1\leq j<k\leq J\Longrightarrow \lim_{n\to\infty}
\frac{\lambda_n^j}{\lambda_n^{k}}+\frac{\lambda_n^{k}}{\lambda_n^j}+\frac{\left|x^j_n-x^{k}_n\right|}{\lambda^j_n}=+\infty.
\end{equation}
We note that by \eqref{no_channels2}, one can assume that the linear solution that appears usually in the expansion \eqref{decomposition} is identically $0$. 

By \eqref{decomposition}, \eqref{ortho}, we have 
$$E(u_0,u_1)=\sum_{j=1}^J E(\vec{Q}^j_{\vell_j}(0)=\sum_{j=1}^J \frac{1}{\sqrt{1-|\vell_j|^2}}E(Q^j,0).$$
Recall that 
$$ Q\in \Sigma\text{ and } E(Q,0)<2E(W,0)\Longrightarrow \exists \lambda>0, \; Q=\pm W_{(\lambda)}$$
(see e.g. \cite{DuKeMe12} for a proof). Thus the assumption $E(u_0,u_1)<2E(W,0)$ implies that $J\leq 1$, and that $Q^1=W$ if $J=1$. If $J=0$, then by conservation of the energy and \eqref{decomposition}, $u$ is identically $0$ and we are done. We are thus reduced to the case $J=1$. If $N\in \{3,5\}$, Corollary \ref{Cor:rigidity_intro} immediately yields a contradiction with \eqref{no_channels2}, unless $u$ is a solitary wave.
\end{proof}
\begin{proof}[Proof of Proposition \ref{P:rigidity_intro3}]
 Let $u$ be a global radial solution of \eqref{NLW} with $N\geq 7$ odd, such that $E(u_0,u_1)<2E(W,0)$ and
 \begin{equation}
  \label{no_radiation_end}
  \forall \tau_0>0,\quad
  \sum_{\pm}\lim_{t\to\pm\infty} \int_{|x|>|t|-\tau_0}|\nabla_{t,x}u(t,x)|^2\,dx=0.
 \end{equation} 
 By \cite{Rodriguez16}, \eqref{decomposition} holds with for all $j$, $Q^j\in \{\pm W\}$, $\vell_j=0$ and $x_n^j=0$. By the assumption $E(u_0,u_1)<2E(W,0)$, we obtain that $J=1$ (again, $J=0$ is excluded since $u$ is nonzero), and $E(u_0,u_1)=E(W,0)$. We are thus reduced to proving:
\begin{claim}
\label{Cl:rigidity_largeN}
Assume $N\geq 7$. 
There exists $\eps_0>0$ with the following property.
 Let $u$ be a global, radial solution of \eqref{NLW} such that $E(u_0,u_1)=E(W,0)$,
 $$\forall A\geq 0,\quad \lim_{t\to \pm\infty} \int_{|x|>|t|-A}|\nabla_{t,x}u(t,x)|^2\,dx=0$$
 and 
 $$\|(u_0,u_1)-(W,0)\|=\delta<\eps_0.$$
Then $(u_0,u_1)=(W,0)$.
\end{claim}
Indeed, Claim \ref{Cl:rigidity_largeN} applied to the solution $(t,x)\mapsto u(t_n+t,x)$ for some large $n$ yields the desired result.
\end{proof}
\begin{proof}[Proof of the Claim]
By a standard use of the implicit function theorem, we can assume
$$ \int \nabla u_{0}\nabla \Lambda W=0.$$
Expand $(u_0,u_1)$ as follows:
$$ (u_0,u_1)=\Big(W+f_0,\beta \Lambda W+ f_1\Big),\quad \int f_1\Lambda W=0.$$
Let $h(t)=u(t)-W$.
Then $\|\vec{h}(0)\|_{\HHH}=\delta$ and
 \begin{equation}
  \label{MNL1}
  \partial_t^2h+L_{W} h=F(h)+\NNN(h),
 \end{equation} 
 where 
 $$\NNN(h)=F\left( W+h \right)-F(W)-F(h)-\frac{N+2}{N-2}W^{\frac{4}{N-2}}h.$$
 By finite speed of propagation, $h$ coincides, for $|x|>|t|$, with the solution $\tlh$ of
\begin{equation}
  \label{MNL1'}
  \partial_t^2\tlh+L_{W} \tlh=(F(h)+\NNN(h))\indxt,
 \end{equation} 
 Let $T>0$ and denote by $\Gamma(T)=\Big\{(t,x),\; |t|\leq \min \{|x|,T\} \Big\}$.
 By the fractional chain rule \eqref{fractional_cones},
 \begin{equation}
  \label{MNL2}
  \left\|F(h)\indxt\right\|_{W'((0,T))}=\left\|F(\tlh)\indxt\right\|_{W'((0,T))}\lesssim \|\tlh\|_{W((0,T))}\|\tlh\|_{S(\Gamma(T))}^{\frac{4}{N-2}}.
 \end{equation} 
By the inequality
\begin{equation}
 \label{esti_linearisation}
\left|F\left( y+h \right)-F(y)-F(h)-\frac{N+2}{N-2}|y|^{\frac{4}{N-2}}h\right|
\lesssim |y|\,|h|^{\frac{N+1}{N-2}},
\end{equation}
we obtain
$$|\NNN(h)|\lesssim W^{\frac{1}{N-2}}|h|^{\frac{N+1}{N-2}}.$$
Furthermore,
\begin{multline*}
\left\|\indic_{\Gamma(T)} W^{\frac{1}{N-2}}|h|^{\frac{N+1}{N-2}}\right\|_{L^1_tL^2}\\
\leq \left\|\indic_{\Gamma(T)} W^{\frac{1}{N-2}}\right\|_{L^{2}_tL^{\infty}_x}\left\|\indic_{\Gamma(T)} |h|^{\frac{N+1}{N-2}}\right\|_{L^2_{t,x}}\lesssim \|\tlh\|^{\frac{N+1}{N-2}}_{S(\Gamma(T)}, 
\end{multline*}
where we have used that since $W^{\frac{1}{N-2}}\lesssim \frac{1}{1+|x|}$,
$W^{\frac{1}{N-2}}\indic_{\{|x|\geq t\}}$ is in  $L^{2}_t\left(\RR,L^{\infty}_x(\RR^N)\right)$.
We let $h_{L}(t)$ be the solution of 
\begin{equation*}
 \partial_t^2h_L+L_{W}h_L=0,\quad \vec{h}_{L\restriction t=0}=(u_0,u_1)-(W,0).
\end{equation*} 
By the perturbation Lemma for the linear wave equation with a potential (Lemma \ref{L:linear_approx}), we obtain
\begin{equation*}
\left\|\tlh-h_L\right\|_{S(\Gamma_T)}\lesssim \|\tlh\|_{S(\Gamma_T)}^{\frac{N+1}{N-2}}+\|\tlh\|_{S(\Gamma_T)}^{\frac{4}{N-2}}\|\tlh\|_{W((-T,T)}.
\end{equation*} 
Using again Strichartz estimates, we deduce
\begin{multline}
\label{MNL3}
\sup_{-T\leq t\leq T}\|\vec{\tlh}(t)-\vec{h}_L(t)\|_{\HHH}+
\left\|\tlh-h_L\right\|_{W((-T,T))\cap S((-T,T))}\\
\lesssim \|\tlh\|_{S(\Gamma_T)}^{\frac{N+1}{N-2}}+\|\tlh\|_{S(\Gamma_T)}^{\frac{4}{N-2}}\|\tlh\|_{W((-T,T)}.
\end{multline} 
and thus, since $\|h_L\|_{W((-T,T))\cap S((-T,T))}\lesssim \delta$,
$$\|\tlh\|_{W((-T,T))\cap S((-T,T))}\lesssim \delta.$$
Going back to \eqref{MNL3} we obtain
\begin{equation*}
\sup_{-T\leq t\leq T} \|\vec{\tlh}(t)-\vec{h}_L(t)\|_{\HHH}\lesssim \delta^{\frac{N+1}{N-2}}
\end{equation*} 
This estimate is uniform in $T$. Hence
\begin{equation*}
\sup_{t\in \RR} \|\vec{\tlh}(t)-\vec{h}_L(t)\|_{\HHH}\lesssim \delta^{\frac{N+1}{N-2}}.
\end{equation*} 
Using that $u$ is non-radiative, we deduce
$$\sum_{\pm}\lim_{t\to\pm\infty} \int_{\{|x|>|t|\}} |\nabla_{t,x}h_L(t,x)|^2\,dx\lesssim \delta^{\frac{N+1}{N-2}}.$$
By Theorem \ref{T:LW}, 
$$\left\|(f_0,f_1)\right\|_{\HHH}=\left\|\Pi_{\ZZZZ^{\bot}}(f_0,u_1)\right\|_{\HHH} \lesssim \delta^{\frac{N+1}{N-2}}.$$
Since $\delta^2\approx \|(f_0,f_1)\|^2_{\HHH}+\beta^2$, this
yields
\begin{equation}
\label{bound_by_alpha}
\left\|(f_0,f_1)\right\|_{\HHH} \lesssim \beta^{\frac{N+1}{N-2}}. 
\end{equation} 
Expanding the equality $E(W,0)=E(W+f_0,\beta \Lambda W+f_1)$, we obtain
$$ \beta^2\lesssim \|f_0\|_{\hdot}^2+\|f_0\|_{\hdot}^{\frac{2N}{N-2}}+\|f_1\|^2_{L^2},$$
which yields
$$\beta^2 \lesssim \beta^{\frac{2(N+1)}{N-2}}.$$
This proves that $\beta=0$ and by \eqref{bound_by_alpha}, that $(f_0,f_1)=(0,0)$. We have proved as announced that $(u_0,u_1)=(W,0)$.
\end{proof}

\appendix

\section{Lorentz transformation}
\label{S:Lorentz}
This appendix concerns the effect of the Lorentz transformations on solutions of \eqref{NLW}. If $u$ is a $C^2$ classical solution of \eqref{NLW}, then by direct computation, $u_{\vell}(t,x)$ (defined by \eqref{defLl} is also a $C^2$ classical solution of \eqref{NLW} on its domain of definition. The Lorentz tranform of a general finite energy solution of \eqref{NLW} (as defined in Definition \ref{D:solution} above) is more difficult to understand. If $u$ is global, the formula \eqref{defLl} makes sense, and one can prove that $u_{\vell}$ has indeed finite energy
and is a solution of \eqref{NLW} in the sense of Definition \ref{D:solution} (see e.g. \cite[Lemma 6.1]{DuKeMe16a}).

If $u$ is not globally defined, the formula \eqref{defLl} does not make sense anymore. In this section we prove however that using the Definition \ref{D:sol_cone} of solutions of \eqref{NLW} outside wave cones,
we can define the Lorentz transformation of a class of nonglobal solutions, that include a neighborhood of any global solution. 

If $\vell\in \RR^N$ with $|\vell|<1$, we denote by
$$c_{\vell}:=\sqrt{\frac{1+|\vell|}{1-|\vell|}}>1.$$
Let $(t,x)\in \RR^N$, and $(s,y)$ given by the change of variable of the Lorentz transformation:
$$(s,y)=\left(\frac{t-\vell\cdot x}{\sqrt{1-\ell^2}},\left(-\frac{t}{\sqrt{1-\ell^2}}+\frac{1}{\ell^2} \left(\frac{1}{\sqrt{1-\ell^2}}-1\right)\vell\cdot x\right)\vell+x\right). $$
Then 
$$ |x|^2-t^2=|y|^2-s^2$$
and 
$$|s|+|y|\leq c_{\vell}(|t|+|x|),\quad |t|+|x|\leq c_{\vell}(|s|+|y|).$$
This can be checked easily, assuming for example that $\vell=(\ell,0,\ldots,0)$, so that 
\begin{equation}
\label{defsy}
(s,y)=\left(\frac{t-\ell x_1}{\sqrt{1-\ell^2}},\frac{x_1-t\ell}{\sqrt{1-\ell^2}},x_2,\ldots,x_N \right). 
\end{equation} 
\begin{lemma}
\label{L:Lorentz}
Let $\eta_0\in (0,1)$. There exists $T>0$ with the following property. Let $\tau\geq T$, $u$ be a scattering solution of \eqref{NLW} in $\{|x|>|t|-\tau\}$ with initial data $(u_0,u_1)\in \HHH$ at $t=0$, and $\vell\in \RR^N$ with $|\vell|\leq \eta_0$. Then the formula \eqref{defLl} makes sense for $t\in [-c_{\vell}^{-1}\tau,c_{\vell}^{-1}\tau]$ and defines a solution of \eqref{NLW} on $[-c_{\vell}^{-1}\tau,c_{\vell}^{-1}\tau]\times \RR^N$.
Furthermore,
\begin{align}
 \label{Lorentz_energy}
 E(\vec{u}_{\vell}(0))&=\frac{E(u_0,u_1)}{\sqrt{1-|\vell|^2}} -\frac{1}{\sqrt{1-|\vell|^2}} \vell \cdot P(u_0,u_1)\\
\label{Lorentz_momentum}
 P(\vec{u}_{\vell}(0))&=P(u_0,u_1)+\frac{\vell \cdot P(u_0,u_1)}{|\vell|^2}\left( \frac{1}{\sqrt{1-|\vell|^2}}-1\right)\vell-\frac{E(u_0,u_1)}{\sqrt{1-|\vell|^2}} \vell.
 \end{align}
\end{lemma}
\begin{lemma}
 \label{L:Lorentz_perturb}
Let $\eta_0$, $\tau$ and $u$ be as in Lemma \ref{L:Lorentz}. There exist constants $\eps_0>0$ and $C>0$ (depending on $u$, $\tau$ and $\eta_0$) such that if $(v_0,v_1)\in \HHH$ and 
$$\|(u_0,u_1)-(v_0,v_1)\|_{\HHH}<\eps_0,$$
then the solution $v$ of \eqref{NLW} in $\{|x|>|t|-\tau\}$ with initial data $(v_0,v_1)$ at $t=0$ is scattering, and, if $|\vell|\leq \eta_0$, 
$$ \left\|\vec{u}_{\vell}(0)-\vec{v}_{\vell}(0)\right\|_{\HHH}\leq C\left\|(u_0,u_1)-(v_0,v_1)\right\|_{\HHH}.$$
\end{lemma}
\begin{remark}
 Let $u$ be a global solution of \eqref{NLW}. Then by \cite{DuKeMe19}, we can see that for all $A\in \RR$,
$$u\indic_{|x|\geq |t|+A}\in L^{\cinq}\left(\RR,L^{\dix}(\RR^N)\right).$$
Thus Lemma \ref{L:Lorentz} applies and one can define the Lorentz transform $u_{\vell}$ (which is global) of $u$ for any parameter $\vell$, with $|\vell|<1$. Furthermore by Lemma \ref{L:Lorentz_perturb}, for all $\eta_0$, there exists $\eps_0$ such that if $\|(u_0,u_1)-(v_0,v_1)\|_{\HHH}<\eps_0$ and $|\vell|\leq \eta_0$, then one can define the Lorentz transform $v_{\vell}$ of the solution $v$ of \eqref{NLW} with initial data $(v_0,v_1)$.
\end{remark}
\subsection{Lorentz transform of a solution}
In this subsection we prove the first part of Lemma \ref{L:Lorentz}, i.e. the fact that $u_{\vell}(t)$ is well-defined for $t\in [-c_{\vell}\tau,c_{\vell}\tau]$. We assume without loss of generality
$$\vell=(\ell,0,\ldots,0).$$
We recall from \cite[Lemma 2.2 and Remark 2.3]{KeMe08} the following claim:
\begin{claim}
\label{C:linear_Lorentz}
Let $\eta_0\in (0,1)$, $h\in L^1(\RR,L^2(\RR^N))$, $(w_0,w_1)\in \hdot\times L^2$, $\vell\in \RR^N$ with $|\vell|\leq \eta_0$ and 
\begin{equation}
\label{LCP}
w(t)=\cos(t\sqrt{-\Delta})w_0+\frac{\sin(t\sqrt{-\Delta})}{\sqrt{-\Delta}}w_1+\int_0^t\frac{\sin\left( (t-s)\sqrt{-\Delta} \right)}{\sqrt{-\Delta}}h(s)\,ds,\quad t\in \RR.
\end{equation}
Then $(w_{\vell},\partial_tw_{\vell})\in C^0\left(\RR,\hdot\times L^2\right)$ and there is a constant $C_{\eta_0}$ (depending only on $\eta_0$) such that
\begin{equation*}
 \sup_{t} \left\|(w_{\ell}(t),\partial_tw_{\ell}(t)\right\|_{\hdot\times L^2}\leq C_{\eta_0}\left( \|(w_0,w_1)\|_{\hdot\times L^2}+\|h\|_{L^1(\RR,L^2)} \right).
\end{equation*} 
\end{claim}
\setcounter{step}{0}
\begin{step}[Smooth compactly-supported initial data]
We first assume $(u_0,u_1)\in \left(C_0^{\infty}(\RR^N)\right)^2$, and denote by $R$ a positive number such that $(u_0,u_1)(y)=0$ for $|
y|\geq R$.  We denote by $\EEE$ the exterior of the wave cone:
$$\EEE:=\left\{(s,y)\in \RR\times \RR^N\;|\;|y|>|s|-\tau\right\},$$
and $\EEE_{\vell}$ its image:
\begin{align*} 
\EEE_{\vell}&:=\left\{\left( \frac{s+\ell y_1}{\sqrt{1-\ell^2}},\frac{y_1+s\ell}{\sqrt{1-\ell^2}},y_2,\ldots,y_N \right),\; (s,y)\in \EEE\right\}\\
&=
\bigg\{(t,x)\in \RR\times \RR^N\;\Big|\;\left( \frac{t-\ell x_1}{\sqrt{1-\ell^2}},\frac{x_1-t\ell}{\sqrt{1-\ell^2}},x_2,\ldots,x_N \bigg)\in \EEE\right\}.
\end{align*}
One can prove 
\begin{equation}
\label{continuity}
u\in C^0(\EEE). 
\end{equation} 
Indeed, since the  nonlinearity $F$ is $C^2$, we have that for all $(s_0,y_0)\in \EEE$, there exists a neighborhood $J\times \omega$ of $(s_0,y_0)$ in $\EEE$ such that 
$$\vec{u}\in C^0\left( J,(H^{3}\times H^2)(\omega) \right),$$
and \eqref{continuity} follows from Sobolev embedding (recall that $N\leq 5$). By \eqref{continuity} and the definition of $u_{\vell}$, 
\begin{equation}
\label{continuity'}
u_{\vell}\in C^0(\EEE_{\vell}). 
\end{equation} 
We next prove that if $t$ satisfies $|t|\leq c_{\vell}^{-1}\tau$ and $x\in \RR^N$, then $(t,x)\in \EEE_{\vell}$. Indeed, letting $(s,y)$ be as in \eqref{defsy}, we must prove that $(s,y)\in \EEE$. We have
$$ |y|-|s|=\frac{|y|^2-s^2}{|y|+|s|}=(|x|-|t|)\frac{|x|+|t|}{|y|+|s|}\geq -|t|\frac{|x|+|t|}{|y|+|s|}\geq -c_{\vell}^{-1}|\tau|\frac{|x|+|t|}{|y|+|s|}.$$
Since $\frac{|x|+|t|}{|y|+|s|}\leq c_{\vell}$, we deduce $(s,y)\in \EEE$, i.e. $(t,x)\in \EEE_{\vell}$. Using that $\indic_{|y|\geq |s|-|\tau|}u\in L^{\cinq}(\RR,L^{\dix})$, we obtain by Claim \ref{C:linear_Lorentz}
\begin{equation}
 \label{uell_H}
\vec{u}_{\vell}\in C^0\left([-c_{\vell}^{-1}\tau,c_{\vell}^{-1}\tau],\HHH\right).
\end{equation}  
Next, we prove 
$$|t|\leq c_{\vell}^{-1}\tau,\; |x|\geq |t|+R\,c_{\vell}\Longrightarrow u_{\vell}(t,x)=0$$
Indeed, the left-hand side of this implication implies
$$|y|-|s|\geq \frac{|x|+|t|}{|y|+|s|} c_{\vell}\,R\geq R,$$
and thus $u_{\vell}(t,x)=u(s,y)=0$. 

Since $u_{\vell}$ is compactly supported in the space variable and continuous on $\left[-\frac{\tau}{c_{\vell}},\frac{\tau}{c_{\vell}}\right]\times \RR^N$, we deduce 
\begin{equation}
\label{uell_L5L10}
 u_{\vell}\in L^{\cinq}\left([-c_{\vell}^{-1}\tau,c_{\vell}^{-1}\tau],L^{\dix} \right).
\end{equation} 
Finally it is easy to see, using that $u$ satisfies \eqref{NLW} in the distributional sense on $\EEE$, that $u_{\vell}$ satisfies \eqref{NLW} in the distributional sense on $\EEE_{\vell}$. By Remark \ref{R:solution}, $u_{\vell}$ is a solution of \eqref{NLW} on the interval $[-c_{\vell}^{-1}\tau,c_{\vell}^{-1}\tau]$. Notice for further use that by a simple change of variables,
\begin{equation}
\label{uniform_huit}
\|u_{\vell}\|_{L^{\huit}\left([-c_{\vell}^{-1}\tau,c_{\vell}^{-1}\tau]\times \RR^N\right)}\lesssim \|u\|_{L^{\huit}(\EEE)},
\end{equation} 
where the implicit constant depends only on $\eta_0$. 
\end{step}
\begin{step}
 We no longer assume $(u_0,u_1)\in \left( C_0^{\infty}(\RR^N) \right)^2$, and prove again that the Lorentz transform of $u$ is a solution of \eqref{NLW} on $\left[-c_{\vell}^{-1}\tau,c_{\vell}^{-1}\tau\right]$. Let $\left\{\left( u_{0}^k,u_1^k \right)\right\}_k$ be a sequence in $\left( C_0^{\infty}(\RR^N) \right)^2$ such that 
 $$\lim_{k\to\infty}\left\|\left( u_0^k,u_1^k)-(u_0,u_1) \right)\right\|_{\HHH}=0.$$
 Let $u$ be the solution of \eqref{NLW} on $\EEE=\{|x|> |t|-\tau\}$ with initial data $(u_0,u_1)$ at $t=0$. By Definition \ref{D:sol_cone}, this is the restriction to $\EEE$ of 
 the solution of 
\begin{equation*}
\partial_t^2u-\Delta u=|u|^{\frac{4}{N-2}}u\indic_{\EEE}, 
\end{equation*} 
with the same initial data (that we will also denote by $u$). We let $u^k$ the solution of the same equation with initial data $\left( u_0^k,u_1^k \right)$. By the above computations, the value of $u_{\vell}(t)$ (respectively $u^k_{\vell}(t)$) for $|t|\leq c_{\vell}^{-1}\tau$ depends only on the value of $u$ (respectively $u^k$) on $\EEE$.
By long-time perturbation theory, we obtain that for large $k$
\begin{equation}
 \label{S3}
 \left\|u^k\right\|_{L^{\frac{2(N+1)}{N-2}}(\RR\times \RR^N)}\leq 2\|u\|_{L^{\huit}(\RR\times \RR^N)}\lesssim \left\|\indic_{\EEE} u\right\|_{L^{\cinq}\big(\RR,L^{\dix}\big)},
\end{equation} 
and 
\begin{equation}
 \label{S4}
 \lim_{k\to\infty} \left\|u-u^k\right\|_{L^{\cinq}\left(\RR,L^{\dix}\right)}=0.
\end{equation} 
By the preceding step, $u^k_{\vell}$ is a solution of \eqref{NLW} on $\left[-c_{\vell}^{-1}\tau,c_{\vell}^{-1}\tau\right]\times \RR^N$. Since \eqref{S4} implies
$$ \lim_{k\to\infty}\left\|\left( F(u^k)-F(u) \right)\indic_{\EEE}\right\|_{L^1(\RR,L^2)}=0,$$
we deduce from Claim \ref{C:linear_Lorentz}
$$\sup_{t\in [-c_{\vell}^{-1}\tau,c_{\vell}^{-1}\tau]} \left\|\vec{u}_{\vell}(t)-\vec{u}^k_{\vell}(t)\right\|_{\HHH}\underset{k\to\infty}{\longrightarrow} 0.$$
Since by \eqref{S3},
$$ \left\|u^k_{\vell}\right\|_{L^{\huit}\left([-c_{\vell}^{-1}\tau,c_{\vell}^{-1}\tau]\times \RR^N\right)}$$
is uniformly bounded (see \eqref{uniform_huit} in the preceding step), we deduce by Remark \ref{R:sol_limit} that $u_{\vell}$ is a solution of \eqref{NLW} on $\left[-c_{\vell}^{-1}\tau,c_{\vell}^{-1}\tau\right]$.
\end{step}
\subsection{Perturbation}
In this subsection we prove Lemma \ref{L:Lorentz_perturb}. We use the notations of the previous subsection. Adapting the standard long-time perturbation theory to the exterior of wave cones, we obtain that there exists $\eps_0$ such that if 
$$\left\|(v_0,v_1)-(u_0,u_1)\right\|_{\HHH}\leq \eps_0,$$
then the solution $v$ of 
\begin{equation*}
 \left\{\begin{aligned}
         \partial_t^2v-\Delta v&=F(v)\indic_{\EEE}\\
         \vec{v}_{\restriction t=0}&=(v_0,v_1)
        \end{aligned}
\right.
\end{equation*} 
scatters and satisfies
$$ \sup_{t\in \RR}\left\|\vec{u}(t)-\vec{v}(t)\right\|_{\HHH}+\left\|u-v\right\|_{L^{\cinq}\left(\RR,L^{\dix}(\RR^N)\right)}\lesssim C\left\|(u_0,u_1)-(v_0,v_1)\right\|_{\HHH}.$$
As a consequence 
$$\left\|\left( |v|^{\frac{4}{N-2}}v-|u|^{\frac{4}{N-2}}u\right)\indic_{\EEE} \right\|_{L^1(\RR,L^2)}\lesssim \left\|\indic_{\EEE}u\right\|_{L^{\cinq}(\RR,L^{\dix})}^{\frac{4}{N-2}}\left\|(u_0,u_1)-(v_0,v_1)\right\|_{\HHH},$$
and the conclusion of the Lemma follows from Claim \ref{C:linear_Lorentz}.
\subsection{Energy and momentum}
\label{SS:en_mom}
It remains to prove the assertion on the energy and the momentum. This is classical (see e.g. \cite{KrNaSc13a}). We give a proof for the sake of completeness. We will assume
$$\vell=(\ell,0,\ldots,0)$$
to simplify notations. Let $\zeta\in \RR$ such that
$$\sinh \zeta=\frac{-\ell}{\sqrt{1-\ell^2}},\quad \cosh\zeta=\frac{1}{\sqrt{1-\ell^2}}.$$
As a consequence,
\begin{equation}
\label{S5}
u_{\vell}(t,x)=u(t\cosh\zeta+x_1\sinh\zeta,x_1\cosh \zeta+t\sinh\zeta,x_2,\ldots,x_N).
\end{equation} 
Let $\LLL_{\zeta}u(t,x)$ be the right-hand side of \eqref{S5}. Formally,
\begin{equation}
 \label{S6}
 \LLL_{\zeta+\xi}=\LLL_{\zeta}\circ\LLL_{\xi}, 
\end{equation} 
and, by direct computation,
\begin{align}
 \label{S7}
 \frac{d}{d\zeta}E\left(\bigvec{\LLL_{\zeta} u}(0)\right)_{\restriction \zeta=0}&=P_1\Big((u_0,u_1)\Big)\\
 \label{S8}
 \frac{d}{d\zeta}P_1\left(\bigvec{\LLL_{\zeta} u}(0)\right)_{\restriction \zeta=0}&=E\Big((u_0,u_1)\Big)\\
 \label{S8'}
\frac{d}{d\zeta}P_j\left(\bigvec{\LLL_{\zeta} u}(0)\right)_{\restriction \zeta=0}&=0,\quad j\in \llbracket 2,N\rrbracket.
 \end{align}
where 
$$P_ju(t)=\int \partial_tu(t,x)\partial_{x_j}u(t,x)\,dx.$$
Combining \eqref{S6}\ldots\eqref{S8'}, we deduce
\begin{align*}
 E\left( \bigvec{\LLL_{\zeta}u}(0) \right)&=\cosh \zeta E(u_0,u_1)+\sinh \zeta P_1(u_0,u_1)\\
 P_1\left(  \bigvec{\LLL_{\zeta}u}(0) \right)&=\cosh \zeta P_1(u_0,u_1)+\sinh \zeta E(u_0,u_1)\\
 P_j\left( \bigvec{\LLL_{\zeta}u}(0) \right)&=P_j\left( (u_0,u_1)\right),\quad j=\llbracket 2,N \rrbracket.
\end{align*}
This is exactly \eqref{Lorentz_energy} and \eqref{Lorentz_momentum}.
To make these formal computation rigorous, we smoothen the nonlinearity and the initial data. Let $\chi\in  C_0^{\infty}(\RR^N)$ such that $\chi(v)=1$ if $|v|\leq 1$ and $\chi(v)=0$ if $|v|\geq 2$. For $\eps>0$, let
$$F_{\eps}(v)=\left( 1-\chi\left( \frac{v}{\eps} \right) \right) )\chi\left( \eps v \right)|v|^{\frac{4}{N-2}}v,$$
and note that $F_{\eps}\in C_0^{\infty}\left( \RR^N \right)$. Let $\left( u_{0,\eps},u_{1,\eps} \right)\in \left( C_0^{\infty}(\RR^N) \right)^2$ such that 
$$\lim_{\eps \to 0} \left\|(u_{0,\eps},u_{1,\eps})-(u_0,u_1)\right\|_{\HHH}=0.$$
Let $u_{\eps}$ be the solution of 
\begin{equation}
\label{NLWeps}
\left\{
\begin{aligned}
 \partial_t^2u_{\eps}-\Delta u_{\eps}=|u_{\eps}|^{\frac{4}{N-2}}u_{\eps},\\
 \bigvec{u}_{\eps\restriction t=0}=(u_{0,\eps},u_{1,\eps})\in \HHH.
\end{aligned}
 \right.
 \end{equation} 
Note that $u_{\eps}$ is global, $C^{\infty}$, and that for all $t$
$$\supp \vec{u}_{\eps}(t)\subset\{|x|\leq |t|+R_{\eps}\},$$
where $R_{\eps}$ is such that $\supp \vec{u}_{\eps}(0)\subset\{|x|\leq R_{\eps}\}$. Let 
$$f_{\eps}(v)=\int_0^{v}F_{\eps}(w)\,dw.$$
The energy
$$ E_{\eps}(u_{\eps})=\frac{1}{2}\int |\nabla u_{\eps}|^2+\frac 12\int (\partial_t u_{\eps})^2-\int f_{\eps}(u_{\eps})$$
and the momentum
$$ P(u_{\eps})=\int \nabla u_{\eps}\partial_t u_{\eps}$$
are independent of time. The Lorentz transformation of $u_{\eps}$, $\LLL_{\zeta}u_{\eps}$ are solutions of \eqref{NLWeps} with $\left(C_0^{\infty}(\RR^N)\right)^2$ initial data. Explicit computations (which are rigorous in this context) prove that 
\begin{align*}
 E_{\eps}\left( \bigvec{\LLL_{\zeta}u_{\eps}}(0) \right)&=\cosh \zeta E_{\eps}(u_{0,{\eps}},u_{1,{\eps}})+\sinh \zeta P_1(u_{0,{\eps}},u_{1,\eps})\\
 P_1\left(  \bigvec{\LLL_{\zeta}u_{\eps}}(0) \right)&=\cosh \zeta P_1(u_{0,\eps},u_{1,\eps})+\sinh \zeta E_{\eps}(u_{0,\eps},u_{1,\eps})\\
 P_j\left( \bigvec{\LLL_{\zeta}u_{\eps}}(0) \right)&=P_j\left( (u_{0,\eps},u_{1,\eps} \right),\quad j=\llbracket 2,N \rrbracket.
\end{align*}
It remains to prove that if $|\vell|\leq \eta_0$, then 
\begin{equation}
 \label{S8''}
 \lim_{\eps\to 0}E_{\eps}\left( \bigvec{\LLL_{\zeta}u_{\eps}}(0) \right)=E\left(\bigvec{\LLL_{\zeta}u}(0) \right),\quad \lim_{\eps\to 0}P\left( \bigvec{\LLL_{\zeta}u_{\eps}}(0) \right)=P\left(\bigvec{\LLL_{\zeta}u}(0) \right).
\end{equation} 
We first prove
\begin{equation}
 \label{S9}
 \lim_{\eps\to 0} \left\|\bigvec{\LLL_{\zeta}(u_{\eps})}(0)-\bigvec{\LLL_{\zeta}(u)}(0)\right\|_{\HHH}=0.
\end{equation} 
For this, we start by proving
\begin{equation}
 \label{S10}
 \sup_{-\tau\leq t\leq \tau} \left\|\vec{u}(t)-\vec{u}_{\eps}(t)\right\|_{\HHH}+\left\|(u-u_{\eps})\indic_{\EEE}\right\|_{L^{\cinq}_tL^{\dix}_x}\underset{\eps\to 0}{\longrightarrow} 0.
\end{equation} 
Denote by $F(u)=|u|^{\frac{4}{N-2}}u$, $\psi_{\eps}(u)=\left( 1-\chi\left( \frac{u}{\eps} \right) \right)\chi(\eps u)$. Then
\begin{align*}
 \partial_t^2(u-u_{\eps})-\Delta(u-u_{\eps})&=\psi_{\eps}(u_{\eps})\left( F(u)\indic_{\EEE}-F(u_{\eps}) \right)+\left( 1-\psi_{\eps}(u_{\eps})\right)F(u)\indic_{\EEE} \\
 \bigvec{u-u_{\eps}}_{\restriction t=0}&= (u_0,u_1)-(u_{0,\eps},u_{1,\eps}).
\end{align*}
As a consequence, for all $t_0\geq 0$,
\begin{multline*}
 \left\|(u-u_{\eps})\indic_{\EEE}\right\|_{L^{\cinq}\left([0,t_0),L^{\dix}_x\right)}\\
 \lesssim \left\|(F(u)-F(u_{\eps}))\indic_{\EEE}\right\|_{L^1\left([0,t_0),L^{2}_x\right)}+\left\|(1-\psi_{\eps}(u_{\eps}))F(u)\indic_{\EEE}\right\|_{L^1\left([0,t_0),L^2_x\right)}
 \\+\left\|(u_0,u_1)-(u_{0,\eps},u_{1,\eps})\right\|_{\HHH}.
\end{multline*}
We write
\begin{equation*}
 \left( 1-\psi_{\eps}(u_{\eps}) \right)F(u)\indic_{\EEE}
 =\left(\indic_{|u-u_{\eps}|<\frac 12 |u|}+\indic_{|u-u_{\eps}|\geq \frac 12 |u|}\right) \left( 1-\psi_{\eps}(u_{\eps}) \right)F(u)\indic_{\EEE}
\end{equation*}
We have 
\begin{equation}
\label{a.e.1}
\lim_{\eps\to 0}
\indic_{|u-u_{\eps}|<\frac 12 |u|}\left( 1-\psi_{\eps}(u_{\eps}) \right)F(u)\indic_{\EEE}=0\quad \text{a.e.} 
\end{equation} 
Indeed, if $x$ is fixed, then 
$$
\indic_{|u-u_{\eps}|<\frac 12 |u|} \left( 1-\psi_{\eps}(u_{\eps}) \right)
\leq\begin{cases}
         0 & \text{ if }\eps\leq |u_{\eps}(x)|\leq \frac{1}{\eps}\\
         \chi\left( \frac{|u(x)|}{2\eps} \right)&\text{ if } |u_{\eps}(x)|\leq \eps\\
         \left( 1-\chi\left( \frac{3}{2}\eps |u(x)| \right) \right)&\text{ if }|u_{\eps}(x)|\geq \frac{1}{\eps},
        \end{cases}$$
where we have used that $\chi$ is decreasing. This obviously implies \eqref{a.e.1}. As a consequence of \eqref{a.e.1}, by the dominated convergence theorem
$$\lim_{\eps\to 0}\left\|\indic_{|u-u_{\eps}|<\frac 12 |u|} \left( 1-\psi_{\eps}(u_{\eps}) \right)F(u)\indic_{\EEE}\right\|_{L^{\cinq}_tL^{\dix}}=0.$$
On the other hand,
$$\indic_{|u-u_{\eps}|\geq \frac 12 |u|}\left( 1-\psi_{\eps}(u_{\eps}) \right)F(u)\lesssim |u|^{\frac 4N}|u-u_{\eps}|.$$
Using Strichartz estimates and the equation satisfied by $u-u_{\eps}$ we deduce that for all $t_0>0$,
\begin{multline*}
\left\|(u-u_{\eps})\indic_{\EEE}\right\|_{L^{\cinq}\left([0,t_0),L^{\dix}_x\right)}\\ \lesssim \left\| \indic_{\EEE}|u-u_{\eps}| |u|^{\frac{4}{N}}\right\|_{L^1\left( [0,t_0),L^2_x \right)}+\left\||u-u_{\eps}|^{1+\frac{4}{N}}\indic_{\EEE}\right\|_{L^1\left( [0,t_0),L^2_x \right)}+o(1),\quad \eps\to 0.
\end{multline*}
Since $u\indic_{\EEE}\in L^{\cinq}(\RR,L^{\dix})$,
we obtain, combining with the same argument for negative times,
$$\lim_{\eps\to\infty}\left\|(u-u_{\eps})\indic_{\EEE}\right\|_{L^{\cinq}\left(\RR,L^{\dix}\right)}=0.$$
Going back to the equation satisfied by $u-u_{\eps}$ and using  Strichartz estimates, we obtain \eqref{S10}. By Claim \ref{C:linear_Lorentz}, we deduce \eqref{S9}.

In view of \eqref{S9}, the following property will imply \eqref{S8''}:
$$\lim_{\eps\to 0}\int f_{\eps}\left(\LLL_{\zeta}(u_{\eps})(0,x) \right)\,dx=\frac{N-2}{2N}\int \left|\LLL_{\zeta}(u)(0,x) \right|^{\frac{2N}{N-2}}\,dx.$$
Denote $w(x)= \LLL_{\zeta}(u)(0,x)$, $w_{\eps}(x)=\LLL_{\zeta}(u_{\eps})(0,x)$,
and $f(u)=\frac{N-2}{2N}|u|^{\frac{2N}{N-2}}$. Write
\begin{multline*}
\int f_{\eps}\left(w_{\eps}(x)\right)\,dx-\int f\left(w(x)\right)\,dx\\
=\int f_{\eps}\left(w_{\eps}(x)\right)\,dx-\int f_{\eps}\left(w(x)\right)\,dx+\int f_{\eps}\left(w(x)\right)\,dx-\int f\left(w(x)\right)\,dx.
\end{multline*}
We have 
$0\leq f_{\eps}\left(w\right)\leq f\left(w\right)$
and
$\lim_{\eps\to 0} f_{\eps}\left(w(x)\right)=f\left(w(x)\right),$
a.e., which implies 
$$ \lim_{\eps\to 0} \int f_{\eps}\left(w(x)\right)\,dx=\int f\left(w(x)\right)\,dx,$$
by the dominated convergence theorem.

On the other hand,
\begin{multline*}
 \left|f_{\eps}\left(w_{\eps}(x)\right)-f_{\eps}\left(w(x)\right)\right|=\left|\int_{w(x)}^{w_{\eps}(x)}F_{\eps}(\sigma)\,d\sigma\right|\\
 \leq \left|\int_{w(x)}^{w_{\eps}(x)}F(\sigma)\,d\sigma\right|\leq \left| F\left( w_{\eps}(x) \right)+F\left( w(x) \right)\right|\,|w_{\eps}(x)-w(x)|,
\end{multline*}
where we have used that $F$ is monotonic. By H\"older inequality,
\begin{multline*}
 \int \left| f_{\eps}(w_{\eps}(x))-f_{\eps}(w(x))\right|\,dx\\ 
 \lesssim \left( \int |w_{\eps}(x)-w(x)|^{\frac{2N}{N-2}}\,dx \right)^{\frac{N-2}{2N}} \left( \|w_{\eps}\|^{\frac{N+2}{N-2}}_{L^{\frac{2N}{N-2}}}+\|w\|^{\frac{N+2}{N-2}}_{L^{\frac{2N}{N-2}}} \right)\underset{\eps\to 0}{\longrightarrow}0.
\end{multline*}
This concludes the proof.

\bibliographystyle{acm}
\bibliography{/home/duyckaerts/ownCloud2/Recherche/toto} 

\end{document}